\documentclass{article}
\usepackage{arxiv}
\pdfoutput=1

\usepackage{graphicx}
\usepackage{epstopdf}
\DeclareGraphicsExtensions{.eps,.png,.jpg,.pdf}
\usepackage{float}
\usepackage{subfig} 
\graphicspath{{./img/}{./}}
    
\usepackage{amsmath}
\usepackage{amsthm, thmtools} 
\usepackage{amssymb}
\usepackage{mathtools}
\usepackage{mathrsfs}
\allowdisplaybreaks[1] 

\usepackage[latin1]{inputenc} 
\usepackage{amsfonts}
\usepackage{textcomp}

\usepackage[table]{xcolor}
\usepackage{colortbl} 
\definecolor{LightCyan}{rgb}{0.70,1,1}

\usepackage{multirow}
\usepackage{tabulary}
\usepackage{rotating}

\usepackage{comment}

\usepackage{caption}

\usepackage{algorithm}
\usepackage{algpseudocodex}

\usepackage[shortlabels, inline]{enumitem}

\usepackage{setspace}

\usepackage{hyperref} 
\newtheorem{definition}{Definition}
\newtheorem{theorem}{Theorem}
\newtheorem{lemma}{Lemma}
\newtheorem{remark}{Remark}

\newtheorem{proposition}{Proposition}

\theoremstyle{thmstyletwo} 

\renewcommand{\min}[1]{\mathchoice
  {\underset{#1}{\operatorname{min}}\,}%
  {\operatorname{min}_{#1}}%
  {\operatorname{min}_{#1}}%
  {\operatorname{min}_{#1}}%
}
\renewcommand{\max}[1]{\mathchoice
  {\underset{#1}{\operatorname{max}}\,}%
  {\operatorname{max}_{#1}}%
  {\operatorname{max}_{#1}}%
  {\operatorname{max}_{#1}}%
}
\newcommand{\argmin}[1]{\mathchoice
  {\operatorname{arg}\, \min{#1}}%
  {\operatorname{arg}\, \min{#1}}%
  {\operatorname{arg}\, \min{#1}}%
  {\operatorname{arg}\, \min{#1}}%
}

\newcommand{\norm}[1]{ \left\Vert {#1} \right\Vert}
\newcommand{\euclideannorm}[1]{\norm{{#1}}_2}
\newcommand{\evaluatedat}[1]{ \Biggr\lvert_{#1}}

\newcommand{\fixed}[1]{\tilde{#1}}

\newcommand{\utilfun}{u_{\scriptscriptstyle \succsim}} 
\newcommand{\preffun}{{\pi}_{\scriptscriptstyle \succsim}} 
\newcommand{\surrpreffun}[1]{{\hat{\pi}}_{{\scriptscriptstyle \succsim}_{#1}}} 

\newcommand{\quotes}[1]{``{#1}''}

\newcommand{\blue}[1]{{\textcolor{blue}{#1}}}
\newcommand{\qedwhite}{\null\nobreak\hfill\ensuremath{\square}}
\newcommand{\important}[1]{\emph{#1}}

\newcommand{\figname}{Fig.}

\newcommand{\RBF}{RBF} 
\newcommand{\GP}{$\mathcal{GP}$} 
\newcommand{\IDW}{IDW} 
\newcommand{\DM}{DM} 
\newcommand{\LHD}{LHD} 
\newcommand{\LOOCV}{LOOCV} 
\newcommand{\PBO}{PBO} 

\newcommand{\GLISpmethod}{\texttt{GLISp} \cite{Bemporad2021}}
\newcommand{\CGLISpmethod}{\texttt{C-GLISp} \cite{zhu2021c}}
\newcommand{\GLISpCGLISpmethods}{\texttt{GLISp}/\texttt{C-GLISp} \cite{Bemporad2021,zhu2021c}}
\newcommand{\GLISprmethod}{\texttt{GLISp-r}}
\newcommand{\MSRSmethod}{\texttt{MSRS} \cite{regis2007stochastic}}
\newcommand{\SOSAmethod}{\texttt{SO-SA} \cite{wang2014general}}
\newcommand{\GutmannRBF}{\texttt{Gutmann-RBF} \cite{gutmann2001radial}}
\newcommand{\CORSmethod}{\texttt{CORS} \cite{regis2005constrained}}

\newcommand{\bemporadGOP}{\texttt{bemporad} \cite{Bemporad2020}}
\newcommand{\gramacyandleeGOP}{\texttt{gramacy-lee} \cite{gramacy2012cases}}
\newcommand{\ackleyGOP}{\texttt{ackley} \cite{jamil2013literature}}
\newcommand{\bukinsixGOP}{\texttt{bukin 6} \cite{jamil2013literature}}
\newcommand{\levythirteenGOP}{\texttt{levy 13} \cite{mishra2006some}}
\newcommand{\adjimanGOP}{\texttt{adjiman} \cite{jamil2013literature}}
\newcommand{\rosenbrockGOP}{\texttt{rosenbrock} \cite{jamil2013literature}}
\newcommand{\steptwoGOP}{\texttt{step 2} \cite{jamil2013literature}}
\newcommand{\salomonGOP}{\texttt{salomon} \cite{jamil2013literature}}

\newenvironment{algparams}{\begin{enumerate*}[(i), nosep]}{\end{enumerate*}}

\newcounter{algsubstate}
\renewcommand{\thealgsubstate}{\alph{algsubstate}}

\algnewcommand{\IfThenElse}[3]{
  \State \algorithmicif\ #1\ \algorithmicthen\ #2\ \algorithmicelse\ #3}
\algnewcommand{\IfThen}[2]{
  \State \algorithmicif\ #1\ \algorithmicthen\ #2}

\newcommand{\setalgorithmstretch}{\setstretch{1}} 
\newcommand{\setalgorithmfontsize}{\small}

\newcommand{\settablestretch}{\setstretch{2}}
\newcommand{\settablefontsize}{\footnotesize}

\hypersetup{pdfpagemode={UseOutlines},
  bookmarksopen=true,
  bookmarksopenlevel=0,
  hypertexnames=false,
  colorlinks=true,
  citecolor=blue,
  linkcolor=blue,
  urlcolor=black,
  pdfstartview={FitV},
  unicode,
  breaklinks=true,
}

\begin{document}

\title{\texttt{GLISp-r}: a preference-based optimization algorithm with convergence guarantees}

\author{
    Davide Previtali\\
    Department of Management, Information and Production Engineering\\
    University of Bergamo (Via G. Marconi 5, 24044, Dalmine (BG), Italy)\\
    \texttt{davide.previtali@unibg.it}\\
    \And
    Mirko Mazzoleni\\
    Department of Management, Information and Production Engineering\\
    University of Bergamo (Via G. Marconi 5, 24044, Dalmine (BG), Italy)\\
    \texttt{mirko.mazzoleni@unibg.it}\\
    \And
    Antonio Ferramosca\\
    Department of Management, Information and Production Engineering\\
    University of Bergamo (Via G. Marconi 5, 24044, Dalmine (BG), Italy)\\
    \texttt{antonio.ferramosca@unibg.it}\\
    \And
    Fabio Previdi\\
    Department of Management, Information and Production Engineering\\
    University of Bergamo (Via G. Marconi 5, 24044, Dalmine (BG), Italy)\\
    \texttt{fabio.previdi@unibg.it}\\
}

\date{}

\maketitle

\begin{abstract}
    Preference-based optimization algorithms are iterative procedures that seek the optimal calibration of a decision vector based only on comparisons between couples of different tunings. At each iteration, a human decision-maker expresses a preference between two calibrations (samples), highlighting which one, if any, is better than the other. The optimization procedure must use the observed preferences to find the tuning of the decision vector that is most preferred by the  decision-maker, while also minimizing the number of comparisons. In this work, we formulate the preference-based optimization problem from a utility theory perspective. Then, we propose \texttt{GLISp-r}, an extension of a recent preference-based optimization procedure called \texttt{GLISp}. The latter uses a Radial Basis Function surrogate to describe the tastes of the decision-maker. Iteratively, \texttt{GLISp} proposes new samples to compare with the best calibration available by trading off exploitation of the surrogate model and exploration of the decision space. In \texttt{GLISp-r}, we propose a different criterion to use when looking for new candidate samples that is inspired by \texttt{MSRS}, a popular procedure in the black-box optimization framework.  
    Compared to \texttt{GLISp}, \texttt{GLISp-r} is less likely to get stuck on local optima of the preference-based optimization problem. We motivate this claim theoretically, with a proof of global convergence, and empirically, by comparing the performances of \texttt{GLISp} and \texttt{GLISp-r} on several benchmark optimization problems.
\end{abstract}


\section{Introduction}
\label{sec:Introduction}
\important{Preference-Based Optimization (\PBO{})} algorithms seek the global solution of an optimization problem whose objective function is unknown and unmeasurable. Instead, the \quotes{goodness} of a 
\important{decision vector} is assessed by a human \important{Decision-Maker (\DM{})}, who compares different samples (i.e. tunings of the decision vector) and states which of them he/she prefers. 
In this work, we 
consider \important{queries} to the \DM{} that involve two options. The output of a query is the preference between the two calibrations of the decision vector expressed by the decision-maker (e.g. \quotes{this tuning is better than that one}) \cite{Bemporad2021}.

Preference-based optimization is closely related to industry practice. 
In the context of control systems, often the performances achieved by a regulator tuning are evaluated by a decision-maker, 
who expresses his/her judgment 
by observing the behavior of the system under control. Multiple experiments must be performed until the \DM{} is satisfied by the closed-loop performances. If a trial-and-error approach is adopted, then there is no guarantee that the final tuning is optimal in some sense. Moreover, the calibration process might be quite time-consuming since possibly many combinations of the controller parameters need to be tested. 
\PBO{} algorithms constitute a better alternative to the trial-and-error methodology 
due to the fact that they drive the experiments by exploiting the information carried by the preferences expressed by the \DM{}. The goal is to \important{seek the calibration of the decision vector that is most preferred by the decision-maker while also minimizing the number of queries}, thus performing fewer experiments. Successful applications of preference-based optimization procedures include \cite{zhu2020preference}, where the authors use algorithm \GLISpmethod{} to tune the controller for a continuous stirring tank reactor and for autonomous driving vehicles. The same algorithm has been employed in \cite{roveda2021pairwise} to calibrate the parameters of a velocity planner for robotic sealing tasks. 

Although preference-based optimization is a valuable tool for solving those calibration tasks that involve human decision-makers, its applicability is broader. In multi-objective optimization, \PBO{} can be employed to scalarize the multi-objective optimization problem into a single objective one by tuning the weights for the weighted sum method, or to choose which, among a set of Pareto optimal solutions, is the most suited for the \DM{} \cite{Bemporad2021, liu2002multiobjective}. \PBO{} can also be used to perform active preference learning. In general, preference learning methods estimate a predictive model that describes the tastes of a human decision-maker \cite{furnkranz2010preference}. Preference-based optimization methods also rely on a predictive model (called surrogate model), which is updated after each query to the \DM{}, 
but its prediction accuracy is not the main concern. Rather, the sole purpose of the surrogate model is to drive the search towards the most preferred tuning of the decision vector. 

The preference-based optimization problem can be formalized by applying notions of 
\important{utility theory} \cite{ok2011real}. Suppose that there exists a binary relation, called the \important{preference relation}, which describes the tastes of the decision-maker (i.e. the outputs of the queries). If the preference relation of the \DM{} 
exhibits certain properties, then it is possible to represent it with a continuous (latent) \important{utility function}\footnote{As a matter of fact, the goal of some preference learning methods is to estimate the utility function of the decision-maker \cite{furnkranz2010preference}.} \cite{debreu1971theory}, which assigns an abstract degree of \quotes{goodness} to all possible calibrations of the decision vector. 
The tuning that is most preferred by the decision-maker is the one with the highest utility and corresponds to the optimizer of the \PBO{} problem.

As previously mentioned, many preference-based optimization algorithms rely on a \important{surrogate model}, 
which is an approximation of the latent utility function built from the preference information at hand. In general, any procedure that uses a surrogate model when solving an optimization problem is said to be a \important{surrogate-based} (or \important{response surface}) \important{method}. These algorithms are mostly used for black-box optimization problems, where the objective function is unknown but measurable (see \cite{vu2017surrogate,jones2001taxonomy}). 
Nevertheless, surrogate-based methods can also be employed for \PBO{} problems, provided that the surrogate model is properly defined. In practice, preference-based response surface methods iteratively propose new samples to be compared to the available best candidate by trading off \important{exploitation} (or local search), i.e. selecting samples that are likely to offer an improvement based on the data at our disposal, and \important{exploration} (or global search), i.e. finding samples in regions of the decision space of which we have little to no knowledge. 
Typically, new candidate samples are sought by minimizing or maximizing an \important{acquisition function} that encapsulates these two aspects.

Most surrogate models either rely on Gaussian Processes (\GP{}s) \cite{rasmussen2006gaussian}, giving rise to preferential Bayesian optimization \cite{gonzalez2017preferential}, or Radial Basis Functions (\RBF{}s) \cite{gutmann2001radial}. For example, in \cite{chu2005preference} the authors propose a predictive model for the latent utility function that is based on \GP{}s. 
The latter is used as a surrogate model in \cite{brochu2007active} to carry out preference-based optimization. Alternative preferential Bayesian optimization algorithms are proposed in \cite{benavoli2021preferential} and in \cite{gonzalez2017preferential}. Recently, in \cite{Bemporad2021} the authors developed a preference-based optimization method, called \texttt{GLISp}, which is based on a \RBF{} approximation of the latent utility function.

In this work, \important{we propose an extension of the \GLISpmethod{} algorithm (that we will refer to as \GLISprmethod{}) which is more robust in finding the global solution of the preference-based optimization problem}. To do so, we:
\begin{enumerate}
      \item Address some limitations of the acquisition function used in \cite{Bemporad2021}, which can cause the procedure to miss the global solution and get stuck on local optima;
      \item Dynamically vary the exploration-exploitation trade-off weight, allowing \GLISprmethod{} to alternate between local and global search. This is commonly done in the black-box optimization framework, see for example \cite{gutmann2001radial,regis2007stochastic,wang2014general,regis2005constrained}, but it has not been tried for \GLISpmethod{};
      \item Provide a \important{proof of global convergence} for \GLISprmethod{}, furtherly motivating its robustness. Currently, no such proof is available for \GLISpmethod{} (or for any of the most popular preference-based response surface methods, 
            cf. also \cite{brochu2007active,gonzalez2017preferential,benavoli2021preferential}).
\end{enumerate}
Regarding this last point, as a minor contribution of this work, we formalize the preference-based optimization problem from a utility theory perspective, allowing us to analyze the existence of a solution and, ultimately, to prove the global convergence of \GLISprmethod{}.

This paper is organized as follows. In Section \ref{sec:Problem_formulation}, we introduce the preference-based optimization problem. Section \ref{sec:Handling_exploration_and_exploitation} addresses how to build the surrogate model (as in \GLISpmethod{}) and look for the next sample to evaluate, keeping in mind the exploration-exploitation dilemma. We also briefly cover the exploration function used in \cite{Bemporad2021}. The latter is thoroughly analyzed in Section \ref{sec:Next_candidate_sample_search}, where we propose a solution to the shortcomings that we have encountered in \GLISpmethod{}. Section \ref{sec:GLISp-r_and_convergence} describes algorithm \GLISprmethod{} and addresses its convergence. Then, Section \ref{sec:Empirical_results} compares the performances of \GLISprmethod{} with \GLISpmethod{} and \CGLISpmethod{}, a revisitation of the latter method by the same authors, on several benchmark optimization problems. Finally, Section \ref{sec:Discussion} gives some concluding remarks.
\section{Problem formulation}
\label{sec:Problem_formulation}
Let $\boldsymbol{x} \in \mathbb{R}^n, \boldsymbol{x} = \begin{bmatrix}x^{(1)} & \ldots & x^{(n)}\end{bmatrix}^{\top},$ be the $n$-dimensional decision vector and suppose that we are interested in finding its calibration that is most preferred by the decision-maker within a subset of $\mathbb{R}^{n}$, namely $\Omega \subset \mathbb{R}^{n}$. In particular, we define the \important{constraint set} $\Omega$ as:
\begin{equation}
    \label{eq:constraint_set_Omega}
    \Omega=\left\{\boldsymbol{x}: \boldsymbol{l}\leq\boldsymbol{x}\leq\boldsymbol{u}, \boldsymbol{g_{ineq}}(\boldsymbol{x})\leq\boldsymbol{0}_{q_{ineq}}, \boldsymbol{g_{eq}}(\boldsymbol{x})=\boldsymbol{0}_{q_{eq}}\right\}.
\end{equation}
In \eqref{eq:constraint_set_Omega}, $\boldsymbol{l}, \boldsymbol{u} \in \mathbb{R}^{n}, \boldsymbol{l} \leq \boldsymbol{u},$ are the lower and upper bounds on the decision vector while $\boldsymbol{g_{ineq}}: \mathbb{R}^{n} \to \mathbb{R}^{q_{ineq}}$ and $\boldsymbol{g_{eq}}: \mathbb{R}^{n} \to \mathbb{R}^{q_{eq}}$ are the constraints functions associated to the inequality and equality constraints respectively (which are $q_{ineq} \in \mathbb{N} \cup \left\{ 0 \right\}$ and $q_{eq} \in \mathbb{N} \cup \left\{ 0 \right\}$). Notation-wise, $\boldsymbol{0}_{q_{ineq}}$ represents the $q_{ineq}$-dimensional zero column vector (and similarly for $\boldsymbol{0}_{q_{eq}}$). We suppose that: (i) all of the constraints in \eqref{eq:constraint_set_Omega} are completely known and (ii) $\Omega$ includes, at least, the bound constraints $\boldsymbol{l}\leq\boldsymbol{x}\leq\boldsymbol{u}$ (the remaining equality and inequality constraints can be omitted, resulting in $q_{ineq} = q_{eq} = 0$).

In this work, we formalize the preference-based optimization problem from a \important{utility theory} \cite{ok2011real} perspective. We start by introducing preference relations, a particular kind of binary relations that play a key role in \PBO{}.
\begin{definition}[Binary relation \cite{ok2011real}]
    \label{def:binary_relation}
    Consider the constraint set $\Omega$ in \eqref{eq:constraint_set_Omega}; we define a generic binary relation $\mathcal{R}$ on $\Omega$ as a subset $\mathcal{R} \subseteq \Omega \times \Omega$.
\end{definition}
Notation-wise, given two samples $\boldsymbol{x}_{i}, \boldsymbol{x}_{j} \in \Omega$, we denote the ordered pairs for which the binary relation holds, $\left(\boldsymbol{x}_{i}, \boldsymbol{x}_{j}\right) \in \mathcal{R}$, as $\boldsymbol{x}_{i} \mathcal{R} \boldsymbol{x}_{j}$.
\begin{definition}[Preference relation \cite{ok2011real}]
    \label{def:preference_relation}
    A preference relation, $\succsim \subseteq \Omega \times \Omega$, is a binary relation that describes the tastes of a human  decision-maker.
\end{definition}

In the context of utility theory, $\boldsymbol{x}_{i} \succsim \boldsymbol{x}_{j}$ implies that the \DM{} with preference relation $\succsim$ on $\Omega$ deems the alternative $\boldsymbol{x}_{i}$ at least as good as $\boldsymbol{x}_{j}$. We say that the decision-maker is rational (in an economics sense) if his/her preference relation exhibits certain properties, as highlighted by the following Definition.
\begin{definition}[Rational decision-maker \cite{ok2011real}]
    \label{def:rational_decision_maker}
    Consider a decision-maker with preference relation $\succsim$ on $\Omega$. We say that the \DM{} is rational if $\succsim$ is a reflexive, transitive and complete binary relation on $\Omega$.
\end{definition}
Now, we briefly review each of the aforementioned properties and give some insights as to why they characterize the rationality of an individual (see \cite{debreu1971theory,ok2011real,feldman2006welfare} for more details):
\begin{itemize}
    \item \important{Reflexivity} of $\succsim$ on $\Omega$ implies that, for the \DM{}, any alternative is as good as itself, i.e. $\boldsymbol{x}_{i} \succsim \boldsymbol{x}_{i}$ for each $\boldsymbol{x}_{i} \in \Omega$;
    \item A decision-maker whose preference relation $\succsim$ on $\Omega$ is \important{transitive} is able to express his/her preferences coherently since if $\boldsymbol{x}_{i} \succsim \boldsymbol{x}_{j}$ and $\boldsymbol{x}_{j} \succsim \boldsymbol{x}_{k}$ hold, then $\boldsymbol{x}_{i} \succsim \boldsymbol{x}_{k}$, for any $\boldsymbol{x}_{i}, \boldsymbol{x}_{j}, \boldsymbol{x}_{k} \in \Omega$;
    \item \important{Completeness} of $\succsim$ on $\Omega$ implies that the \DM{} is able to express a preference between any two alternatives in $\Omega$, i.e. either $\boldsymbol{x}_{i} \succsim \boldsymbol{x}_{j}$ or $\boldsymbol{x}_{j} \succsim \boldsymbol{x}_{i}$ holds for each $\boldsymbol{x}_{i}, \boldsymbol{x}_{j} \in \Omega$.
\end{itemize}
The preference relation $\succsim$ on $\Omega$ of a rational decision-maker is usually \quotes{split} into two transitive binary relations \cite{ok2011real}:
\begin{itemize}
    \item The \important{strict preference relation} $\succ$ on $\Omega$, i.e. $\boldsymbol{x}_{i} \succ \boldsymbol{x}_{j}$ if and only if $\boldsymbol{x}_{i} \succsim \boldsymbol{x}_{j}$ but not $\boldsymbol{x}_{j} \succsim \boldsymbol{x}_{i}$ ($\boldsymbol{x}_{i}$ is \quotes{better than} $\boldsymbol{x}_{j}$ or, equivalently, $\boldsymbol{x}_j$ is \quotes{worse than} $\boldsymbol{x}_i$), and
    \item The \important{indifference relation} $\sim$ on $\Omega$, i.e. $\boldsymbol{x}_{i} \sim \boldsymbol{x}_{j}$ if and only if $\boldsymbol{x}_{i} \succsim \boldsymbol{x}_{j}$ and $\boldsymbol{x}_{j} \succsim \boldsymbol{x}_{i}$ ($\boldsymbol{x}_{i}$ is \quotes{as good as} $\boldsymbol{x}_{j}$).
\end{itemize}
One last relevant property for preference relations is continuity.
\begin{definition}[Continuous preference relation \cite{ok2011real}]
    \label{def:continuous_preference_relation}
    A preference relation $\succsim$ on $\Omega$ is continuous if the strict upper and lower $\succsim$-contour sets:
    \begin{align*}
        \mathcal{U}_{\scriptscriptstyle \succ}\left(\boldsymbol{x}\right) & = \left\{\fixed{\boldsymbol{x}}: \fixed{\boldsymbol{x}} \in \Omega, \fixed{\boldsymbol{x}} \succ \boldsymbol{x}\right\} \quad \text{and}           \\
        \mathcal{L}_{\scriptscriptstyle \succ}\left(\boldsymbol{x}\right) & = \left\{\fixed{\boldsymbol{x}}: \fixed{\boldsymbol{x}} \in \Omega, \boldsymbol{x} \succ \fixed{\boldsymbol{x}}\right\} \quad \text{respectively},
    \end{align*}
    are open subsets of $\Omega$ for each $\boldsymbol{x} \in \Omega$.
\end{definition}
Intuitively speaking, if $\succsim$ on $\Omega$ is continuous and $\boldsymbol{x}_i \succ \boldsymbol{x}_j$ holds, then an alternative $\boldsymbol{x}_k$ which is \quotes{very close} to $\boldsymbol{x}_j$ should also be deemed strictly worse than $\boldsymbol{x}_i$ by the decision-maker, i.e. $\boldsymbol{x}_i \succ \boldsymbol{x}_k$.

Having defined the preference relation $\succsim$ on $\Omega$ thoroughly, we can finally state the goal of preference-based optimization:
\begin{equation}
    \label{eq:goal_preference-based_optimization}
    find\ \boldsymbol{x}^{\boldsymbol{*}} \in \Omega\ such\ that\  \boldsymbol{x^{*}}\succsim\boldsymbol{x},\forall\boldsymbol{x}\in\Omega.
\end{equation}
Formally, $\boldsymbol{x^{*}}$ is called the \important{$\succsim$-maximum of $\Omega$} \cite{ok2011real}, i.e. the sample that is most preferred by the decision-maker with preference relation $\succsim$ on $\Omega$. Concerning the existence of $\boldsymbol{x}^*$, we can state the following Proposition, which can be seen as a generalization of the Extreme Value Theorem \cite{audet2017derivative} for preference relations.
\begin{proposition}[Existence of a $\succsim$-maximum of $\Omega$ \cite{ok2011real}]
    \label{prop:existence_of_preference_relation_maximum}
    A $\succsim$-maximum of $\Omega$ is guaranteed to exist if $\Omega$ is a compact subset of a metric space (in our case $\Omega \subset \mathbb{R}^{n}$) and $\succsim$ is a continuous preference relation on $\Omega$ of a rational decision-maker (see Definition \ref{def:rational_decision_maker} and Definition \ref{def:continuous_preference_relation}).
\end{proposition}
Proposition \ref{prop:existence_of_preference_relation_maximum} will be relevant when proving the convergence of the proposed algorithm,  in Section \ref{sec:GLISp-r_and_convergence}. Lastly, in order to write Problem \eqref{eq:goal_preference-based_optimization} as a  typical global optimization problem, we need to state one of the most important results in utility theory.
\begin{theorem}[Debreu's utility representation Theorem for $\mathbb{R}^{n}$ \cite{debreu1971theory}]
    \label{theo:debreu_utility_representation}
    Let $\Omega$ be any nonempty subset of $\mathbb{R}^{n}$ and $\succsim$ be a preference relation on $\Omega$ of a rational decision-maker (as in Definition \ref{def:rational_decision_maker}). If $\succsim$ on $\Omega$ is continuous, then it can be represented by a continuous utility function $\utilfun:\Omega \to \mathbb{R}$ such that, for any $\boldsymbol{x}_i, \boldsymbol{x}_j \in \Omega$:
    \begin{align*}
        \begin{split}
            \boldsymbol{x}_i \succsim \boldsymbol{x}_j \quad&\text{if and only if} \quad
            \utilfun\left(\boldsymbol{x}_i\right) \geq \utilfun\left(\boldsymbol{x}_j\right), \\
            \boldsymbol{x}_i \succ \boldsymbol{x}_j \quad &\text{if and only if} \quad
            \utilfun\left(\boldsymbol{x}_i\right) > \utilfun\left(\boldsymbol{x}_j\right), \\
            \boldsymbol{x}_i \sim \boldsymbol{x}_j \quad &\text{if and only if} \quad
            \utilfun\left(\boldsymbol{x}_i\right) = \utilfun\left(\boldsymbol{x}_j\right).
        \end{split}
    \end{align*}
\end{theorem}
Using Theorem \ref{theo:debreu_utility_representation}, we can build an optimization problem to find the $\succsim$-maximum of $\Omega$. In particular, we define the \important{scoring function}, $f: \mathbb{R}^{n} \to \mathbb{R}$, as $f\left(\boldsymbol{x}\right) = - \utilfun\left(\boldsymbol{x}\right)$ and re-write Problem \eqref{eq:goal_preference-based_optimization} as:
\begin{align}
    \label{eq:preference-based_optimization_problem}
    \boldsymbol{x}^{\boldsymbol{*}} & = \argmin{\boldsymbol{x}} f(\boldsymbol{x}) \\
    \text{s.t.}                     & \quad\boldsymbol{x}\in\Omega. \nonumber
\end{align}
\begin{remark}
    \label{rem:multiple_global_minimizers}
    Formally, there could be more than one $\succsim$-maximum of $\Omega$, i.e. Problem \eqref{eq:preference-based_optimization_problem} could admit multiple global solutions, as described by the set:
    \begin{equation*}
        \mathcal{X}^* = \left\{\boldsymbol{x}^{\boldsymbol{*}}_i: \boldsymbol{x}^{\boldsymbol{*}}_i \in \Omega \text{\ such that \ }  \nexists \boldsymbol{x}: \boldsymbol{x} \succ \boldsymbol{x}^{\boldsymbol{*}}_i \right\}.
    \end{equation*}
    In this work, without loss of generality, we assume that there exists only one global solution $\boldsymbol{x}^{\boldsymbol{*}}$ as in \eqref{eq:preference-based_optimization_problem}. In practice, as we will see in Section \ref{sec:GLISp-r_and_convergence}, any globally convergent preference-based optimization procedure generates a set of samples that is dense in $\Omega$ and thus, at least asymptotically, it actually finds all the global minimizers in $\mathcal{X}^*$. We do not make any assumptions on the local solutions of \eqref{eq:preference-based_optimization_problem}, which can be more than one.
\end{remark}
On a side note, we could view preference-based optimization as a particular instance of black-box optimization \cite{vu2017surrogate,jones2001taxonomy} where the cost function is not measurable in any way. Instead, information on $f(\boldsymbol{x})$ in \eqref{eq:preference-based_optimization_problem} comes in the form of preferences, as described in the next Section.

\subsection{Data available to preference-based optimization procedures}
\label{subsec:available_data_PBO}
In \GLISpmethod{}, instead of considering the preference relation $\succsim$ on $\Omega$ explicitly, the authors describe the outputs of the queries to the \DM{} using the \important{preference function} $\preffun:\mathbb{R}^{n} \times \mathbb{R}^{n} \to \{-1,0,1\}$, defined as:
\begin{equation}
    \label{eq:preference_function}
    \preffun \left(\boldsymbol{x}_{i},\boldsymbol{x}_{j} \right)=\begin{cases}
        -1 & \text{if }\boldsymbol{x}_{i} \succ \boldsymbol{x}_{j} \Leftrightarrow f\left(\boldsymbol{x}_{i}\right)<f\left(\boldsymbol{x}_{j}\right)  \\
        0  & \text{if } \boldsymbol{x}_{i} \sim \boldsymbol{x}_{j} \Leftrightarrow f\left(\boldsymbol{x}_{i}\right)=f\left(\boldsymbol{x}_{j}\right)  \\
        1  & \text{if } \boldsymbol{x}_{j} \succ \boldsymbol{x}_{i} \Leftrightarrow f\left(\boldsymbol{x}_{i}\right)>f\left(\boldsymbol{x}_{j}\right)
    \end{cases}.
\end{equation}
In light of the just reviewed utility theory literature, we can see that $\preffun(\boldsymbol{x}_{i},\boldsymbol{x}_{j})$ in \eqref{eq:preference_function} is obtained from the utility representation of the preference relation $\succsim$ on $\Omega$ (see Theorem \ref{theo:debreu_utility_representation}) and from the fact that $f(\boldsymbol{x}) = - \utilfun(\boldsymbol{x})$. In the case of rational decision-makers (Definition \ref{def:rational_decision_maker}), reflexivity and transitivity of the preference relation are highlighted by the following properties of the preference function:
\begin{itemize}
    \item $\preffun(\boldsymbol{x}_{i},\boldsymbol{x}_{i}) = 0$, for each $\boldsymbol{x}_{i} \in \mathbb{R}^{n}$,
    \item $\preffun(\boldsymbol{x}_{i},\boldsymbol{x}_{j}) = \preffun(\boldsymbol{x}_{j},\boldsymbol{x}_{k}) = b \Rightarrow \preffun(\boldsymbol{x}_{i},\boldsymbol{x}_{k}) = b$, for any $\boldsymbol{x}_{i}, \boldsymbol{x}_{j}, \boldsymbol{x}_{k} \in \mathbb{R}^{n}$.
\end{itemize}
In the context of \PBO{}, surrogate-based methods aim solve Problem \eqref{eq:preference-based_optimization_problem} starting from a set of $N \in \mathbb{N}$ \important{distinct samples} of the decision vector:
\begin{equation}
    \label{eq:sample_set_X}
    \mathcal{X}=\left\{ \boldsymbol{x}_{i}:i=1,\ldots,N,\boldsymbol{x}_{i}\in\Omega,
    \boldsymbol{x}_{i} \neq \boldsymbol{x}_{j}, \forall i \neq j\right\}
\end{equation}
and a set of $M \in \mathbb{N}$ \important{preferences} expressed by the  decision-maker:
\begin{equation}
    \label{eq:Set_of_preferences_B}
    \mathcal{B}=\left\{ b_{h}:h=1,\ldots,M,b_{h}\in\{-1,0,1\}\right\}.
\end{equation}
The term $b_h$ in \eqref{eq:Set_of_preferences_B} is the output of the $h$-th query, where the decision-maker was asked to compare two samples in $\mathcal{X}$, as highlighted by the following \important{mapping set}:
\begin{align}
    \label{eq:Mapping_set_for_preferences_S}
    \mathcal{S}=\Big\{\left(\ell(h),\kappa(h)\right): \  & h=1,\ldots,M,\ell(h),\kappa(h)\in\mathbb{N},  b_{h}=\preffun\left(\boldsymbol{x}_{\ell(h)},\boldsymbol{x}_{\kappa(h)}\right), \\
                                                         & b_{h}\in\mathcal{B},\boldsymbol{x}_{\ell(h)},\boldsymbol{x}_{\kappa(h)}\in\mathcal{X}\Big\}.\nonumber
\end{align}
In \eqref{eq:Mapping_set_for_preferences_S}, $\ell: \mathbb{N} \to \mathbb{N}$ and $\kappa: \mathbb{N} \to \mathbb{N}$ are two mapping functions that associate the indexes of the samples, contained inside $\mathcal{X}$, to their respective preferences in $\mathcal{B}$. The cardinalities of these sets are $\lvert\mathcal{X}\rvert=N$ and $\lvert\mathcal{B}\rvert=\lvert\mathcal{S}\rvert=M$. Also note that $1\leq M\leq\begin{pmatrix}N \\
        2
    \end{pmatrix}$.
\section{Handling exploration and exploitation}
\label{sec:Handling_exploration_and_exploitation}
In this Section, we review some key concepts that are common in most preference-based optimization algorithms. We also cover briefly how exploration and exploitation are handled by algorithm \GLISpmethod.

Preference-based response surface methods iteratively propose new samples to evaluate with the objective of solving Problem \eqref{eq:preference-based_optimization_problem} while also minimizing the number of queries. Suppose that, at iteration $k$, we have at our disposal the set of samples $\mathcal{X}$ in \eqref{eq:sample_set_X},  $\lvert\mathcal{X}\rvert = N$, and the sets $\mathcal{B}$ in \eqref{eq:Set_of_preferences_B} and $\mathcal{S}$ in \eqref{eq:Mapping_set_for_preferences_S}, $\lvert\mathcal{B}\rvert=\lvert\mathcal{S}\rvert=M$. We define the best sample found so far as:
\begin{equation*}
    \boldsymbol{x_{best}}\left(N\right) \in \Omega: \boldsymbol{x_{best}}\left(N\right) \in \mathcal{X}, \lvert \mathcal{X}\rvert = N, \text{ and } \boldsymbol{x_{best}}\left(N\right) \succsim \boldsymbol{x}_i, \forall \boldsymbol{x}_i \in \mathcal{X}.
\end{equation*}
The new candidate sample,
\begin{equation*}
    \boldsymbol{x}_{N+1} \in \Omega,
\end{equation*}
is obtained by solving an additional global optimization problem:
\begin{align}
    \label{eq:Next_sample_search_no_black-box_constraints}
    \boldsymbol{x}_{N+1} & = \argmin{\boldsymbol{x}} a_N(\boldsymbol{x}) \\
    \text{s.t.}          & \quad\boldsymbol{x}\in\Omega, \nonumber
\end{align}
where $a_N: \mathbb{R}^{n} \to \mathbb{R}$ is a properly defined \important{acquisition function} which trades off exploration and exploitation.

In practice, once $\boldsymbol{x}_{N+1}$ has been computed, we let the decision-maker express a preference between the best sample found so far and the new one, obtaining:
\begin{equation*}
    b_{M+1} = \preffun\left(\boldsymbol{x}_{N+1},\boldsymbol{x_{best}}\left(N\right)\right).
\end{equation*}
After that, $\boldsymbol{x}_{N+1}$ is added to the set $\mathcal{X}$ and, similarly, the sets $\mathcal{B}$ and $\mathcal{S}$ are also updated with the new preference $b_{M+1}$. The process is iterated until a certain condition is met. Typically, a \important{budget}, or rather a maximum number of samples $N_{max} \in \mathbb{N}$, is set and the optimization procedure is stopped once it is reached.

In our case, $a_N\left(\boldsymbol{x}\right)$ in \eqref{eq:Next_sample_search_no_black-box_constraints} is defined starting from a surrogate model $\hat{f}_N: \mathbb{R}^{n} \to \mathbb{R}$, which approximates the scoring function $f\left(\boldsymbol{x}\right)$ of Problem \eqref{eq:preference-based_optimization_problem}, and a function $z_N: \mathbb{R}^{n} \to \mathbb{R}$ that promotes the exploration of those regions of $\Omega$ where fewer samples have been evaluated. The acquisition function that we will propose in this work is a weighted sum of these two contributions (see Section \ref{subsec:Acquisition_function}). 
$\hat{f}_N\left(\boldsymbol{x}\right)$ and $z_N\left(\boldsymbol{x}\right)$ are defined as in \GLISpmethod{}, which we now review.

\subsection{Surrogate model}
\label{subsec:Surrogate_model}
Given $N$ samples $\boldsymbol{x}_i \in \mathcal{X}$ in \eqref{eq:sample_set_X}, we define the surrogate model $\hat{f}_N:\mathbb{R}^{n} \to \mathbb{R}$ as the \important{radial basis function expansion} \cite{fasshauer2007meshfree}:
\begin{equation}
    \begin{split}
        \label{eq:RBF_surrogate_model}
        \hat{f}_N\left(\boldsymbol{x}\right) &=\sum_{i=1}^{N}\beta^{(i)}\cdot\varphi\left(\epsilon\cdot\euclideannorm{\boldsymbol{x}-\boldsymbol{x}_{i}}\right) \\
        &=\sum_{i=1}^{N}\beta^{(i)}\cdot\phi_{i}\left(\boldsymbol{x}\right) \\
        &=\boldsymbol{\phi}\left(\boldsymbol{x}\right)^{\top}\cdot\boldsymbol{\beta},
    \end{split}
\end{equation}
where $\varphi:\mathbb{R}_{\geq 0}\to\mathbb{R}$ is a properly chosen \important{radial function}, $\epsilon\in\mathbb{R}_{>0}$ is the so-called \important{shape parameter} (which needs to be tuned) and $\phi_{i}:\mathbb{R}^{n}\to\mathbb{R}$ is the \important{radial basis function} originated from $\varphi\left(\cdot\right)$ and center $\boldsymbol{x}_{i}\in\mathcal{X}$, namely $\phi_{i}\left(\boldsymbol{x}\right)=\varphi\left(\epsilon\cdot \euclideannorm{\boldsymbol{x}-\boldsymbol{x}_{i}}\right)$. Moreover, $\boldsymbol{\phi}\left(\boldsymbol{x}\right) \in \mathbb{R}^{N}$, $\boldsymbol{\phi}\left(\boldsymbol{x}\right)=\begin{bmatrix}\phi_{1}\left(\boldsymbol{x}\right) & \ldots & \phi_{N}\left(\boldsymbol{x}\right)\end{bmatrix}^{\top}$, is the radial basis function vector and $\boldsymbol{\beta}=\begin{bmatrix}\beta^{(1)} & \ldots & \beta^{(N)}\end{bmatrix}^{\top}\in\mathbb{R}^{N}$ is a vector of weights that has to be estimated from the preferences in $\mathcal{B}$ \eqref{eq:Set_of_preferences_B} and $\mathcal{S}$ \eqref{eq:Mapping_set_for_preferences_S}. Given a distance $r = \euclideannorm{\boldsymbol{x} - \boldsymbol{x}_i}$, some commonly used radial functions are \cite{fornberg2015primer}:
\begin{itemize}
    \item Inverse quadratic: $\varphi\left(\epsilon \cdot r\right)=\frac{1}{1+\left(\epsilon\cdot r\right)^{2}}$;
    \item Multiquadratic: $\varphi\left(\epsilon \cdot r\right)=\sqrt{1+\left(\epsilon\cdot r\right)^{2}}$;
    \item Linear: $\varphi\left(\epsilon \cdot r\right)=\epsilon\cdot r$;
    \item Gaussian: $\varphi\left(\epsilon \cdot r\right)=e^{-\left(\epsilon\cdot r\right)^{2}}$;
    \item Thin plate spline: $\varphi\left(\epsilon \cdot r\right)=\left(\epsilon\cdot r\right)^{2}\cdot\log\left(\epsilon\cdot r\right)$;
    \item Inverse multiquadratic: $\varphi\left(\epsilon \cdot r\right)=\frac{1}{\sqrt{1+\left(\epsilon\cdot r\right)^{2}}}$.
\end{itemize}
One advantage of using \eqref{eq:RBF_surrogate_model} as the surrogate model is highlighted by the following Proposition.
\begin{proposition}
    \label{prop:Differentiability_of_surrogate_function}
    $\hat{f}_N\left(\boldsymbol{x}\right)$ in \eqref{eq:RBF_surrogate_model} is differentiable everywhere\footnote{Note that, whenever we say that a multivariable function, such as $\hat{f}_N\left(\boldsymbol{x}\right)$ in \eqref{eq:RBF_surrogate_model}, is \quotes{differentiable everywhere}, we imply that it is differentiable at each $\boldsymbol{x} \in \mathbb{R}^{n}$.} 
    if and only if the chosen radial basis function $\phi_i \left(\boldsymbol{x}\right) = \varphi\left(\epsilon\cdot\euclideannorm{\boldsymbol{x}-\boldsymbol{x}_{i}} \right)$ is differentiable everywhere.
\end{proposition}
The surrogate model in \eqref{eq:RBF_surrogate_model} can be used to define the \important{surrogate preference function}
$\surrpreffun{N}:{\mathbb{R}^{n}\times\mathbb{R}^{n}\to\{-1,0,1\}}$. Differently from $\preffun(\boldsymbol{x}_{i},\boldsymbol{x}_{j})$ in \eqref{eq:preference_function}, we consider a tolerance $\sigma \in \mathbb{R}_{>0}$ to avoid using strict inequalities and equalities and define $\surrpreffun{N}(\boldsymbol{x}_{i},\boldsymbol{x}_{j})$ as \cite{Bemporad2021}:
\begin{equation}
    \label{eq:RBF_surrogate_preference_function}
    \surrpreffun{N}(\boldsymbol{x}_{i},\boldsymbol{x}_{j})=\begin{cases}
        -1 & \text{if }\hat{f}_N\left(\boldsymbol{x}_{i}\right)-\hat{f}_N\left(\boldsymbol{x}_{j}\right)\leq-\sigma            \\
        0  & \text{if }\lvert\hat{f}_N\left(\boldsymbol{x}_{i}\right)-\hat{f}_N\left(\boldsymbol{x}_{j}\right)\rvert\leq\sigma \\
        1  & \text{if }\hat{f}_N\left(\boldsymbol{x}_{i}\right)-\hat{f}_N\left(\boldsymbol{x}_{j}\right)\geq\sigma
    \end{cases}.
\end{equation}
Suppose now that we have at our disposal the sets $\mathcal{X}$ in \eqref{eq:sample_set_X}, $\mathcal{B}$ in \eqref{eq:Set_of_preferences_B} and $\mathcal{S}$ in \eqref{eq:Mapping_set_for_preferences_S}. Then, we are interested in a surrogate model $\hat{f}_N\left(\boldsymbol{x}\right)$ in \eqref{eq:RBF_surrogate_model} that correctly describes the preferences expressed by the decision-maker, i.e. we would like the corresponding surrogate preference function $\surrpreffun{N}(\boldsymbol{x}_{i},\boldsymbol{x}_{j})$ in \eqref{eq:RBF_surrogate_preference_function} to be such that:
\begin{equation*}
    b_{h}=\surrpreffun{N}\left(\boldsymbol{x}_{\ell(h)},\boldsymbol{x}_{\kappa(h)}\right), \quad
    \forall b_{h}\in\mathcal{B}, \left(\ell(h),\kappa(h)\right) \in \mathcal{S}, h = 1, \ldots, M.
\end{equation*}
This, in turn, translates into some constraints on  $\hat{f}_N\left(\boldsymbol{x}\right)$, which can be used to estimate $\boldsymbol{\beta}$. To do so, the authors of \GLISpmethod{} define the following convex optimization problem:
\begin{align}
    \label{eq:Beta_computation_optimization_problem}
                & \quad \argmin{\boldsymbol{\varepsilon},\boldsymbol{\beta}}
    \frac{\lambda}{2} \cdot \boldsymbol{\beta}^{\top} \cdot \boldsymbol{\beta} + \boldsymbol{r}^\top \cdot \boldsymbol{\varepsilon}                                                                                                   \\
    \text{s.t.} & \quad\hat{f}_N\left(\boldsymbol{x}_{\ell(h)}\right)-\hat{f}_N\left(\boldsymbol{x}_{\kappa(h)}\right)\leq-\sigma+\varepsilon^{(h)}                                                    & \forall h:b_{h}=-1 \nonumber \\
                & \quad\lvert\hat{f}_N\left(\boldsymbol{x}_{\ell(h)}\right)-\hat{f}_N\left(\boldsymbol{x}_{\kappa(h)}\right)\rvert\leq\sigma+\varepsilon^{(h)}                                         & \forall h:b_{h}=0  \nonumber \\
                & \quad\hat{f}_N\left(\boldsymbol{x}_{\ell(h)}\right)-\hat{f}_N\left(\boldsymbol{x}_{\kappa(h)}\right)\geq\sigma-\varepsilon^{(h)}                                                     & \forall h:b_{h}=1 \nonumber  \\
                & \quad\varepsilon^{(h)} \geq 0                                                                                                                                              \nonumber                                \\
                & \quad h=1,\ldots,M, \nonumber
\end{align}
where:
\begin{itemize}
    \item $\boldsymbol{\varepsilon} = \begin{bmatrix} \varepsilon^{(1)} & \ldots & \varepsilon^{(M)} \end{bmatrix}^{\top} \in\mathbb{R}_{\geq 0}^{M}$ is a vector of \important{slack variables} (one for each preference) which takes into consideration that: (i) there might be some outliers in $\mathcal{B}$ and $\mathcal{S}$ if the decision-maker expresses the preferences in an inconsistent way, and (ii) the selected radial function and/or shape parameter for $\hat{f}_N\left(\boldsymbol{x}\right)$ in \eqref{eq:RBF_surrogate_model} do not allow a proper approximation of the scoring function $f\left(\boldsymbol{x}\right)$;
    \item $\boldsymbol{r} = \begin{bmatrix} r^{(1)} & \ldots & r^{(M)} \end{bmatrix}^{\top} \in \mathbb{R}^{M}_{>0}$ is a vector of weights that can be used to penalize more some slacks related to the most important preferences.
          In \GLISpmethod{}, the authors weigh more the preferences associated to the current best candidate and define $\boldsymbol{r}$ as follows:
          \begin{align*}
              r^{(h)} = 1, \quad  & \forall h: \left(\ell(h),\kappa(h)\right) \in \mathcal{S}, \boldsymbol{x}_{\ell(h)} \neq \boldsymbol{x_{best}}(N)      \text{ and } \boldsymbol{x}_{\kappa(h)} \neq \boldsymbol{x_{best}}(N), \\
              r^{(h)} = 10, \quad & \forall h: \left(\ell(h),\kappa(h)\right) \in \mathcal{S}, \boldsymbol{x}_{\ell(h)} = \boldsymbol{x_{best}}(N) \text{ or } \boldsymbol{x}_{\kappa(h)} = \boldsymbol{x_{best}}(N).
          \end{align*}
    \item $\lambda \in \mathbb{R}_{\geq 0}$ plays the role of a \important{regularization parameter}. It is easy to see that, for $\lambda=0$, Problem \eqref{eq:Beta_computation_optimization_problem} is a Linear Program (LP) while, for $\lambda>0$, it is a Quadratic Program (QP).
\end{itemize}
Problem \eqref{eq:Beta_computation_optimization_problem} ensures that, at least approximately, $\hat{f}_N\left(\boldsymbol{x}\right)$ in \eqref{eq:RBF_surrogate_model} is a suitable representation of the unknown preference relation $\succsim$ on $\Omega$ which produced the data described in Section \ref{subsec:available_data_PBO} (see Theorem \ref{theo:debreu_utility_representation}).
\subsection{Exploration function}
\label{subsec:Exploration_function}
Consider a sample $\boldsymbol{x}_i \in \mathcal{X}$ in \eqref{eq:sample_set_X}. Its corresponding \important{Inverse Distance Weighting (\IDW{}) function} $w_{i}:\mathbb{R}^{n}\setminus\left\{ \boldsymbol{x}_{i}\right\} \to\mathbb{R}_{>0}$ is defined as \cite{shepard1968two}:
\begin{equation}
    \label{eq:Inverse_distance_weighting_function}
    w_{i}\left(\boldsymbol{x}\right)=\frac{1}{\euclideannorm{\boldsymbol{x}-\boldsymbol{x}_{i}}^{2}}.
\end{equation}
\GLISpmethod{} uses the so-called \important{\IDW{} distance function} $z_N: \mathbb{R}^{n} \to (-1, 0]$,
\begin{equation}
    \label{eq:Inverse_distance_weighting_distance_function}
    z_N\left(\boldsymbol{x}\right)=\begin{cases}
        0                                                                                                 & \text{if }\boldsymbol{x}\in\mathcal{X} \\
        {-\frac{2}{\pi}}\cdot\arctan\left(\frac{1}{\sum_{i=1}^{N}w_{i}\left(\boldsymbol{x}\right)}\right) & \text{otherwise}
    \end{cases},
\end{equation}
to promote the exploration in those regions of $\mathbb{R}^{n}$ where fewer samples have been evaluated. 
It is possible to prove the following Proposition and Lemma.
\begin{proposition}
    \label{prop:Differentiability_of_IDW_distance_function}
    The IDW distance function $z_N\left(\boldsymbol{x}\right)$ in \eqref{eq:Inverse_distance_weighting_distance_function} is differentiable everywhere.
\end{proposition}
The proof of Proposition \ref{prop:Differentiability_of_IDW_distance_function} is given in \cite{Bemporad2020}. Here, we derive the gradient of $z_N\left(\boldsymbol{x}\right)$ in \eqref{eq:Inverse_distance_weighting_distance_function} (see Appendix \ref{sec:additional_proofs} for its derivation).
\begin{lemma}
    \label{lemma:Gradient_of_IDW_distance_function}
    The gradient of the \IDW{} distance function $z_N\left(\boldsymbol{x}\right)$ in \eqref{eq:Inverse_distance_weighting_distance_function} is:
    \begin{equation}
        \label{eq:Inverse_distance_weighting_distance_function_gradient}
        \nabla _{\boldsymbol{x}} z_N\left(\boldsymbol{x}\right)=\begin{cases}
            \boldsymbol{0}_{n}                                                                                                                                                                                  & \text{if }\boldsymbol{x}\in\mathcal{X} \\
            -\frac{4}{\pi}\cdot\frac{\sum_{i=1}^{N}\left(\boldsymbol{x}-\boldsymbol{x}_{i}\right)\cdot w_{i}\left(\boldsymbol{x}\right)^{2}}{1+\left[\sum_{i=1}^{N}w_{i}\left(\boldsymbol{x}\right)\right]^{2}} & \text{otherwise}
        \end{cases}.
    \end{equation}
\end{lemma}
The gradient $\nabla _{\boldsymbol{x}} z_N\left(\boldsymbol{x}\right)$ in \eqref{eq:Inverse_distance_weighting_distance_function_gradient} will be particularly relevant in the following Section.
\section{Next candidate sample search}
\label{sec:Next_candidate_sample_search}
As previously mentioned in Section \ref{sec:Handling_exploration_and_exploitation}, a key aspect of surrogate-based methods is the exploration-exploitation dilemma. Typically, new candidate samples are sought by minimizing an acquisition function $a_N\left(\boldsymbol{x}\right)$ that is a weighted sum between the surrogate model and the exploration function. 
In practice, $\hat{f}_N\left(\boldsymbol{x}\right)$ in \eqref{eq:RBF_surrogate_model} and $z_N\left(\boldsymbol{x}\right)$ in \eqref{eq:Inverse_distance_weighting_distance_function} often exhibit different ranges and need to be rescaled. In particular, \GLISpmethod{} adopts the following $a_N\left(\boldsymbol{x}\right)$:
\begin{equation}
    \label{eq:Acquisition_function_GLISp}
    a_N(\boldsymbol{x})=\frac{\hat{f}_N\left(\boldsymbol{x}\right)}{\Delta \hat{F}} + \delta \cdot z_N\left(\boldsymbol{x}\right),
\end{equation}
where $\delta \in \mathbb{R}_{\geq 0}$ is the exploration-exploitation trade-off weight. The division by
\begin{equation}
    \label{eq:Delta_f_acquisition_function_GLISp}
    \Delta \hat{F} = \max{\boldsymbol{x}_i \in \mathcal{X}} \hat{f}_N(\boldsymbol{x}_i) - \min{\boldsymbol{x}_i \in \mathcal{X}} \hat{f}_N(\boldsymbol{x}_i)
\end{equation}
aims to rescale the surrogate model to make it assume a range that is comparable to that of the \IDW{} distance function in \eqref{eq:Inverse_distance_weighting_distance_function}, which is $(-1, 0]$.

In this Section, we address some limitations of $a_N\left(\boldsymbol{x}\right)$ in \eqref{eq:Acquisition_function_GLISp}, which might prevent \GLISpmethod{} from reaching the global minimizer of Problem \eqref{eq:preference-based_optimization_problem}, and define an alternative acquisition function. Furthermore, we propose a strategy to iteratively vary the exploration-exploitation trade-off.

\subsection{Shortcomings of \texttt{GLISp}}
\label{subsec:Improving_exploration_capabilities_of_GLISp}
There are two shortcomings of $a_N\left(\boldsymbol{x}\right)$ in \eqref{eq:Acquisition_function_GLISp} that limit the exploratory capabilities of \GLISpmethod{}. First, the rescaling of $\hat{f}_N\left(\boldsymbol{x}\right)$ in \eqref{eq:Acquisition_function_GLISp}, which relies on $\Delta \hat{F}$ in \eqref{eq:Delta_f_acquisition_function_GLISp}, only takes into account the previously evaluated samples inside $\mathcal{X}$ in \eqref{eq:sample_set_X} and thus it can be ineffective in making $\frac{\hat{f}_N\left(\boldsymbol{x}\right)}{\Delta \hat{F}}$ and $z_N\left(\boldsymbol{x}\right)$ comparable over all $\Omega$ (see Problem \eqref{eq:Next_sample_search_no_black-box_constraints}). Second, the \IDW{} distance function in \eqref{eq:Inverse_distance_weighting_distance_function} exhibits two characteristics that can make its contribution negligible in $a_N\left(\boldsymbol{x}\right)$ in \eqref{eq:Acquisition_function_GLISp} and complicate the selection of $\delta$:
\begin{enumerate}
    \item Even though the range of $z_N\left(\boldsymbol{x}\right)$ is $(-1, 0]$, what we are really interested in when solving Problem \eqref{eq:Next_sample_search_no_black-box_constraints} and, ultimately, Problem \eqref{eq:preference-based_optimization_problem}, are the values that it assumes for $\boldsymbol{x} \in \Omega$ and not on its whole domain $\mathbb{R}^{n}$. In particular, there are some situations for which $\lvert z_N\left(\boldsymbol{x}\right) \rvert \ll 1, \forall \boldsymbol{x} \in \Omega$. 
          Consider, for example, the case $\mathcal{X} = \left\{\boldsymbol{x}_1\right\}$ ($N = 1$); then, $\forall \boldsymbol{x} \in \mathbb{R}^n \setminus \mathcal{X}$, the \IDW{} distance function simply becomes:
          \begin{equation*}
              z_1\left(\boldsymbol{x}\right) =-\frac{2}{\pi}\cdot\arctan\left(\euclideannorm{\boldsymbol{x}-\boldsymbol{x}_{1}}^{2}\right).
          \end{equation*}
          Suppose that Problem \eqref{eq:preference-based_optimization_problem} is only bound constrained, i.e. $\Omega = \left\{ \boldsymbol{x}: \boldsymbol{l}\leq\boldsymbol{x}\leq\boldsymbol{u}  \right\}$. Then,  $z_1\left(\boldsymbol{x}\right)$ assumes its lowest value at one (or more) of the vertices of the box defined by the constraints in $\Omega$. Define $\boldsymbol{v}_{\Omega} \in \Omega$ as one of such vertices; if $\euclideannorm{\boldsymbol{v}_{\Omega} - \boldsymbol{x}_{1}}$ is close to zero, then $\lvert z_N\left(\boldsymbol{x}\right) \rvert \ll 1, \forall \boldsymbol{x} \in \Omega$. 
          Thus, unless $\delta \gg 1$ in \eqref{eq:Acquisition_function_GLISp}, $\hat{f}_N\left(\boldsymbol{x}\right)$ in \eqref{eq:RBF_surrogate_model} and $z_N\left(\boldsymbol{x}\right)$ in \eqref{eq:Inverse_distance_weighting_distance_function} might not be comparable.
    \item The (absolute) values assumed by the \IDW{} distance function decrease as the number of samples increases. To clarify this, consider two sets of samples:
          \begin{align*}
              \mathcal{X}'  & = \left\{ \boldsymbol{x}_{1}, \ldots , \boldsymbol{x}_{N}  \right\}, \quad \lvert\mathcal{X}'\rvert = N, \\
              \mathcal{X}'' & = \mathcal{X}' \cup \left\{ \boldsymbol{x}_{N + 1} \right\}, \quad  \lvert\mathcal{X}''\rvert = N + 1.
          \end{align*}
          Given any point $\fixed{\boldsymbol{x}} \in \mathbb{R}^n \setminus \mathcal{X}''$, the \IDW{} distance functions obtained from the previously defined sets are:
          \begin{align*}
              z_{N}\left(\boldsymbol{\tilde{x}}\right)   & = {-\frac{2}{\pi}}\cdot\arctan\left(\frac{1}{\sum_{i=1}^{N}w_{i}\left(\boldsymbol{\tilde{x}}\right)}\right),   \\
              z_{N+1}\left(\boldsymbol{\tilde{x}}\right) & = {-\frac{2}{\pi}}\cdot\arctan\left(\frac{1}{\sum_{i=1}^{N+1}w_{i}\left(\boldsymbol{\tilde{x}}\right)}\right).
          \end{align*}
          Note that $w_i\left(\boldsymbol{\tilde{x}} \right) > 0, \forall \boldsymbol{\tilde{x}}\in \mathbb{R}^n \setminus \mathcal{X}''$ and $i = 1, \ldots, N + 1$ (see \eqref{eq:Inverse_distance_weighting_function}). Hence:
          \begin{equation*}
              \lvert z_{N}\left(\boldsymbol{\tilde{x}}\right)\rvert > \lvert z_{N+1}\left(\boldsymbol{\tilde{x}}\right)\rvert > 0,
          \end{equation*}
          proving the above point. In practice, unless $\delta$ in \eqref{eq:Acquisition_function_GLISp} is progressively increased as the iterations go on, \GLISpmethod{} will explore the constraint set $\Omega$ of Problem \eqref{eq:preference-based_optimization_problem} less as the number of samples increases, regardless of whether a region that contains the global minimizer $\boldsymbol{x^*}$ has been located.
\end{enumerate}
A visualization of these two characteristics of $z_N\left(\boldsymbol{x}\right)$ in \eqref{eq:Inverse_distance_weighting_distance_function} is presented in \figname{} \ref{fig:Limitations_of_the_IDW_distance_function}.

\begin{figure*}[!htb]
    \centering
    \subfloat{
        \centering
        \includegraphics[width=.5\textwidth]{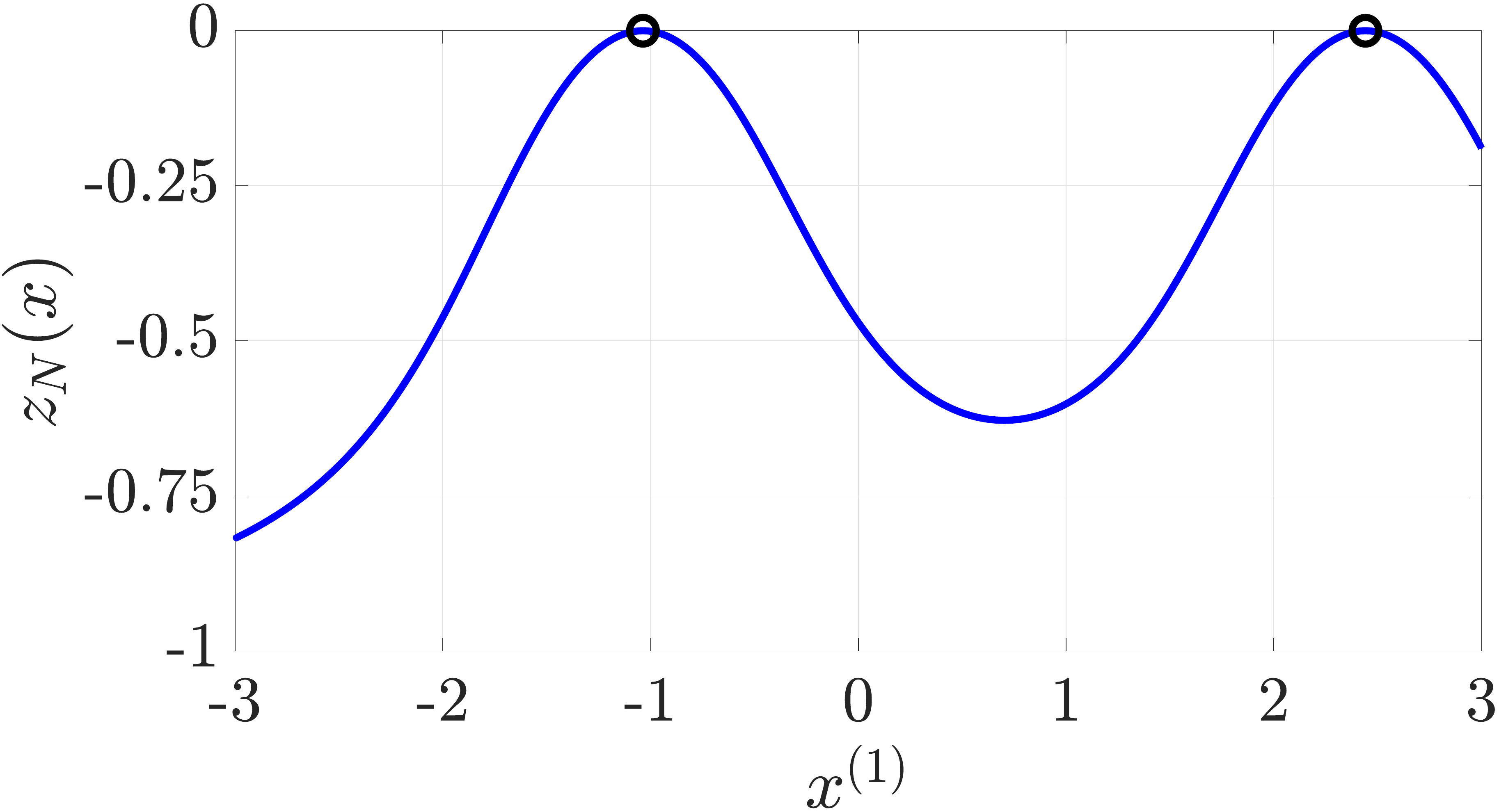}
    }
    \subfloat{
        \centering
        \includegraphics[width=.5\textwidth]{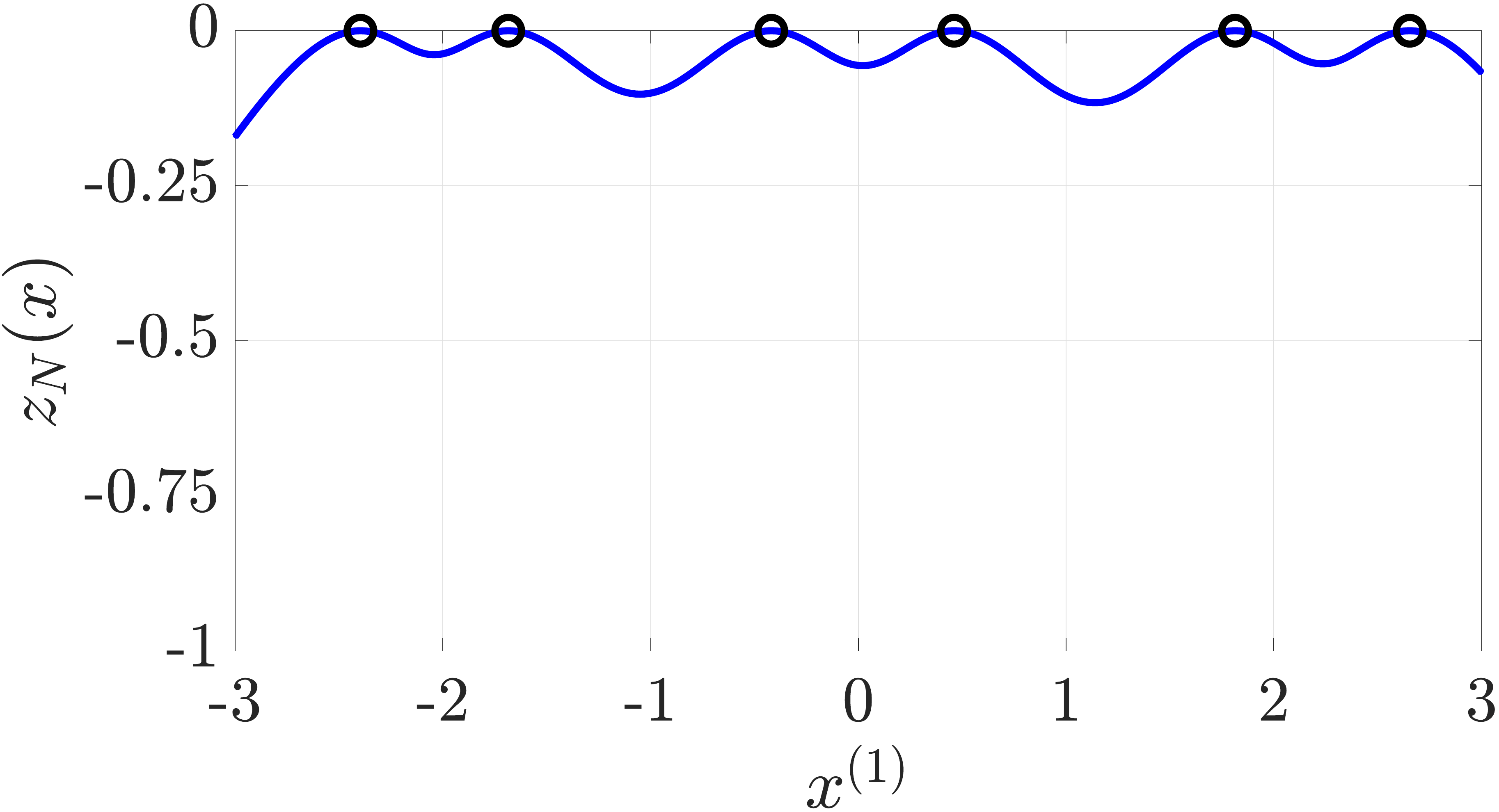}
    }

    \caption{
        \label{fig:Limitations_of_the_IDW_distance_function}
        Examples of the \IDW{} distance function $z_N\left(x\right)$ in \eqref{eq:Inverse_distance_weighting_distance_function} for  different numbers of points ($N = 2$ on the left, $N = 6$ on the right) and $-3 = l \leq x \leq u = 3$. Notice how $z_N\left(x\right)$ does not cover its whole range $(-1, 0]$, at least inside the bound constraints, and its absolute values decrease as the number of samples increases.
    }
\end{figure*}

In this work, we overcome the limitations of $a_N\left(\boldsymbol{x}\right)$ in \eqref{eq:Acquisition_function_GLISp} by defining an acquisition function that is similar to the one used by the Metric Stochastic Response Surface (\MSRSmethod{}) algorithm, a popular black-box optimization procedure. In \MSRSmethod{}, the surrogate $\hat{f}_N\left(\boldsymbol{x}\right)$ and the exploration function $z_N\left(\boldsymbol{x}\right)$ 
are made comparable by randomly sampling $\Omega$ and rescaling the two using min-max normalization. Then, Problem \eqref{eq:Next_sample_search_no_black-box_constraints} is not solved explicitly but by choosing the generated random sample that achieves the lowest value for the acquisition function. Here, we propose to rescale $z_N\left(\boldsymbol{x}\right)$ in \eqref{eq:Inverse_distance_weighting_distance_function} and $\hat{f}_N\left(\boldsymbol{x}\right)$ in \eqref{eq:RBF_surrogate_model} using some insights on the stationary points of the \IDW{} distance function and solve Problem \eqref{eq:Next_sample_search_no_black-box_constraints} explicitly, using a proper global optimization solver.  

\subsection{Novel rescaling strategy}
In this Section, we derive the approximate locations of the stationary points of the \IDW{} distance function in \eqref{eq:Inverse_distance_weighting_distance_function} and use them to define an augmented set of samples $\mathcal{X}_{aug} \supset \mathcal{X}$ that is suited for the min-max normalization of both $z_N\left(\boldsymbol{x}\right)$ and $\hat{f}_N\left(\boldsymbol{x}\right)$.

\subsubsection{Stationary points of the \IDW{} distance function}
\label{subsubsec:Stationary_points_for_IDW_distance_function}
The locations of the global maximizers of $z_N\left(\boldsymbol{x}\right)$ can be deduced immediately from \eqref{eq:Inverse_distance_weighting_distance_function} and Lemma \ref{lemma:Gradient_of_IDW_distance_function}, as stated by the following Proposition.
\begin{proposition}
    \label{prop:global_maximizers_of_IDW_distance}
    Each $\boldsymbol{x}_i \in \mathcal{X}$ in \eqref{eq:sample_set_X} is a global maximizer of $z_N\left(\boldsymbol{x}\right)$ in \eqref{eq:Inverse_distance_weighting_distance_function}.
\end{proposition}
\begin{proof}
    Recall that:
    \begin{enumerate}[(i)]
        \item $\nabla _{\boldsymbol{x}} z_N\left(\boldsymbol{x}_i\right) = \boldsymbol{0}_{n}, \forall \boldsymbol{x}_i \in \mathcal{X}$, see \eqref{eq:Inverse_distance_weighting_distance_function_gradient}; \label{proof:item:global_maximizers_of_IDW_distance_1}
        \item $z_N\left(\boldsymbol{x}\right) < 0, \forall \boldsymbol{x} \in \mathbb{R}^n \setminus \mathcal{X}$, see \eqref{eq:Inverse_distance_weighting_distance_function}; \label{proof:item:global_maximizers_of_IDW_distance_2}
        \item $z_N\left(\boldsymbol{x}_i\right) = 0, \forall \boldsymbol{x}_i \in \mathcal{X}$, see \eqref{eq:Inverse_distance_weighting_distance_function}. \label{proof:item:global_maximizers_of_IDW_distance_3}
    \end{enumerate}
    From Item \ref{proof:item:global_maximizers_of_IDW_distance_1} we deduce that each $\boldsymbol{x}_i \in \mathcal{X}$ is a stationary point of $z_N\left(\boldsymbol{x}\right)$. Item \ref{proof:item:global_maximizers_of_IDW_distance_2}, in conjunction with Item \ref{proof:item:global_maximizers_of_IDW_distance_3}, implies that such samples are local maximizers of the \IDW{} distance function in \eqref{eq:Inverse_distance_weighting_distance_function} since there exists a neighborhood of $\boldsymbol{x}_i \in \mathcal{X}$, denoted as $\mathcal{N}\left(\boldsymbol{x}_i\right)$, such that $z_N\left(\boldsymbol{x}\right) \leq z_N\left(\boldsymbol{x}_i \right), \forall \boldsymbol{x} \in \mathcal{N}\left(\boldsymbol{x}_i\right)$. Moreover, due to Item \ref{proof:item:global_maximizers_of_IDW_distance_2}, $z_N\left(\boldsymbol{x}\right) \leq z_N\left(\boldsymbol{x}_i \right), \forall \boldsymbol{x} \in \mathbb{R}^n$ and not just in a neighborhood of $\boldsymbol{x}_i \in \mathcal{X}$. Hence, each $\boldsymbol{x}_i \in \mathcal{X}$ is a global maximizer of $z_N\left(\boldsymbol{x}\right)$ in \eqref{eq:Inverse_distance_weighting_distance_function}.
\end{proof}
Reaching similar conclusions for the minimizers of $z_N\left( \boldsymbol{x} \right)$ is much harder; however, we can consider some simplified situations. Note that we are not necessarily interested in finding the minimizers of the \IDW{} distance function in \eqref{eq:Inverse_distance_weighting_distance_function} with high accuracy, but rather to gain some insights on where they are likely to be located so that we can rescale both $z_N\left( \boldsymbol{x} \right)$ in \eqref{eq:Inverse_distance_weighting_distance_function} and $\hat{f}_N\left(\boldsymbol{x}\right)$ in \eqref{eq:RBF_surrogate_model} sufficiently enough to make them comparable. Moreover, their approximate locations can be used to solve the following global  optimization problem (\important{pure exploration}):
\begin{align}
    \label{eq:Minimization_of_IDW_distance_function_optimization_problem}
    \boldsymbol{x}_{N+1} & = \argmin{\boldsymbol{x}} z_N(\boldsymbol{x}) \\
    \text{s.t.}          & \quad\boldsymbol{x}\in\Omega \nonumber
\end{align}
by using a multi-start derivative-based optimization method with warm-start \cite{nocedal1999numerical,marti2018multi} (recall that $z_N\left(\boldsymbol{x}\right)$ is differentiable everywhere, see Proposition \ref{prop:Differentiability_of_IDW_distance_function}). Problem \eqref{eq:Minimization_of_IDW_distance_function_optimization_problem} is quite relevant for the global convergence of the algorithm that we will propose in Section \ref{sec:GLISp-r_and_convergence}.
\begin{remark}
    In the following Paragraphs, we  analyze where the local minimizers of $z_N\left( \boldsymbol{x} \right)$ in \eqref{eq:Inverse_distance_weighting_distance_function} and the solution(s) of the simplified problem:
    \begin{align}
        \label{eq:Minimization_of_IDW_distance_function_optimization_problem_simplified}
        \boldsymbol{x}_{N+1} & = \argmin{\boldsymbol{x}} z_N(\boldsymbol{x})                        \\
        \text{s.t.}          & \quad\boldsymbol{l} \leq \boldsymbol{x} \leq \boldsymbol{u}\nonumber
    \end{align}
    are located in some specific cases. Note that $\left\{\boldsymbol{x}: \boldsymbol{l} \leq \boldsymbol{x} \leq \boldsymbol{u}\right\} \supseteq \Omega$ (see \eqref{eq:constraint_set_Omega}) and thus the global minimum 
    of Problem \eqref{eq:Minimization_of_IDW_distance_function_optimization_problem_simplified} is lower than or at most equal to the global minimum of Problem \eqref{eq:Minimization_of_IDW_distance_function_optimization_problem}. Therefore, the minimizers of Problem \eqref{eq:Minimization_of_IDW_distance_function_optimization_problem_simplified} are better suited to perform min-max rescaling of $z_N\left( \boldsymbol{x} \right)$ than those of Problem \eqref{eq:Minimization_of_IDW_distance_function_optimization_problem}.
\end{remark}

\paragraph{Case $\mathcal{X} = \left\{\boldsymbol{x}_1\right\}$ ($N = 1$)} The \IDW{} distance function and its gradient $\forall \boldsymbol{x} \in \mathbb{R}^n \setminus \mathcal{X}$ are:
\begin{align*}
    z_N\left(\boldsymbol{x}\right)                         & =-\frac{2}{\pi}\cdot\arctan\left(\euclideannorm{\boldsymbol{x}-\boldsymbol{x}_{1}}^{2}\right),                                                              \\
    \nabla_{\boldsymbol{x}} z_N\left(\boldsymbol{x}\right) & =-\frac{4}{\pi}\cdot\left(\boldsymbol{x}-\boldsymbol{x}_{1}\right)\cdot\frac{w_{1}\left(\boldsymbol{x}\right)^{2}}{1+w_{1}\left(\boldsymbol{x}\right)^{2}}.
\end{align*}
Clearly, $\forall \boldsymbol{x} \in \mathbb{R}^n \setminus \mathcal{X}$, the gradient is never zero since $w_{1}\left(\boldsymbol{x}\right) > 0$. Therefore, the only stationary point is the global maximizer $\boldsymbol{x}_1$ (see Proposition \ref{prop:global_maximizers_of_IDW_distance}). However, if we were to consider Problem \eqref{eq:Minimization_of_IDW_distance_function_optimization_problem_simplified}, 
then its solution would be located at one of the vertices of the box defined by the bound constraints $\boldsymbol{l} \leq \boldsymbol{x} \leq \boldsymbol{u}$.

\paragraph{Case $\mathcal{X} = \left\{\boldsymbol{x}_1, \boldsymbol{x}_2\right\}$ ($N = 2$)} The gradient of the \IDW{} distance function $\forall \boldsymbol{x} \in \mathbb{R}^n \setminus \mathcal{X}$ is:
\begin{equation*}
    \nabla_{\boldsymbol{x}}z_N\left(\boldsymbol{x}\right)=-\frac{4}{\pi}\cdot\frac{\left(\boldsymbol{x}-\boldsymbol{x}_{1}\right)\cdot w_{1}\left(\boldsymbol{x}\right)^{2}+\left(\boldsymbol{x}-\boldsymbol{x}_{2}\right)\cdot w_{2}\left(\boldsymbol{x}\right)^{2}}{1+\left[w_{1}\left(\boldsymbol{x}\right)+w_{2}\left(\boldsymbol{x}\right)\right]^{2}}.
\end{equation*}
Let us consider the midpoint $\boldsymbol{x_{\mu}}=\frac{\boldsymbol{x}_{1}+\boldsymbol{x}_{2}}{2}$, that is such that $\euclideannorm{\boldsymbol{x_{\mu}}-\boldsymbol{x}_{1}}=\euclideannorm{\boldsymbol{x_{\mu}}-\boldsymbol{x}_{2}}$ and for which $w_{1}\left(\boldsymbol{x_{\mu}}\right)=w_{2}\left(\boldsymbol{x_{\mu}}\right)$. If we substitute it in the previous expression, we obtain:
\begin{equation*}
    \nabla_{\boldsymbol{x}}z_N\left(\boldsymbol{x_{\mu}}\right)=\boldsymbol{0}_{n},
\end{equation*}
which means that $\boldsymbol{x_{\mu}}$ is a stationary point for $z_N\left(\boldsymbol{x}\right)$ in \eqref{eq:Inverse_distance_weighting_distance_function}. It is easy to see by visual inspection that such point is actually a local minimizer for the \IDW{} distance function (see for example \figname{} \ref{fig:IDW_distance_function_and_gradient}). However, note that $\boldsymbol{x_{\mu}}$ is not necessarily the global solution of Problem \eqref{eq:Minimization_of_IDW_distance_function_optimization_problem_simplified}, it might just be a local one.

\begin{figure}[!htb]
    \centering
    \subfloat[$\mathcal{X} = \left\{x_1 \right\}$.]{
        \centering
        \includegraphics[width=.5\textwidth]{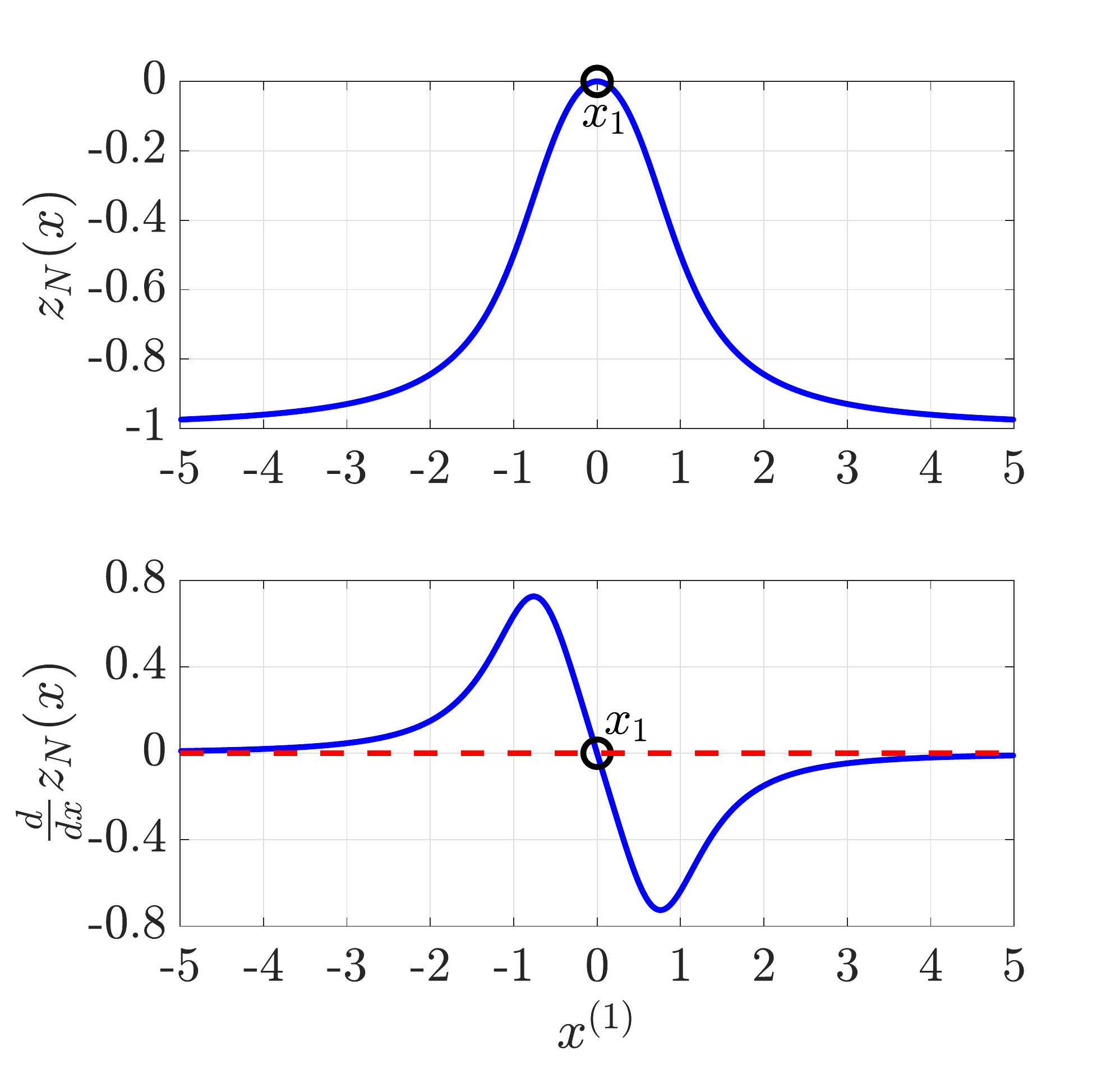}
    }
    \subfloat[$\mathcal{X} = \left\{x_1, x_2\right\}$. \label{fig:IDW_distance_function_and_gradient_2_points}]{
        \centering
        \includegraphics[width=.5\textwidth]{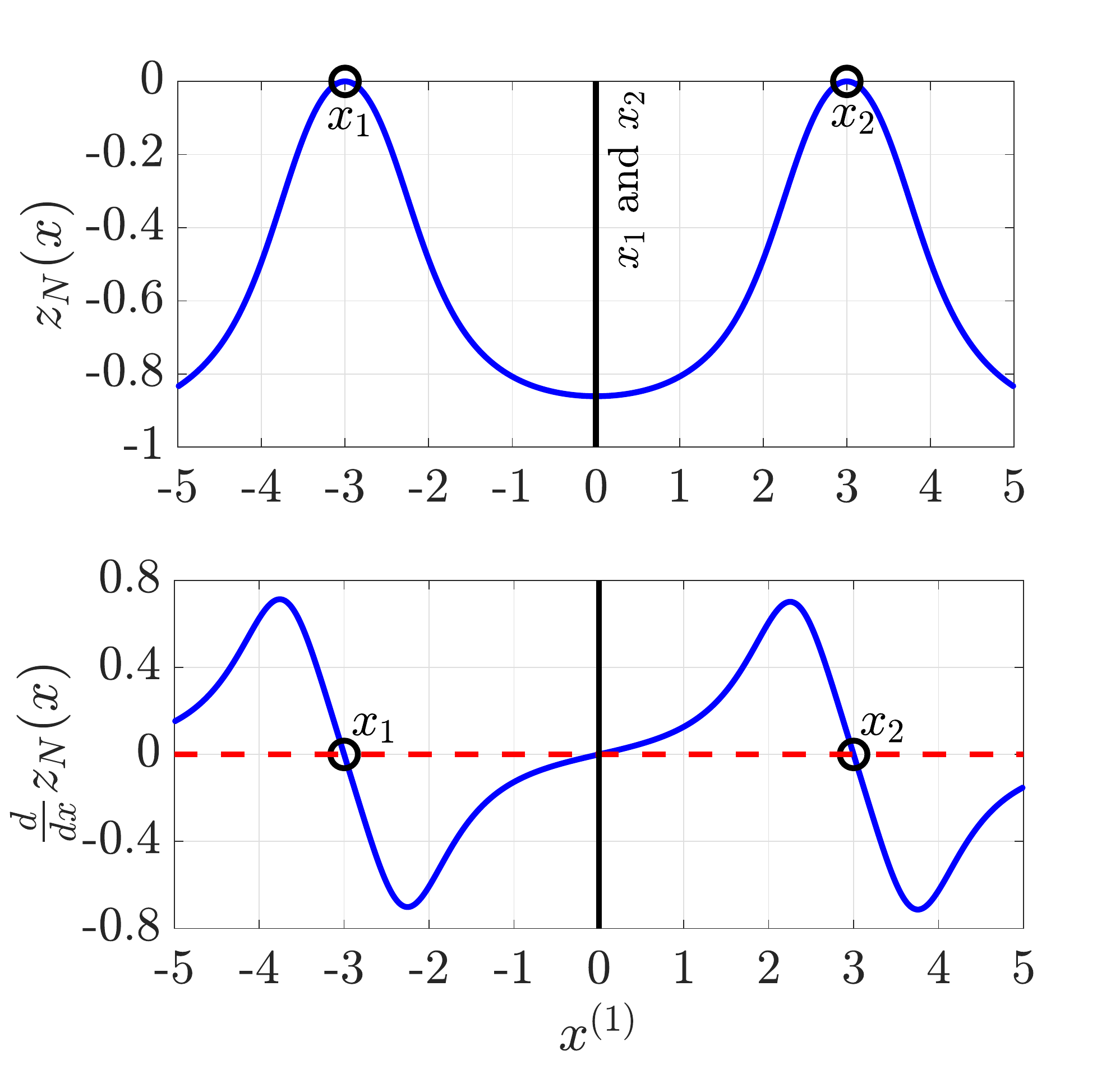}
    }

    \subfloat[$\mathcal{X} = \mathcal{X}^{(1)} \cup \mathcal{X}^{(2)}$. \label{fig:IDW_distance_function_and_gradient_2_clusters}]{
        \centering
        \includegraphics[width=.5\textwidth]{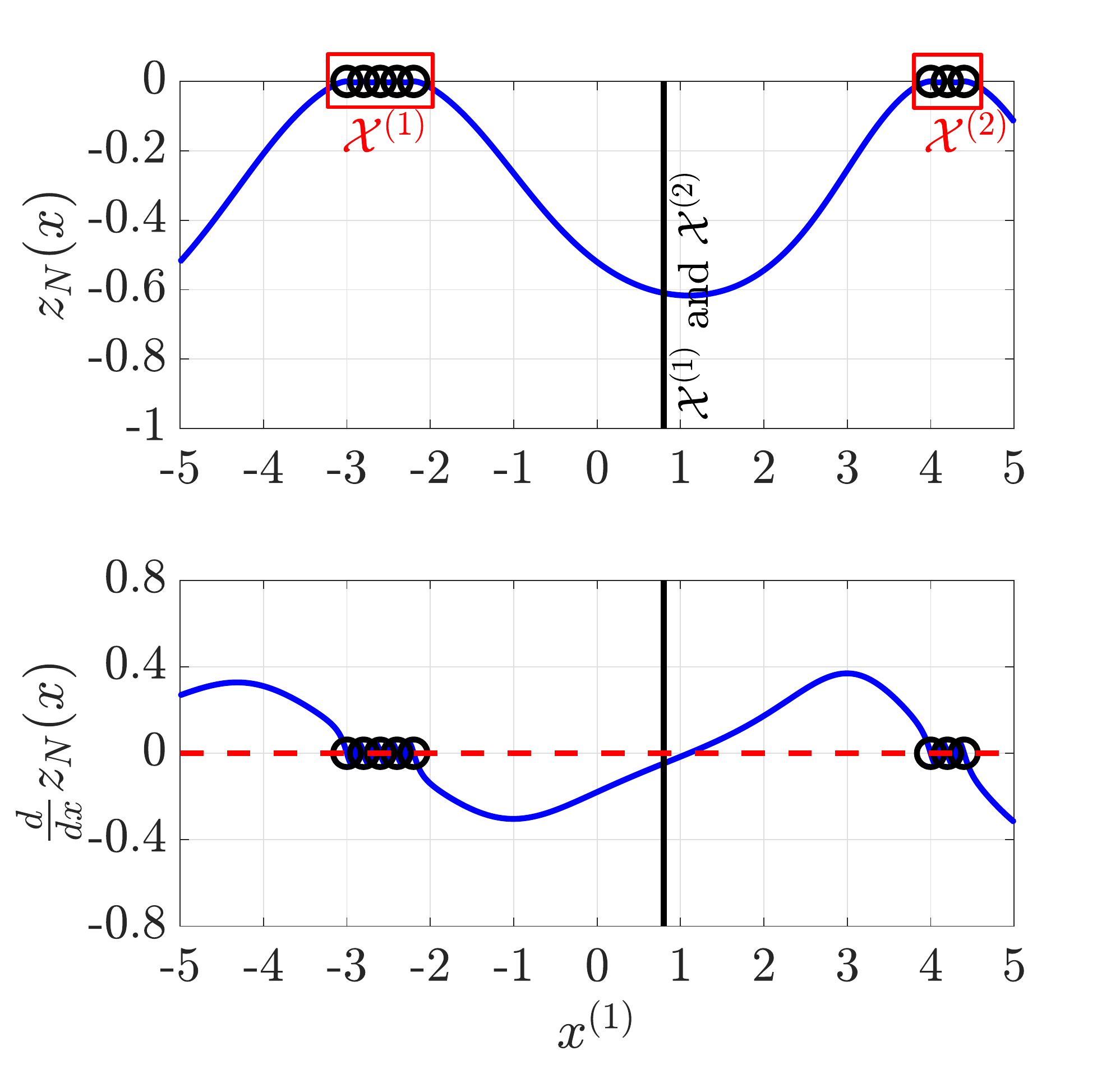}
    }
    \subfloat[$\mathcal{X} = \mathcal{X}^{(1)} \cup \mathcal{X}^{(2)} \cup \mathcal{X}^{(3)} \cup \mathcal{X}^{(4)} \cup \mathcal{X}^{(5)}$. \label{fig:IDW_distance_function_and_gradient_5_clusters}]{
        \centering
        \includegraphics[width=.5\textwidth]{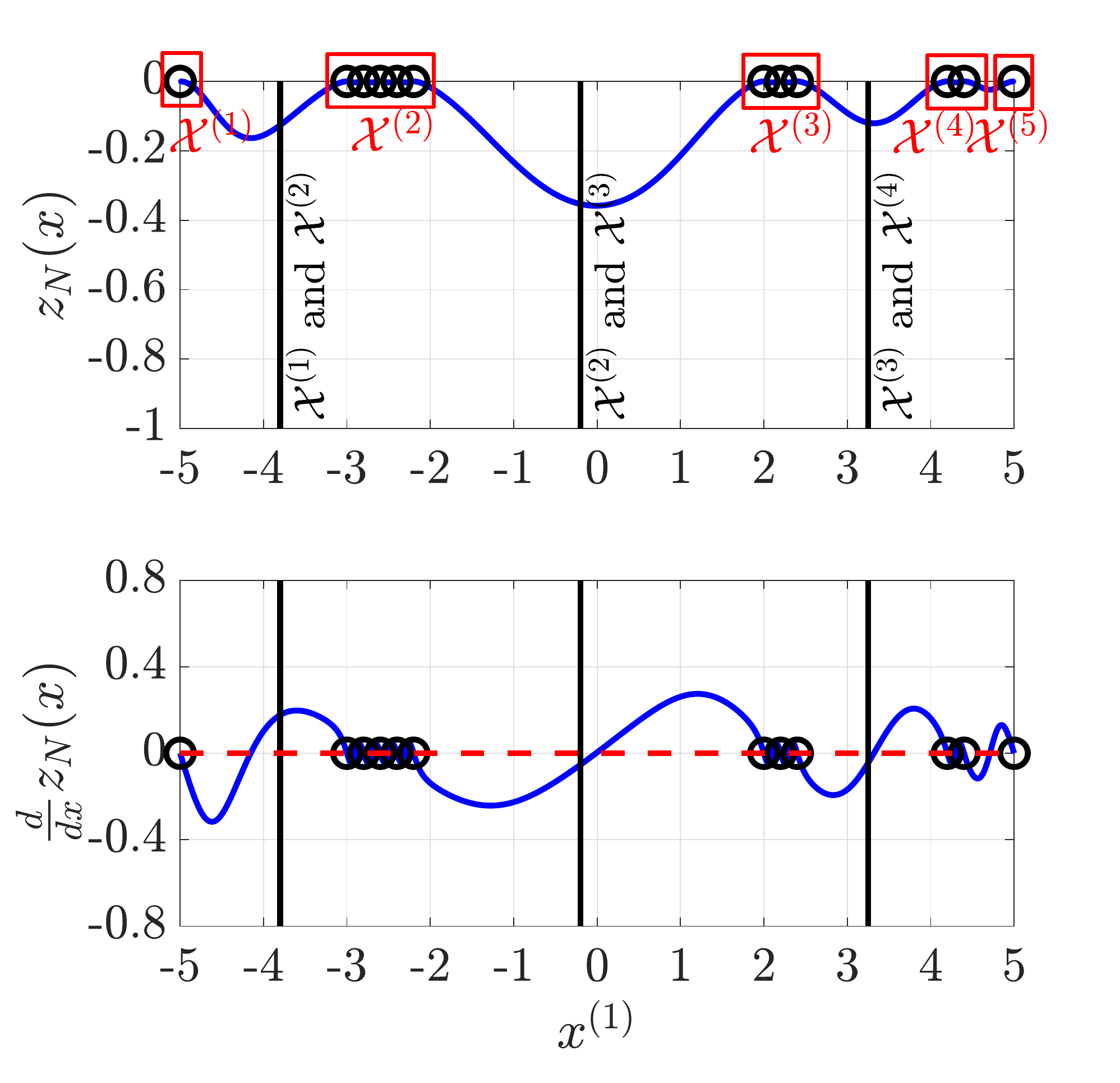}
    }

    \caption{
        \label{fig:IDW_distance_function_and_gradient}
        One-dimensional examples of the \IDW{} distance function $z_N\left(x\right)$ in \eqref{eq:Inverse_distance_weighting_distance_function} and its gradient $\nabla_{x} z_N\left(x\right)$ in \eqref{eq:Inverse_distance_weighting_distance_function_gradient} in the four analyzed cases. The red boxes mark the different clusters while the red dashed lines highlight the values of $x$ for which the first derivative of $z_N\left(x\right)$ is zero. Finally, the black vertical lines mark the midpoints (either between points or centroids of the clusters). Only a portion of all possible midpoints between centroids has been reported in the general case. 
        Notice that the midpoints in \figname{} \ref{fig:IDW_distance_function_and_gradient_2_clusters} and \figname{} \ref{fig:IDW_distance_function_and_gradient_5_clusters} are quite close to the local minimizers of $z_N\left(\boldsymbol{x}\right)$, while the midpoint in \figname{} \ref{fig:IDW_distance_function_and_gradient_2_points} is an exact local solution of Problem \eqref{eq:Minimization_of_IDW_distance_function_optimization_problem_simplified}.
    }
\end{figure}

\paragraph{Case $\mathcal{X} = \mathcal{X}^{(1)} \cup \mathcal{X}^{(2)}$ ($N > 2$)} Suppose now that the samples contained in $\mathcal{X}$ \eqref{eq:sample_set_X} can be partitioned into two clusters:
\begin{itemize}
    \item $\mathcal{X}^{(1)} = \left\{\boldsymbol{x}_1, \ldots, \boldsymbol{x}_{N_1}\right\}$ ($\lvert \mathcal{X}^{(1)}\rvert = N_1$),
    \item $\mathcal{X}^{(2)} = \left\{\boldsymbol{x}_{N_1 + 1}, \ldots, \boldsymbol{x}_{N}\right\}$ ($\lvert \mathcal{X}^{(2)}\rvert = N - N_1$),
\end{itemize}
such that $\mathcal{X}^{(1)} \cap \mathcal{X}^{(2)} = \emptyset$ and $\mathcal{X}^{(1)} \cup \mathcal{X}^{(2)} = \mathcal{X}$.
Consider the midpoint between the centroids of each cluster:
\begin{equation}
    \label{eq:Midpoint_2_clusters_case}
    \boldsymbol{x_{\mu}} = \frac{1}{2} \cdot \left[\frac{\sum_{\boldsymbol{x}_{i}\in\mathcal{X}^{(1)}} \boldsymbol{x}_{i}}{N_1}+\frac{\sum_{\boldsymbol{x}_{i}\in\mathcal{X}^{(2)}} \boldsymbol{x}_{i} }{N - N_1} \right].
\end{equation}
We make the simplifying assumption that all the points contained inside each cluster are quite close to each other, i.e. $\boldsymbol{x}_{1} \approx \boldsymbol{x}_{2} \approx \ldots \approx \boldsymbol{x}_{N_1}$ and $\boldsymbol{x}_{N_1 + 1} \approx \boldsymbol{x}_{N_1 + 2} \approx \ldots \approx \boldsymbol{x}_{N}$. Then, the midpoint in \eqref{eq:Midpoint_2_clusters_case} is approximately equal to $\boldsymbol{x_{\mu}} \approx \frac{\boldsymbol{x}_1 + \boldsymbol{x}_N}{2}$. 
Moreover, we have that $w_{1}\left(\boldsymbol{x_{\mu}}\right)\approx\ldots\approx w_{N_{1}}\left(\boldsymbol{x_{\mu}}\right)\approx w_{N_{1}+1}\left(\boldsymbol{x_{\mu}}\right)\approx\ldots\approx w_{N}\left(\boldsymbol{x_{\mu}}\right)$. Thus, the gradient of the \IDW{} distance function at $\boldsymbol{x_{\mu}}$ is approximately equal to:
\begin{align*}
    \nabla_{\boldsymbol{x}}z_N\left(\boldsymbol{x_{\mu}}\right) & =-\frac{4}{\pi}\cdot\frac{\sum_{i=1}^{N}\left(\boldsymbol{x_{\mu}}-\boldsymbol{x}_{i}\right)\cdot w_{i}\left(\boldsymbol{x_{\mu}}\right)^{2}}{1+\left[\sum_{i=1}^{N}w_{i}\left(\boldsymbol{x_{\mu}}\right)\right]^{2}}                                                                                    \\
                                                                & \approx -\frac{4}{\pi}\!\cdot\!\frac{w_{1}\left(\boldsymbol{x_{\mu}}\right)^2}{1\!+\!\left[N\cdot w_{1}\left(\boldsymbol{x_{\mu}}\right)\right]^{2}}\!\cdot\!\left[ \left(\frac{N}{2}\!-\!N_1\right)\!\cdot\!\boldsymbol{x}_1\!+\!\left(- \frac{N}{2}\!+\!N_1\right)\!\cdot\! \boldsymbol{x}_{N} \right].
\end{align*}
Clearly, if the clusters are nearly equally sized, i.e. $N_1 \approx N - N_1 \approx \frac{N}{2}$, then:
\begin{equation*}
    \nabla_{\boldsymbol{x}}z_N\left(\boldsymbol{x_{\mu}}\right) \approx \boldsymbol{0}_{n},
\end{equation*}
reaching a similar result to the one that we have seen for the case $N = 2$. 

\paragraph{General case ($N > 2$)} Any set of samples $\mathcal{X}$ in \eqref{eq:sample_set_X} can be partitioned into an arbitrary number of disjoint clusters, say $K \in \mathbb{N}, K \leq \lvert\mathcal{X}\rvert = N$, i.e.:
\begin{equation*}
    \mathcal{X}= \mathcal{X}^{(1)}\cup\mathcal{X}^{(2)}\cup\ldots\cup\mathcal{X}^{(K)}, \quad \text{such that } \mathcal{X}^{(i)} \cap \mathcal{X}^{(j)} = \emptyset, \forall i \neq j.
\end{equation*}
In this case, finding the local solutions of Problem \eqref{eq:Minimization_of_IDW_distance_function_optimization_problem_simplified} explicitly, or even approximately, is quite complex. Heuristically speaking, if the clusters are \quotes{well spread} (i.e. all the points contained inside each cluster $\mathcal{X}^{(i)}$ are sufficiently far away from the others in $\mathcal{X}^{(j)}, j = 1, \ldots, K, j \neq i$), then we can approximately deduce where the local minimizers of $z_N\left(\boldsymbol{x}\right)$ in \eqref{eq:Inverse_distance_weighting_distance_function} are located.
\begin{figure}[!htb]
    \centering
    \includegraphics[width=0.8\textwidth]{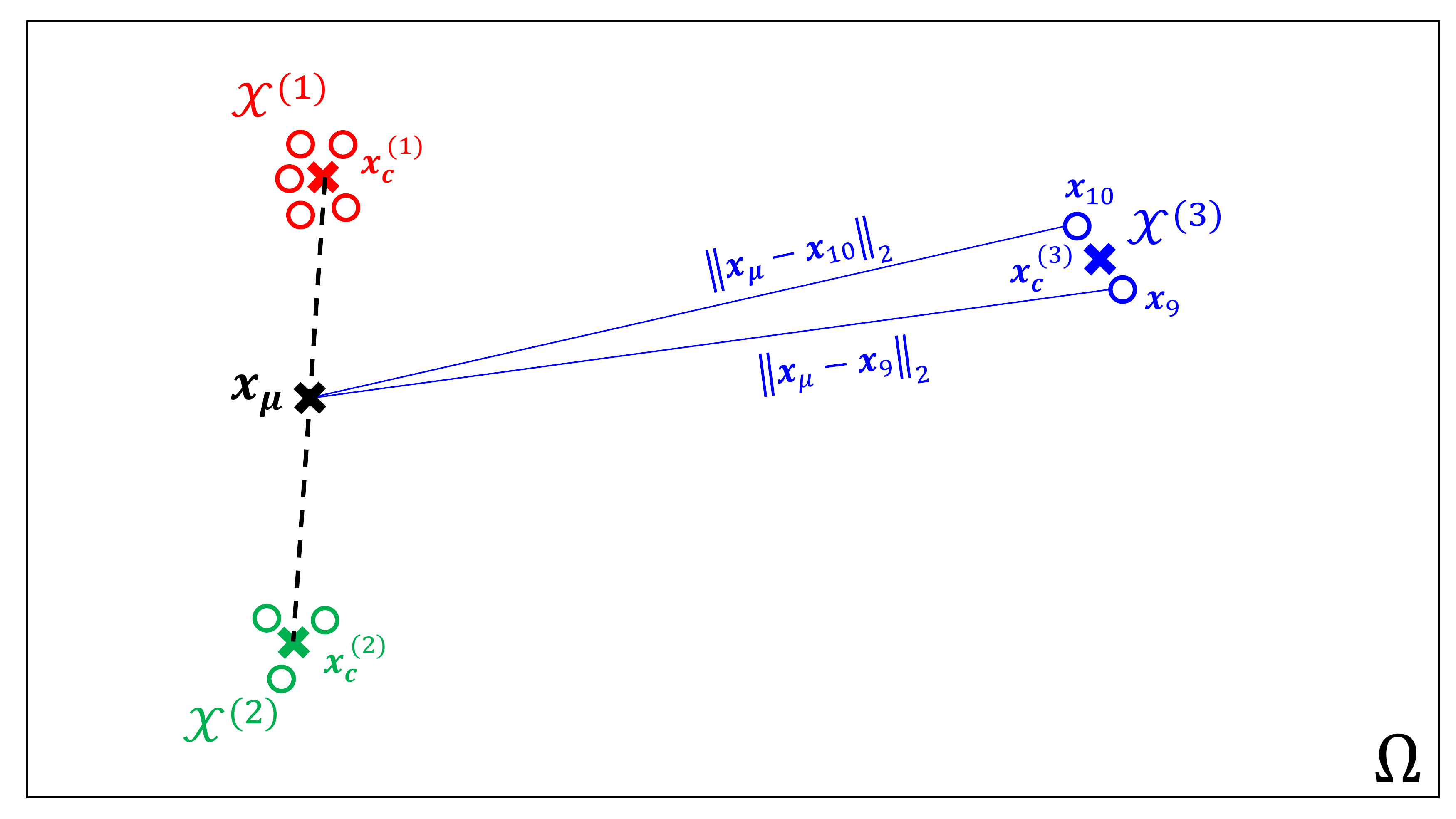}

    \caption{
        \label{fig:IDW_general_case_example}
        Two-dimensional example of \quotes{well spread} clusters, highlighted with different colors. The circles denote the points contained in $\mathcal{X}$ while the crosses represent the centroids of $\mathcal{X}^{(1)}, \mathcal{X}^{(2)}$ and $\mathcal{X}^{(3)}$. $\boldsymbol{x_{\mu}}$ is the midpoint between the centroids of clusters $\mathcal{X}^{(1)}$ and $\mathcal{X}^{(2)}$. Finally, the blue lines highlight the distances between the samples of cluster $\mathcal{X}^{(3)}$ and $\boldsymbol{x_{\mu}}$.
    }
\end{figure}
For instance, \figname{} \ref{fig:IDW_general_case_example} depicts a set of samples $\mathcal{X}$ that has been partitioned into three clusters, $\mathcal{X}^{(1)}, \mathcal{X}^{(2)}$ and $\mathcal{X}^{(3)}$, and for which the previous hypothesis is satisfied. 
In the general case, given the clusters $\mathcal{X}^{(i)}$ and $\mathcal{X}^{(j)}$, we compute their centroids $\boldsymbol{x}^{(i)}_{\boldsymbol{c}}$, $\boldsymbol{x}^{(j)}_{\boldsymbol{c}}$, and the corresponding midpoint $\boldsymbol{x_{\mu}}$ between them as:
\begin{subequations}
    \begin{align}
        \label{eq:Centroids_and_midpoint_of_clusters:centroid}
        \boldsymbol{x}^{(k)}_{\boldsymbol{c}} & = \frac{\sum_{\boldsymbol{x}_i \in \mathcal{X}^{(k)}} \boldsymbol{x}_i}{\lvert\mathcal{X}^{(k)}\rvert} & \text{(centroid of $k$-th cluster)}, \\
        \label{eq:Centroids_and_midpoint_of_clusters:midpoint}
        \boldsymbol{x_{\mu}}                  & = \frac{\boldsymbol{x}^{(i)}_{\boldsymbol{c}} + \boldsymbol{x}^{(j)}_{\boldsymbol{c}}}{2}              & \text{(midpoint)}.
    \end{align}
\end{subequations}
Going back to the example depicted in \figname{} \ref{fig:IDW_general_case_example}, if we consider the midpoint $\boldsymbol{x_{\mu}}$ between the centroids of $\mathcal{X}^{(1)}$ and $\mathcal{X}^{(2)}$, due to the \quotes{well spread} hypothesis we can say that $\euclideannorm{\boldsymbol{x_{\mu}} - \boldsymbol{x}_i} \gg 0, \forall \boldsymbol{x}_i \in \mathcal{X}^{(3)}$, making the contributions of the points inside the third cluster negligible when evaluating $z_N\left(\boldsymbol{x}\right)$ in \eqref{eq:Inverse_distance_weighting_distance_function} at $\boldsymbol{x_{\mu}}$. Therefore, the \IDW{} distance function at $\boldsymbol{x_{\mu}}$ is approximately equal to:
\begin{align*}
    z_N\left(\boldsymbol{x_{\mu}}\right) & = {-\frac{2}{\pi}}\cdot \arctan
    \left\{\left[
        \sum_{k = 1}^{3} \left(\sum_{\boldsymbol{x}_i \in \mathcal{X}^{(k)}} \frac{1}{\euclideannorm{\boldsymbol{x_{\mu}} - \boldsymbol{x}_i}^2}\right)
    \right]^{-1}\right\}                                                        \\
                                         & \approx {-\frac{2}{\pi}}\cdot\arctan
    \left[\left(
        \sum_{\boldsymbol{x}_i \in \mathcal{X}^{(1)}} \frac{1}{\euclideannorm{\boldsymbol{x_{\mu}} - \boldsymbol{x}_i}^2} + \sum_{\boldsymbol{x}_i \in \mathcal{X}^{(2)}} \frac{1}{\euclideannorm{\boldsymbol{x_{\mu}} - \boldsymbol{x}_i}^2}
        \right)^{-1} \right].
\end{align*}
In general, given $K$ clusters, if these are \quotes{well spread}, then we can consider each possible couple of clusters separately and neglect the contributions of the remaining ones. Approximately speaking, we could split the general case into $\begin{pmatrix}K \\
        2
    \end{pmatrix}$ distinct problems that read as follows: find the stationary points of the \IDW{} distance function $z_{N_{i \cup j}}\left(\boldsymbol{x}\right)$ in \eqref{eq:Inverse_distance_weighting_distance_function} defined from the set of samples $\mathcal{X}^{(i)}\cup \mathcal{X}^{(j)}, i \neq j$ and $N_{i \cup j} = \lvert\mathcal{X}^{(i)}\cup \mathcal{X}^{(j)}\rvert$. Hence, rough locations of the stationary points of $z_{N_{i \cup j}}\left(\boldsymbol{x}\right)$ can be found by following the same rationale proposed for the previously analyzed cases.
\\\\
Some one-dimensional examples of all the previously analyzed situations are reported in \figname{} \ref{fig:IDW_distance_function_and_gradient}.

\subsubsection{Min-max rescaling and augmented sample set}
Given a generic set of samples $\mathcal{X} = \left\{\boldsymbol{x}_1, \ldots, \boldsymbol{x}_N\right\}$ and a  multivariable function $h:\mathbb{R}^{n} \to \mathbb{R}$, \important{min-max rescaling} (or normalization) \cite{han2011data} rescales $h\left(\boldsymbol{x}\right)$ as:
\begin{equation}
    \label{eq:min-max_rescaling}
    \bar{h}\left(\boldsymbol{x}; \mathcal{X}\right) = \frac{h\left(\boldsymbol{x}\right) - h_{min}\left(\mathcal{X}\right)}{\Delta H\left(\mathcal{X}\right)},
\end{equation}
where\footnote{Note that, to avoid dividing by zero in \eqref{eq:min-max_rescaling}, $\Delta H\left(\mathcal{X}\right)$ can be set to $h_{max}\left(\mathcal{X}\right)$ or $1$ whenever $h_{min}\left(\mathcal{X}\right) = h_{max}\left(\mathcal{X}\right) \neq 0$ or $h_{min}\left(\mathcal{X}\right) = h_{max}\left(\mathcal{X}\right) = 0$ respectively.}:
\begin{subequations}
    \begin{align}
        h_{min}\left(\mathcal{X}\right)  & = \min{\boldsymbol{x}_i \in \mathcal{X}}h(\boldsymbol{x}_i),       \\
        h_{max}\left(\mathcal{X}\right)  & = \max{\boldsymbol{x}_i \in \mathcal{X}}h(\boldsymbol{x}_i),       \\
        \Delta H\left(\mathcal{X}\right) & = h_{max}\left(\mathcal{X}\right)-h_{min}\left(\mathcal{X}\right).
    \end{align}
\end{subequations}
The objective of min-max rescaling is to obtain a function with range $[0, 1]$, i.e. we would like to have $\bar{h}:\mathbb{R}^{n} \to [0, 1]$. Clearly, the quality of the normalization depends on the information brought by the samples contained inside $\mathcal{X}$, as pointed out in the following Remark.
\begin{remark}
    \label{rem:Quality_of_min-max_rescaling}
    We can observe that:
    \begin{enumerate}
        \item If $\mathcal{X}$ contains the global minimizer(s) and maximizer(s) of $h\left(\boldsymbol{x}\right)$, then $\bar{h}\left(\boldsymbol{x}\right)$ defined as in \eqref{eq:min-max_rescaling} effectively has codomain $[0, 1]$,
        \item Otherwise, we can only ensure that $0 \leq \bar{h}\left(\boldsymbol{x}_i\right) \leq 1, \forall \boldsymbol{x}_i \in \mathcal{X}$.
        \item In general, if we increase the amount of distinct samples in $\mathcal{X}$, then the rescaling of $h\left(\boldsymbol{x}\right)$ gets better (or, worst case, stays the same).
    \end{enumerate}
\end{remark}
Going back to the problem of rescaling the \IDW{} distance function $z_N\left(\boldsymbol{x}\right)$ in \eqref{eq:Inverse_distance_weighting_distance_function}, if we were to apply \eqref{eq:min-max_rescaling} using the set of previously evaluated samples $\mathcal{X}$ in \eqref{eq:sample_set_X}, then it would not be effective since $z_N\left(\boldsymbol{x}_i\right) = 0, \forall \boldsymbol{x}_i \in \mathcal{X}$ (see Proposition \ref{prop:global_maximizers_of_IDW_distance}). Instead, we have opted to generate a sufficiently expressive \important{augmented sample set} $\mathcal{X}_{aug} \supset \mathcal{X}$ and perform min-max normalization using $\mathcal{X}_{aug}$ instead of $\mathcal{X}$.

Consider the general case described in Section \ref{subsubsec:Stationary_points_for_IDW_distance_function}. Then,  the augmented sample set $\mathcal{X}_{aug}$ can be built in the following fashion: 
\begin{enumerate}
    \item Partition the points in $\mathcal{X}$ \eqref{eq:sample_set_X} into different clusters. Here, for simplicity, we  fix a-priori the number $K_{aug} \in \mathbb{N}$ of clusters and apply $K$-means clustering \cite{lloyd1982least, hastie2009elements, bishop2006pattern} to obtain the sets $\mathcal{X}^{(1)}, \ldots, \mathcal{X}^{(K_{aug})}$;
    \item Compute the centroids of each cluster, using \eqref{eq:Centroids_and_midpoint_of_clusters:centroid}, and group them inside the set $\mathcal{X}_c = \left\{\boldsymbol{x}_{\boldsymbol{c}}^{(1)}, \ldots, \boldsymbol{x}_{\boldsymbol{c}}^{(K_{aug})} \right\}$;
    \item Calculate all the midpoints $\boldsymbol{x_\mu}$ between each possible couple of centroids $\boldsymbol{x}_{\boldsymbol{c}}^{(i)}, \boldsymbol{x}_{\boldsymbol{c}}^{(j)} \in \mathcal{X}_c, \boldsymbol{x}_{\boldsymbol{c}}^{(i)} \neq \boldsymbol{x}_{\boldsymbol{c}}^{(j)},$ using \eqref{eq:Centroids_and_midpoint_of_clusters:midpoint}; 
    \item Build the augmented sample set as $\mathcal{X}_{aug} = \mathcal{X} \cup \mathcal{X}_\mu$, where $\mathcal{X}_\mu$ is the set which groups all the previously computed midpoints.
\end{enumerate}
Clearly, as highlighted by \eqref{eq:min-max_rescaling} and Remark \ref{rem:Quality_of_min-max_rescaling}, if $\mathcal{X}_{aug}$ contains points that are close (or equal) to the global minimizer(s) and maximizer(s) of $z_N\left(\boldsymbol{x}\right)$ in \eqref{eq:Inverse_distance_weighting_distance_function}, then the quality of the min-max rescaling of the \IDW{} distance function improves.

Algorithm \ref{alg:Augmented_sample_set} formalizes these steps while also taking into consideration the case $\lvert\mathcal{X}\rvert \leq K_{aug}$ (for which no clustering is performed). Note that we also include the bounds $\boldsymbol{l}$ and $\boldsymbol{u}$ inside $\mathcal{X}_{c}$ and $\mathcal{X}_{aug}$ for two reasons: (i) $\boldsymbol{l}$ or $\boldsymbol{u}$ might actually be the solutions of Problem \eqref{eq:Minimization_of_IDW_distance_function_optimization_problem_simplified}
\footnote{We could add all $2^n$ vertices of the box defined by the bound constraints $\left\{\boldsymbol{x}: \boldsymbol{l} \leq \boldsymbol{x} \leq \boldsymbol{u}\right\}$. However, we have preferred to include only $\boldsymbol{l}$ and $\boldsymbol{u}$ to avoid increasing the cardinality of the augmented sample set, especially in the case of high-dimensional problems.}
and (ii) given that we also want to rescale $\hat{f}_N\left(\boldsymbol{x}\right)$ in \eqref{eq:RBF_surrogate_model}, adding additional points to the augmented sample set improves the quality of min-max normalization (see Remark \ref{rem:Quality_of_min-max_rescaling}). Notice that the number of points contained inside $\mathcal{X}_{aug}$ obtained from Algorithm \ref{alg:Augmented_sample_set} is:
\begin{equation*}
    \lvert\mathcal{X}_{aug}\rvert = \lvert\mathcal{X}\rvert + \begin{pmatrix}K_{aug} + 2 \\
        2
    \end{pmatrix}
    + 2.
\end{equation*}
Therefore, to avoid excessively large augmented sample sets, $K_{aug}$ needs to be chosen appropriately.

As a final remark, we point out that we could perform min-max normalization in \eqref{eq:min-max_rescaling} by using the real minima and maxima of $z_N\left(\boldsymbol{x}\right)$ in \eqref{eq:Inverse_distance_weighting_distance_function} and $\hat{f}_N\left(\boldsymbol{x}\right)$ in \eqref{eq:RBF_surrogate_model}, which can be obtained by solving four additional global optimization problems. However, we have preferred to stick with the proposed heuristic way since we are not interested in an extremely accurate rescaling and, also, to avoid potentially large overhead times due to solving additional global optimization problems.
\begin{algorithm}[!htb]
    \setalgorithmstretch
    \setalgorithmfontsize
    \caption{Computation of $\mathcal{X}_{aug}$ for min-max rescaling}
    \label{alg:Augmented_sample_set}
    \textbf{Input}:
    \begin{algparams}
        \item Set of samples $\mathcal{X}$ in \eqref{eq:sample_set_X};
        \item Number of clusters $K_{aug} \in \mathbb{N}$;
        \item Lower bounds $\boldsymbol{l} \in \mathbb{R}^n$ and upper bounds $\boldsymbol{u} \in \mathbb{R}^n$ of Problem \eqref{eq:preference-based_optimization_problem}.
    \end{algparams}

    \textbf{Output}:
    \begin{algparams}
        \item Augmented sample set $\mathcal{X}_{aug}\supset \mathcal{X}$.
    \end{algparams}
    \hrule
    \begin{algorithmic}[1]
        \If{$\lvert \mathcal{X} \rvert > K_{aug}$}
        \State Perform $K$-means clustering \cite{lloyd1982least, hastie2009elements, bishop2006pattern} to
        group the samples in $\mathcal{X}$ into $K_{aug}$ clusters $\mathcal{X}^{(1)}, \ldots, \mathcal{X}^{(K_{aug})}$
        \State Compute the set of centroids $\mathcal{X}_c$ using \eqref{eq:Centroids_and_midpoint_of_clusters:centroid}:
        \begin{flalign*}
            \hskip\parindent
            \mathcal{X}_c =\left\{\boldsymbol{x}_{\boldsymbol{c}}^{(k)}: \boldsymbol{x}_{\boldsymbol{c}}^{(k)} = \frac{\sum_{\boldsymbol{x}_i \in \mathcal{X}^{(k)}} \boldsymbol{x}_i}{\lvert\mathcal{X}^{(k)}\rvert}, k = 1, \ldots, K_{aug}\right\} &  &
        \end{flalign*}
        \Else
        \State Set $\mathcal{X}_c = \mathcal{X}$
        \EndIf
        \State Add the bounds to $\mathcal{X}_c$: $\mathcal{X}_c = \mathcal{X}_c \cup \left\{ \boldsymbol{l},\boldsymbol{u}\right\}$
        \State Group all possible couples of $\mathcal{X}_c$ (without repetition):
        \begin{flalign*}
            \mathcal{X}_{couples} =\left\{ \left(\boldsymbol{x}_{\boldsymbol{c}}^{(i)},\boldsymbol{x}_{\boldsymbol{c}}^{(j)}\right):\boldsymbol{x}_{\boldsymbol{c}}^{(i)},\boldsymbol{x}_{\boldsymbol{c}}^{(j)}\in\mathcal{X}_c,\boldsymbol{x}_{\boldsymbol{c}}^{(i)}\neq\boldsymbol{x}_{\boldsymbol{c}}^{(j)}\right\} &  &
        \end{flalign*}
        \State Calculate the midpoints between all the couples inside $\mathcal{X}_{couples}$, obtaining the set:
        \begin{flalign*}
            \mathcal{X}_{\mu}=\left\{ \boldsymbol{x_{\mu}}:\boldsymbol{x_{\mu}}=\frac{\boldsymbol{x}_{\boldsymbol{c}}^{(i)}+\boldsymbol{x}_{\boldsymbol{c}}^{(j)}}{2},\left(\boldsymbol{x}_{\boldsymbol{c}}^{(i)},\boldsymbol{x}_{\boldsymbol{c}}^{(j)}\right)\in\mathcal{X}_{couples}\right\} &  &
        \end{flalign*}
        \State Build the augmented sample set as $\mathcal{X}_{aug}=\mathcal{X}\cup\mathcal{X}_{\mu}\cup\left\{ \boldsymbol{l},\boldsymbol{u}\right\} $
    \end{algorithmic}
\end{algorithm}

\subsection{Definition of the acquisition function}
\label{subsec:Acquisition_function}
In this Section, we take advantage of the results on the stationary points of $z_N\left(\boldsymbol{x}\right)$ presented in Section \ref{subsubsec:Stationary_points_for_IDW_distance_function} to rescale the surrogate model and the exploration function. In particular, we define the following acquisition function:
\begin{equation}
    \label{eq:Acquisition_function_GLISp-r}
    a_N\left(\boldsymbol{x}\right) = \delta \cdot \hat{\bar{f}}_N\left(\boldsymbol{x}; \mathcal{X}_{aug}\right) + \left(1 - \delta\right) \cdot \bar{z}_N\left( \boldsymbol{x}; \mathcal{X}_{aug}\right),
\end{equation}
where $\hat{f}_N\left(\boldsymbol{x}\right)$ in \eqref{eq:RBF_surrogate_model} and $z_N\left(\boldsymbol{x}\right)$ in \eqref{eq:Inverse_distance_weighting_distance_function} have been rescaled using min-max normalization as in \eqref{eq:min-max_rescaling} and $\mathcal{X}_{aug}$ is generated by Algorithm \ref{alg:Augmented_sample_set}. $\delta \in \left[0, 1\right]$ is the exploration-exploitation trade-off weight; also note that $\delta = 0$ corresponds to pure exploration, while $\delta = 1$ results in pure exploitation. $a_N\left(\boldsymbol{x}\right)$ in \eqref{eq:Acquisition_function_GLISp-r} is similar to the acquisition function of \MSRSmethod{} (for black-box optimization) but here we use an ad-hoc augmented sample set instead of a randomly generated one and a different exploration function. We will refer to the algorithm that we will propose in Section \ref{sec:GLISp-r_and_convergence}, which uses $a_N\left(\boldsymbol{x}\right)$ in \eqref{eq:Acquisition_function_GLISp-r}, as \GLISprmethod{}, where the \quotes{\texttt{r}} highlights the min-max rescaling performed for the acquisition function.

A comparison between the terms of the acquisition functions in \eqref{eq:Acquisition_function_GLISp-r} \linebreak (\GLISprmethod{}) and in \eqref{eq:Acquisition_function_GLISp} (\GLISpmethod{}) is depicted in \figname{} \ref{fig:Comparison_between_acquisition_function_terms}. As the number of samples $N$ increases, the absolute values of $z_N\left(\boldsymbol{x}\right)$ in \eqref{eq:Inverse_distance_weighting_distance_function} get progressively smaller (see Section \ref{subsec:Improving_exploration_capabilities_of_GLISp}) and simply dividing $\hat{f}_N\left(\boldsymbol{x}\right)$ by $\Delta \hat{F}$ as in \eqref{eq:Acquisition_function_GLISp} is not enough to make the exploration and exploitation contributions comparable. Thus, unless $\delta$ in \eqref{eq:Acquisition_function_GLISp} is dynamically varied in between iterations of \GLISpmethod{}, solving Problem \eqref{eq:Next_sample_search_no_black-box_constraints} with $a_N\left(\boldsymbol{x}\right)$ in \eqref{eq:Acquisition_function_GLISp} becomes similar to performing pure exploitation. This, in turn, can make  \GLISpmethod{} more prone to getting stuck on local minima of Problem \eqref{eq:preference-based_optimization_problem} with no way of escaping (especially if the surrogate model is not expressive enough to capture the location of the global minimizer). Vice-versa, by performing min-max rescaling as proposed in \eqref{eq:Acquisition_function_GLISp-r}, the exploration and exploitation contributions stay comparable throughout the whole optimization process and approximately assume the same range. For this reason, it is also more straightforward to define $\delta$ in \eqref{eq:Acquisition_function_GLISp-r} compared to the weight in \eqref{eq:Acquisition_function_GLISp}.

From Propositions \ref{prop:Differentiability_of_surrogate_function} and \ref{prop:Differentiability_of_IDW_distance_function}, we can immediately deduce the following results on the differentiability of $a_N\left(\boldsymbol{x}\right)$ in \eqref{eq:Acquisition_function_GLISp-r}.
\begin{proposition}
    The acquisition function $a_N\left(\boldsymbol{x}\right)$ in \eqref{eq:Acquisition_function_GLISp-r}
    is differentiable everywhere provided that the surrogate model $\hat{f}_N\left(\boldsymbol{x}\right)$ in \eqref{eq:RBF_surrogate_model} is differentiable everywhere.
\end{proposition}
At each iteration of \GLISprmethod{} we find the next candidate for evaluation, i.e. $\boldsymbol{x}_{N+1}$, by solving Problem \eqref{eq:Next_sample_search_no_black-box_constraints} with the acquisition function in \eqref{eq:Acquisition_function_GLISp-r}. It is possible to use derivative-based optimization solvers since $a_N\left(\boldsymbol{x}\right)$ in \eqref{eq:Acquisition_function_GLISp-r} is differentiable everywhere. In general, the acquisition function is \important{multimodal} and thus it is better to employ a global optimization procedure. Moreover, $a_N\left(\boldsymbol{x} \right)$ is cheap to evaluate; therefore, we are not particularly concerned on its number of function evaluations.
\begin{figure}[!htb]
    \subfloat{
        \centering
        \includegraphics[width=.5\textwidth]{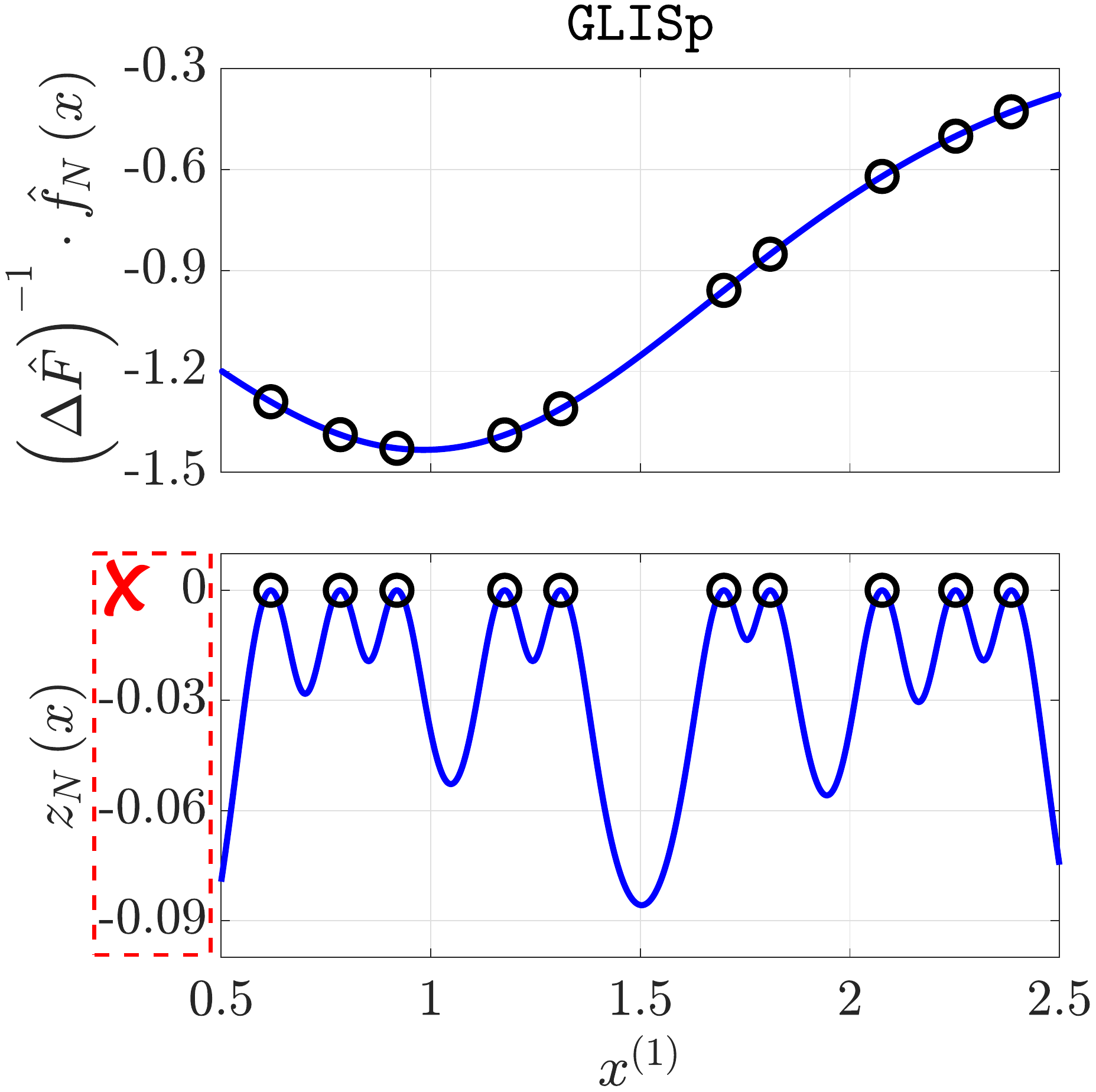}
        \includegraphics[width=.5\textwidth]{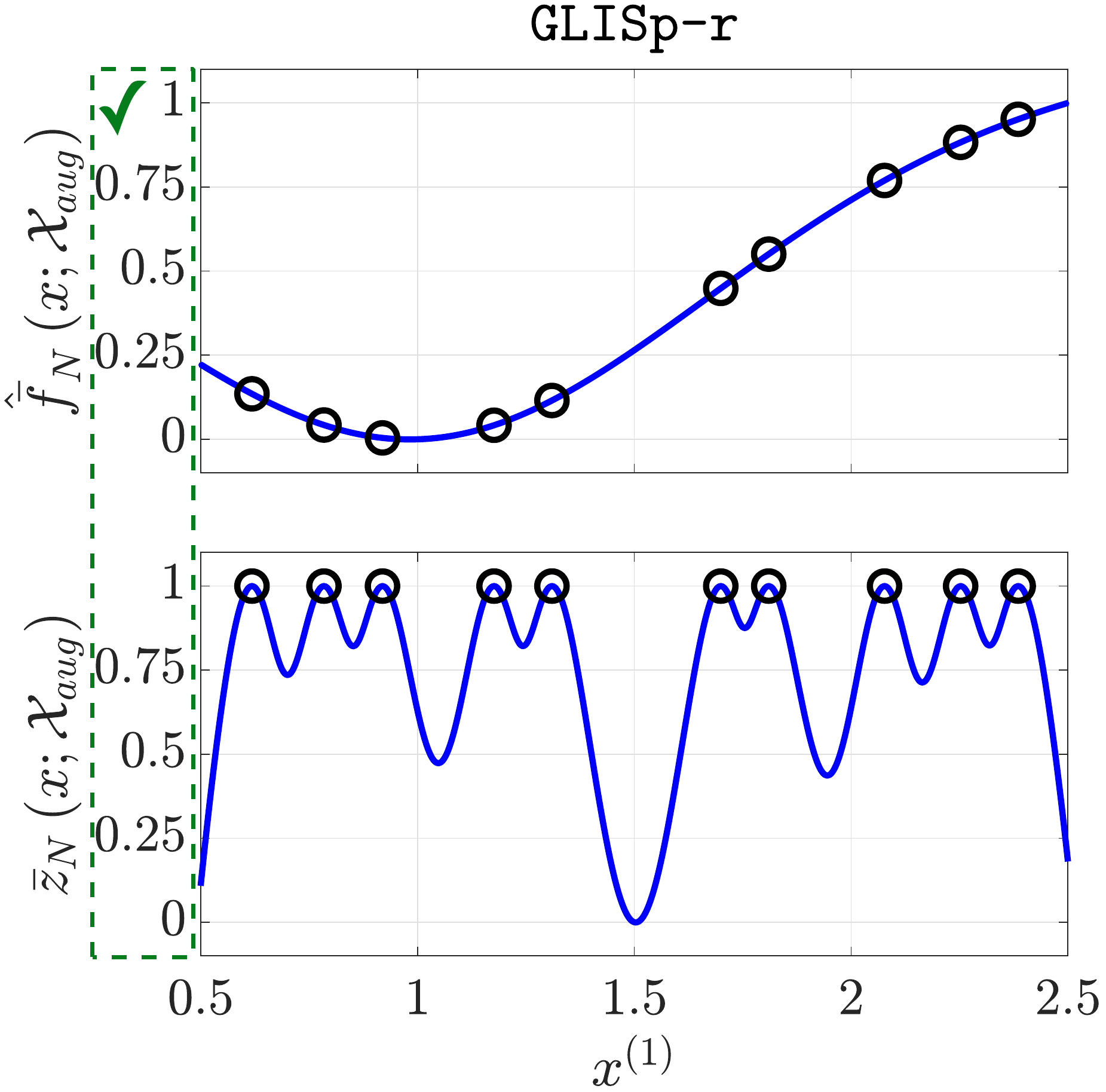}
    }


    \caption{
        \label{fig:Comparison_between_acquisition_function_terms}
        Comparison between the terms of the acquisition functions in \eqref{eq:Acquisition_function_GLISp} (left, \GLISpmethod{}) and in \eqref{eq:Acquisition_function_GLISp-r} (right, \GLISprmethod{}). 
        The scoring function $f\left(\boldsymbol{x}\right)$ is the \gramacyandleeGOP{} function while $\lvert\mathcal{X}\rvert = N = 10$. For \GLISprmethod{}, the number of clusters used to build $\mathcal{X}_{aug}$ through Algorithm \ref{alg:Augmented_sample_set} is $K_{aug} = 5$. Notice how, for \GLISpmethod{}, the exploration function $z_N\left(\boldsymbol{x}\right)$ is not comparable with the rescaled surrogate $\frac{\hat{f}_N\left(\boldsymbol{x}\right)}{\Delta \hat{F}}$ since it assumes values that are one to two orders of magnitude lower.
    }
\end{figure}
\subsubsection{Greedy $\delta$-cycling}
\label{subsubsec:Greedy_delta_cycling}
Many black-box optimization algorithms explicitly vary their emphasis on  exploration and exploitation in between iterations. Just to cite a few:
\begin{itemize}
    \item \GutmannRBF{} uses an acquisition function that is a measure of \quotes{bumpiness} of the \RBF{} surrogate that depends upon a target value $\tau$ to aim for.
          The author suggests to cycle the values of $\tau$ between $\tau = \min{\boldsymbol{x} \in \Omega} \hat{f}_N\left(\boldsymbol{x}\right)$  (local search) and $\tau = -\infty$ (global search).
    \item The authors of \MSRSmethod{}, which uses an acquisition function that is similar to \eqref{eq:Acquisition_function_GLISp-r}, propose to cycle between different values of $\delta$ as to prioritize exploration or exploitation more.
    \item In algorithm \SOSAmethod{}, which is a revisitation of \MSRSmethod{}, the weight $\delta$ is chosen in a random fashion at each iteration. Moreover, the authors adopt a greedy strategy, i.e. $\delta$ is kept unaltered until it fails to find a significantly better solution. 
\end{itemize}
In 
\GLISpmethod{}, the weight $\delta$ for $a_N\left(\boldsymbol{x}\right)$ in \eqref{eq:Acquisition_function_GLISp} is kept constant throughout the whole optimization process. Also, defining some form of cycling for such hyper-parameter can be quite complex 
since the additive terms that compose the acquisition function are not always comparable. In this work, we propose a strategy that is in between that of \MSRSmethod{} and \SOSAmethod{}, which we refer to as \important{greedy $\delta$-cycling}. We define a sequence of $N_{cycle} \in \mathbb{N}$ weights to cycle:
\begin{equation}
    \label{eq:Delta_cycle}
    \Delta_{cycle} = \langle \delta_0, \ldots, \delta_{N_{cycle - 1}}\rangle.
\end{equation}
$\Delta_{cycle}$ should contain values that are well spread within the $\left[0, 1\right]$ range as to properly alternate between local and global search. Greedy $\delta$-cycling operates as follows. Suppose that, at iteration $k$ of \GLISprmethod{}, we have at our disposal $\lvert\mathcal{X}\rvert = N$ samples and denote the trade-off weight $\delta$ in \eqref{eq:Acquisition_function_GLISp-r} as $\delta\left(k\right)$ to highlight the iteration number. Furthermore, assume $\delta \left(k\right) = \delta_j \in \Delta_{cycle}$, which is used to find the new candidate sample $\boldsymbol{x}_{N+1}$ at iteration $k$ by solving Problem \eqref{eq:Next_sample_search_no_black-box_constraints}. Then, if $\boldsymbol{x}_{N+1} \succ \boldsymbol{x_{best}}\left(N\right)$ (i.e. there has been some improvement), the trade-off weight is kept unchanged, $\delta \left(k+1\right) = \delta \left(k\right) = \delta_j$. Otherwise, we cycle the values in $\Delta_{cycle}$, obtaining $\delta \left(k+1\right) = \delta_{\left(j+1\right)\text{mod}{N_{cycle}}}$. Thus:
\begin{equation}
    \label{eq:greedy_delta_cycling}
    \delta \left(k+1\right) = \begin{cases}
        \delta_j                                       & \text{if $\boldsymbol{x}_{N+1} \succ \boldsymbol{x_{best}}\left(N\right)$}    \\
        \delta_{\left(j+1\right)\text{mod}{N_{cycle}}} & \text{if $\boldsymbol{x_{best}}\left(N\right) \succsim \boldsymbol{x}_{N+1}$}
    \end{cases}
\end{equation}
In Section \ref{sec:GLISp-r_and_convergence}, we will discuss the choice of the cycling sequence in \eqref{eq:Delta_cycle} more in detail and also cover its relationship with the global convergence of \GLISprmethod{}.
\section{Algorithm \texttt{GLISp-r} and convergence}
\label{sec:GLISp-r_and_convergence}
Algorithm \ref{alg:GLISp-r_algorithm} describes each step of the \GLISprmethod{} procedure. As with any \linebreak surrogate-based method, \texttt{GLISp-r} starts from an initial set of samples $\mathcal{X},  \lvert\mathcal{X}\rvert = N_{init} \in \mathbb{N}, N_{init} \geq 2$, generated by a \important{space-filling experimental design} \cite{vu2017surrogate}, such as a Latin Hypercube Designs (\LHD{}) \cite{mckay2000comparison}. The sets $\mathcal{B}$ in \eqref{eq:Set_of_preferences_B} and $\mathcal{S}$ in \eqref{eq:Mapping_set_for_preferences_S}, as well as the initial best candidate $\boldsymbol{x_{best}}\left(N_{init}\right)$, are obtained by asking the decision-maker to compare the samples in $\mathcal{X}$ \eqref{eq:sample_set_X} as proposed in Algorithm \ref{alg:queries_for_preference_based_optimization}, which prompts $M = N_{init} - 1$ queries. Once the initial sampling phase is concluded, the new candidate samples are obtained by solving Problem \eqref{eq:Next_sample_search_no_black-box_constraints}. 
The procedure is stopped once $\lvert\mathcal{X}\rvert = N_{max}$, where $N_{max} \in \mathbb{N}$ is a budget specified by the user. Overall, the decision-maker is queried $N_{max} - 1$ times.

\GLISprmethod{} follows the same scheme of \GLISpmethod{} but, additionally, at each iteration, builds the augmented sample set $\mathcal{X}_{aug}$ using Algorithm \ref{alg:Augmented_sample_set}. New candidate samples are found by minimizing the acquisition function in \eqref{eq:Acquisition_function_GLISp-r} instead of $a_N\left(\boldsymbol{x}\right)$ in \eqref{eq:Acquisition_function_GLISp}. Moreover, $\delta$ is cycled as proposed in Section \ref{subsubsec:Greedy_delta_cycling}. Similarly to \GLISpmethod{}, the shape parameter $\epsilon$ of the surrogate model in \eqref{eq:RBF_surrogate_model} is recalibrated at certain iterations of the algorithm, as specified by the set $\mathcal{K}_{R} \subseteq \left\{1, \ldots, N_{max} - N_{init}\right\}$, using grid-search Leave One Out Cross-Validation \linebreak (\LOOCV{}). In particular, at each iteration $k \in \mathcal{K}_{R}$, $\epsilon$ for $\hat{f}_N\left(\boldsymbol{x}\right)$ in \eqref{eq:RBF_surrogate_model} is selected among a set $\mathcal{E}_{\LOOCV{}}$ of possible shape parameters as the one whose corresponding surrogate preference function $\surrpreffun{N}\left(\boldsymbol{x}_i, \boldsymbol{x}_j\right)$ in \eqref{eq:RBF_surrogate_preference_function} classifies (out-of-sample) most of the preferences in $\mathcal{B}$ and $S$ correctly \cite{Bemporad2021}. Lastly, consistently with \GLISpmethod{}, Problem \eqref{eq:preference-based_optimization_problem} is rescaled so that each decision variable assumes the $\left[-1, 1\right]$ range (at least inside $\Omega$).
\begin{algorithm}[!htb]
    \setalgorithmstretch
    \setalgorithmfontsize
    \caption{\GLISprmethod{}}
    \label{alg:GLISp-r_algorithm}
    \textbf{Input}:
    \begin{algparams}
        \item Constraint set $\Omega$ of Problem \eqref{eq:preference-based_optimization_problem};
        \item Number of initial samples $N_{init} \in \mathbb{N}, N_{init} \geq 2$;
        \item Budget $N_{max} \in \mathbb{N}, N_{max} > N_{init}$;
        \item Hyper-parameters for the surrogate model $\hat{f}_N\left(\boldsymbol{x}\right)$ in \eqref{eq:RBF_surrogate_model}, i.e. shape parameter $\epsilon \in \mathbb{R}_{>0}$, radial function $\varphi(\cdot)$, regularization parameter $\lambda \in \mathbb{R}_{\geq 0}$ and tolerance $\sigma \in \mathbb{R}_{>0}$;
        \item Cycling sequence $\Delta_{cycle}$ in \eqref{eq:Delta_cycle} for the acquisition function $a_N\left( \boldsymbol{x}\right)$ in \eqref{eq:Acquisition_function_GLISp-r};
        \item Number of clusters $K_{aug} \in \mathbb{N}$ for the augmented sample set $\mathcal{X}_{aug}$ generated by Algorithm \ref{alg:Augmented_sample_set};
        \item Possible shape parameters $\mathcal{E}_{\LOOCV{}}$ for the recalibration of the surrogate model $\hat{f}_N\left(\boldsymbol{x}\right)$ in \eqref{eq:RBF_surrogate_model};
        \item Set of indexes for the recalibration of the surrogate model $\hat{f}_N\left(\boldsymbol{x}\right)$ in \eqref{eq:RBF_surrogate_model}, i.e. $\mathcal{K}_{R} \subseteq \left\{1, \ldots, N_{max} - N_{init}\right\}$.
    \end{algparams}

    \textbf{Output}:
    \begin{algparams}
        \item Best sample obtained by the procedure $\boldsymbol{x_{best}}\left(N_{max}\right)$.
    \end{algparams}
    \hrule
    \begin{algorithmic}[1]
        \State Rescale Problem \eqref{eq:preference-based_optimization_problem} as in \GLISpmethod{}
        \State Generate a set $\mathcal{X}$ in \eqref{eq:sample_set_X} of $N_{init}$ starting points using a \LHD{} \cite{mckay2000comparison}
        \State Evaluate the samples in $\mathcal{X}$ by querying the decision-maker as in Algorithm \ref{alg:queries_for_preference_based_optimization}, obtaining the sets $\mathcal{B}$ in \eqref{eq:Set_of_preferences_B} and $\mathcal{S}$ in \eqref{eq:Mapping_set_for_preferences_S}, as well as the best candidate $\boldsymbol{x_{best}}\left(N_{init}\right)$
        \State Set $N = N_{init}$ and $M = \lvert\mathcal{B}\rvert$
        \State Set $\delta = \delta_0 \in \Delta_{cycle}$ and $j = 0$
        \For{$k = 1, 2, \ldots, N_{max} - N_{init}$}
        \IfThen{$k \in \mathcal{K}_{R}$}{recalibrate the surrogate model $\hat{f}_N\left( \boldsymbol{x}\right)$ in \eqref{eq:RBF_surrogate_model} as in \GLISpmethod{}}
        \State Build the surrogate model $\hat{f}_N\left(\boldsymbol{x}\right)$ in \eqref{eq:RBF_surrogate_model} from $\mathcal{X}, \mathcal{B}$ and $\mathcal{S}$ by solving Problem \eqref{eq:Beta_computation_optimization_problem}
        \State Generate the augmented sample set $\mathcal{X}_{aug}$ through Algorithm \ref{alg:Augmented_sample_set}
        \State Look for the next candidate sample $\boldsymbol{x}_{N+1}$ by solving Problem \eqref{eq:Next_sample_search_no_black-box_constraints} with $a_N\left( \boldsymbol{x}\right)$ in \eqref{eq:Acquisition_function_GLISp-r}
        \State Let the decision-maker express the preference $b_{M+1} = \preffun \left(\boldsymbol{x}_{N+1}, \boldsymbol{x_{best}}\left(N\right)\right)$
        \If{$b_{M+1} = - 1$ (improvement, $\boldsymbol{x}_{N+1} \succ \boldsymbol{x_{best}}\left(N\right)$)}
        \State Set $\boldsymbol{x_{best}}\left(N + 1\right) = \boldsymbol{x}_{N+1}$
        \Else
        \State Keep $\boldsymbol{x_{best}}\left(N + 1\right) = \boldsymbol{x_{best}}\left(N\right)$
        \State Set $\delta = \delta_{\left(j + 1\right)\text{mod}{N_{cycle}}} \in \Delta_{cycle}$ (greedy $\delta$-cycling) and $j = j + 1$
        \EndIf
        \State Update the set of samples $\mathcal{X}$ and the preference information in the sets $\mathcal{B}$ and $\mathcal{S}$
        \State Set $N = N + 1$ and $M = M + 1$
        \EndFor
    \end{algorithmic}
\end{algorithm}
\begin{algorithm}[!htb]
    \setalgorithmstretch
    \setalgorithmfontsize
    \caption{Initial queries for preference-based optimization}
    \label{alg:queries_for_preference_based_optimization}
    \textbf{Input}:
    \begin{algparams}
        \item Initial set of samples $\mathcal{X}$, $\lvert\mathcal{X}\rvert = N_{init} \in \mathbb{N}, N_{init} \geq 2$, in \eqref{eq:sample_set_X}.
    \end{algparams}

    \textbf{Output}:
    \begin{algparams}
        \item Set of preferences $\mathcal{B}$ in \eqref{eq:Set_of_preferences_B};
        \item Mapping set $\mathcal{S}$ in \eqref{eq:Mapping_set_for_preferences_S};
        \item Initial best sample $\boldsymbol{x_{best}}\left(N_{init}\right)$.
    \end{algparams}
    \hrule
    \begin{algorithmic}[1]
        \State Initialize the best candidate as $\boldsymbol{x_{best}}\left(1\right) = \boldsymbol{x}_{1}$,
        $i_{best} = 1$
        \State Initialize the sets $\mathcal{B}$ and $\mathcal{S}$: $\mathcal{B} = \emptyset$ and $\mathcal{S} = \emptyset$
        \For{$i = 2$ to $\lvert\mathcal{X}\rvert = N_{init}$}
        \State Let the decision-maker express a preference between $\boldsymbol{x_{best}}\left(i-1\right)$ and $\boldsymbol{x}_{i}$, obtaining $b = \pi \left(\boldsymbol{x_{best}}\left(i-1\right), \boldsymbol{x}_i \right)$
        \State Update the sets $\mathcal{B}$ and $\mathcal{S}$: $\mathcal{B} = \mathcal{B} \cup \{b\}$ and $\mathcal{S} = \mathcal{S} \cup \{(i_{best}, i)\}$
        \If{$b = 1$ (i.e. $ \boldsymbol{x}_i \succ \boldsymbol{x_{best}}\left(i-1\right)$)}
        \State Update the best candidate, $\boldsymbol{x_{best}}\left(i\right) = \boldsymbol{x}_{i}$ and $i_{best} = i$
        \Else
        \State Keep the best candidate unaltered, $\boldsymbol{x_{best}}\left(i\right) = \boldsymbol{x_{best}}\left(i-1\right)$
        \EndIf
        \EndFor
    \end{algorithmic}
\end{algorithm}
\subsection{Global convergence of \GLISprmethod{}}
\label{subsec:Convergence}
Whenever we are dealing with any global optimization algorithm, it is possible to guarantee its convergence to the global minimizer of Problem \eqref{eq:preference-based_optimization_problem} by checking if the conditions of the following Theorem hold.
\begin{theorem}[Convergence of a global optimization algorithm \cite{torn1989global}]
    \label{theo:convergence_theorem_torn}
    Let $\Omega \subset \mathbb{R}^{n}$ be a compact set. An algorithm converges to the global minimum of every continuous function $f:\mathbb{R}^n \to \mathbb{R}$ over $\Omega$ if and only if its sequence of iterates,
    \begin{equation*}
        \langle\boldsymbol{x}_i\rangle_{i \geq 1} = \langle\boldsymbol{x}_1, \boldsymbol{x}_2, \ldots \rangle,
    \end{equation*}
    is dense in $\Omega$.
\end{theorem}
In what follows, for the sake of clarity, we denote:
\begin{itemize}
    \item $\mathcal{X}_\infty$ as the set containing all the elements of $\langle\boldsymbol{x}_i\rangle_{i \geq 1}$ (infinite sequence),
    \item $\mathcal{X}_k \subseteq \mathcal{X}_\infty$ as the set containing all the elements of $\langle\boldsymbol{x}_i\rangle_{i = 1}^k$, which is a subsequence of $\langle\boldsymbol{x}_i\rangle_{i \geq 1}$ composed of its first $k \in \mathbb{N}$ entries.
\end{itemize}
To prove the convergence of \GLISprmethod{}, we also make use of the following Theorem, which gives us a sufficient condition that ensures the denseness of the sequence of iterates produced by any global optimization algorithm.
\begin{theorem}[A sufficient condition for the denseness of $\mathcal{X}_\infty$ \cite{regis2005constrained}]
    \label{theo:convergence_theorem_cors}
    Let $\Omega$ be a compact subset of $\mathbb{R}^n$ and let $\langle\boldsymbol{x}_{i}\rangle_{i\geq1}$ be the sequence of iterates generated by an algorithm $\texttt{A}$ (when run indefinitely). Suppose that there exists a strictly increasing sequence of positive integers $\langle i_{t}\rangle_{t\geq1}, i_{t}\in\mathbb{N},$ such that $\langle\boldsymbol{x}_{i}\rangle_{i\geq1}$ satisfies the following condition for some $\alpha\in(0,1]$:
    \begin{equation}
        \label{eq:convergence_theorem_CORS_condition}
        \min{1\leq i\leq i_{t}-1} \euclideannorm{\boldsymbol{x}_{i_{t}}-\boldsymbol{x}_{i}} \geq\alpha\cdot d_{\Omega}\left(\mathcal{X}_{i_t - 1}\right), \quad \forall t \in \mathbb{N},
    \end{equation}
    where:
    \begin{equation}
        \label{eq:convergence_theorem_CORS_d_Omega}
        d_{\Omega}\left(\mathcal{X}_{i_t - 1}\right) =\max{\boldsymbol{x}\in\Omega}\min{1\leq i\leq i_{t}-1}\euclideannorm{\boldsymbol{x}-\boldsymbol{x}_{i}}.
    \end{equation}
    Then, $\mathcal{X}_{\infty}$ generated by $\texttt{A}$ is dense in $\Omega$. 
\end{theorem}
The aforementioned Theorem has been used to prove the global convergence of \CORSmethod{}, a black-box optimization procedure. Clearly, if $\Omega$ is compact and \eqref{eq:convergence_theorem_CORS_condition} holds for some $\alpha\in(0,1]$ (making $\mathcal{X}_{\infty}$ dense in $\Omega$) then, due to Theorem \ref{theo:convergence_theorem_torn}, algorithm $\texttt{A}$ converges to the global minimum of every continuous function $f\left(\boldsymbol{x}\right)$ over $\Omega$. For what concerns the preference-based framework, we need to ensure that the scoring function $f\left(\boldsymbol{x}\right)$, which represents the preference relation $\succsim$ on $\Omega$, is continuous. Theorem \ref{theo:debreu_utility_representation} gives us necessary conditions on $\succsim$ to achieve such property. Furthermore, Proposition \ref{prop:existence_of_preference_relation_maximum} can be used to check the existence of a $\succsim$-maximum of $\Omega$. The next Theorem addresses the global convergence of \GLISprmethod{} (Algorithm \ref{alg:GLISp-r_algorithm}).
\begin{theorem}[Convergence of \GLISprmethod{}]
    \label{theo:convergence_of_GLISp-r}
    Let $\Omega \subset \mathbb{R}^n$ be a compact set and $\succsim$ be a continuous preference relation on $\Omega$ of a rational (as in Definition \ref{def:rational_decision_maker}) human decision-maker. Then, provided that $\exists \delta_j \in \Delta_{cycle}$ in \eqref{eq:Delta_cycle} such that $\delta_j = 0$ and $N_{max} \to \infty$, \GLISprmethod{} converges to the global minimum of Problem \eqref{eq:preference-based_optimization_problem} for any choice of its remaining hyper-parameters\footnote{Formally, we should also ensure that the surrogate model $\hat{f}_N\left(\boldsymbol{x}\right)$ in \eqref{eq:RBF_surrogate_model} is continuous. However, that is the case for any of the radial basis functions reported in Section \ref{subsec:Surrogate_model}.}.
\end{theorem}
\begin{proof}
    Compactness of $\Omega$, continuity of $\succsim$ on $\Omega$ and rationality of the decision-maker are conditions that ensure the existence of a solution for Problem \eqref{eq:preference-based_optimization_problem} (cf. Theorem \ref{theo:debreu_utility_representation} and Proposition \ref{prop:existence_of_preference_relation_maximum}).

    Consider the sequence of iterates $\langle\boldsymbol{x}_i\rangle_{i \geq 1}$ produced by Algorithm \ref{alg:GLISp-r_algorithm}. The first $N_{init} \in \mathbb{N}, N_{init} \geq 2,$ elements of $\langle\boldsymbol{x}_i\rangle_{i \geq 1}$ are obtained by the \LHD{} \cite{mckay2000comparison}. Instead, each $\boldsymbol{x}_i \in \mathcal{X}_{\infty}, i > N_{init},$ is selected as the solution of Problem \eqref{eq:Next_sample_search_no_black-box_constraints} with $a_N\left(\boldsymbol{x}\right)$ in \eqref{eq:Acquisition_function_GLISp-r}.

    Now, suppose that $\delta$ in \eqref{eq:Acquisition_function_GLISp-r} is cycled regardless of the improvement that the new candidate samples might bring (non-greedy cycling). Denote the exploration-exploitation trade-off weight at iteration $k$ of Algorithm \ref{alg:GLISp-r_algorithm} as $\delta\left(k\right)$ and assume that $\delta(k) = \delta_j \in \Delta_{cycle}$. Then, the cycling is performed as:
    \begin{equation}
        \label{eq:GLISp_r_proof_1}
        \delta \left(k + 1\right) = \delta_{\left(j + 1\right)\text{mod}{N_{cycle}}}, \quad \forall k \in \mathbb{N},
    \end{equation}
    instead of \eqref{eq:greedy_delta_cycling}. Without loss of generality, suppose that $\Delta_{cycle}$ in \eqref{eq:Delta_cycle} is defined as:
    \begin{equation*}
        \delta_j \neq 0, \forall j = 0, \ldots, N_{cycle} - 2, \text{ and } \delta_{N_{cycle} - 1} = 0.
    \end{equation*}
    Then, every $N_{cycle}$ iterations Algorithm \ref{alg:GLISp-r_algorithm} looks for a new candidate sample by minimizing the (min-max rescaled) \IDW{} distance function in \eqref{eq:Inverse_distance_weighting_distance_function}, regardless of the surrogate model in \eqref{eq:RBF_surrogate_model}, see $a_N\left( \boldsymbol{x}\right)$ in \eqref{eq:Acquisition_function_GLISp-r} and Problem \eqref{eq:Next_sample_search_no_black-box_constraints}. In practice, minimizing $\bar{z}_N\left(\boldsymbol{x}; \mathcal{X}_{aug}\right)$ over $\Omega$ is equivalent to solving $\argmin{\boldsymbol{x} \in \Omega} z_N\left(\boldsymbol{x}\right)$ since scaling and shifting the \IDW{} distance function does not change its minimizers \cite{nocedal1999numerical}. We define the following strictly increasing sequence of positive integers:
    \begin{equation}
        \label{eq:GLISp_r_proof_2}
        \langle i_{t'} \rangle_{t' \geq 1} =  \langle N_{init} + t' \cdot N_{cycle}\rangle_{t' \geq 1},
    \end{equation}
    which is such that:
    \begin{align}
        \label{eq:GLISp_r_proof_4}
        \boldsymbol{x}_{i_{t'}} & = \argmin{\boldsymbol{x}} z_{i_{t'} - 1}\left(\boldsymbol{x}\right), \quad \forall t' \in \mathbb{N} \\
        \text{s.t.}             & \quad\boldsymbol{x}\in\Omega.\nonumber
    \end{align}
    Now, recall from Proposition \ref{prop:global_maximizers_of_IDW_distance} that each $\boldsymbol{x}_i \in \mathcal{X}_{i_{t'} - 1}$ is  a global maximizer of Problem \eqref{eq:GLISp_r_proof_4}. Furthermore, $z_{i_{t'} - 1}\left(\boldsymbol{x}\right)$ in \eqref{eq:Inverse_distance_weighting_distance_function} is differentiable everywhere and thus continuous (see Proposition \ref{prop:Differentiability_of_IDW_distance_function}). Then, by the Extreme Value Theorem \cite{audet2017derivative}, Problem \eqref{eq:GLISp_r_proof_4} admits at least a solution. Hence, we can conclude that:
    \begin{equation}
        \label{eq:GLISp_r_proof_5}
        \boldsymbol{x}_{i_{t'}} \notin \mathcal{X}_{i_{t'} - 1} \implies \exists \fixed{\alpha}' \in \mathbb{R}_{>0} \text{ such that } \min{1\leq i\leq i_{t'}-1} \euclideannorm{\boldsymbol{x}_{i_{t'}}-\boldsymbol{x}_{i}} \geq \fixed{\alpha}', \quad \forall t' \in \mathbb{N}.
    \end{equation}
    Clearly, due to how new candidate samples are sought (i.e. by minimizing some acquisition function over $\Omega$), we have that (recall \eqref{eq:convergence_theorem_CORS_d_Omega}):
    \begin{equation}
        \label{eq:convergence_theorem_CORS_condition_extended}
        \alpha'\cdot d_{\Omega}\left(\mathcal{X}_{i_{t'} - 1}\right) \leq \min{1\leq i\leq i_{t'}-1} \euclideannorm{\boldsymbol{x}_{i_{t'}}-\boldsymbol{x}_{i}} \leq d_{\Omega}\left(\mathcal{X}_{i_{t'} - 1}\right), \quad \forall t' \in \mathbb{N},
    \end{equation}
    for some $\alpha' \in (0, 1]$. Then, by combining \eqref{eq:GLISp_r_proof_5} and \eqref{eq:convergence_theorem_CORS_condition_extended}, we get:
    \begin{equation}
        \label{eq:GLISp_r_proof_7}
        0 < \fixed{\alpha}' \leq d_{\Omega}\left(\mathcal{X}_{i_{t'} - 1}\right), \quad \forall t' \in \mathbb{N}.
    \end{equation}
    Therefore, $\exists \alpha' \in (0, 1]$ such that $\fixed{\alpha}' = \alpha' \cdot d_{\Omega}\left(\mathcal{X}_{i_{t'} - 1}\right)$ which satisfies the condition \eqref{eq:convergence_theorem_CORS_condition} of Theorem \ref{theo:convergence_theorem_cors}, $\forall t' \in \mathbb{N}$. Thus,  Algorithm \ref{alg:GLISp-r_algorithm} with $\delta$ cycled as in \eqref{eq:GLISp_r_proof_1} 
    produces a sequence of iterates that is dense in $\Omega$.

    Next, consider the greedy $\delta$-cycling strategy proposed in Section \ref{subsubsec:Greedy_delta_cycling} and for an arbitrary choice of $\Delta_{cycle}$ in \eqref{eq:Delta_cycle}. Let us examine the case in \eqref{eq:greedy_delta_cycling} when $\delta$ is kept unchanged from an iteration of Algorithm \ref{alg:GLISp-r_algorithm} to the other. Denote as $\langle i_{t''} \rangle_{t'' = 1}^{t_{max}''}, t_{max}'' \in \mathbb{N},$ the sequence of indexes of those samples that improve upon the current best candidate, resulting in no change in the exploration-exploitation trade-off weight. We have that:
    \begin{equation*}
        \label{eq:GLISp_r_proof_8}
        \boldsymbol{x}_{i_{t''}} \succ \boldsymbol{x_{best}}\left(i_{t''} - 1\right), \quad \forall t'': 1 \leq t'' \leq t_{max}''.
    \end{equation*}
    Clearly, $\boldsymbol{x}_{i_{t''}} \notin \mathcal{X}_{i_{t''} - 1}$ since it is strictly preferred to all the other samples in $\mathcal{X}_{i_{t'' - 1}}$.
    Thus, we could define a positive constant $\fixed{\alpha}'' \in \mathbb{R}_{> 0}$ analogously to \eqref{eq:GLISp_r_proof_5}: 
    \begin{equation}
        \begin{split}
            \label{eq:GLISp_r_proof_9}
            \boldsymbol{x}_{i_{t''}} \notin \mathcal{X}_{i_{t''} - 1} \implies &\exists \fixed{\alpha}'' \in \mathbb{R}_{>0} \text{ such that } \min{1\leq i\leq i_{t''}-1} \euclideannorm{\boldsymbol{x}_{i_{t''}}-\boldsymbol{x}_{i}} \geq \fixed{\alpha}'',\\
            &\forall t'': 1 \leq t'', \leq t_{max}''.
        \end{split}
    \end{equation}
    Finally, let us consider the greedy $\delta$-cycling strategy in \eqref{eq:greedy_delta_cycling} as whole and assume, as in Theorem \ref{theo:convergence_of_GLISp-r}, that $\exists \delta_j \in \Delta_{cycle}$ in \eqref{eq:Delta_cycle} such that $\delta_j = 0$. We can build a strictly increasing  sequence of positive integers $\langle i_{t}\rangle_{t \geq 1}$ by merging:
    \begin{itemize}
        \item The elements of the sequence $\langle i_{t''} \rangle_{t'' = 1}^{t_{max}''}$, which are the indexes of those samples that improve upon the best candidates $\boldsymbol{x_{best}}\left(i_{t''} - 1\right), \forall t'': 1 \leq t'' \leq t_{max}''$;
        \item The elements of the sequence $\langle i_{t'} \rangle_{t' \geq 1}$, which constitute the indexes of those samples found by solving the pure exploration problem in \eqref{eq:GLISp_r_proof_4}. Note that, unless Algorithm \ref{alg:GLISp-r_algorithm} always improves upon its current best candidate (in which case $\langle i_{t''} \rangle_{t'' = 1}^{t_{max}''}$ is actually infinite and hence $\mathcal{X}_{\infty}$ is dense in $\Omega$), Problem \eqref{eq:Next_sample_search_no_black-box_constraints} with $a_N\left(\boldsymbol{x}\right)$ in \eqref{eq:Acquisition_function_GLISp-r} is solved using $\delta = 0$ infinitely often, although not necessarily every $N_{cycle}$ iterations as in \eqref{eq:GLISp_r_proof_2}.
    \end{itemize}
    Hence, we can select a positive constant $\fixed{\alpha} \in \mathbb{R}_{>0}$ as $\fixed{\alpha} = \min{} \left\{\fixed{\alpha}', \fixed{\alpha}''\right\}$ which, analogously to \eqref{eq:GLISp_r_proof_7}, is such that:
    \begin{equation}
        \label{eq:GLISp_r_proof_10}
        0 < \fixed{\alpha} \leq d_{\Omega}\left(\mathcal{X}_{i_{t} - 1}\right), \quad \forall t \in \mathbb{N}.
    \end{equation}
    Therefore, $\exists \alpha \in (0, 1]$ such that $\fixed{\alpha} = \alpha \cdot d_{\Omega}\left(\mathcal{X}_{i_{t} - 1}\right)$ which satisfies the  condition \eqref{eq:convergence_theorem_CORS_condition} of Theorem \ref{theo:convergence_theorem_cors}, $\forall t \in \mathbb{N}$. Thus, \GLISprmethod{} with $\delta$ in \eqref{eq:Acquisition_function_GLISp-r} cycled following the greedy $\delta$-cycling strategy in \eqref{eq:greedy_delta_cycling} converges to the global minimum of Problem \eqref{eq:preference-based_optimization_problem}.
\end{proof}
Most preference-based response surface methods, such as the ones in \cite{Bemporad2021,brochu2007active,gonzalez2017preferential,benavoli2021preferential}, do not address their convergence to the global minimum of Problem \eqref{eq:preference-based_optimization_problem}. In this work, we have shown that, by leveraging some results from the utility theory literature (see Section \ref{sec:Problem_formulation}), we can find sufficient conditions on the preference relation of the human decision-maker ($\succsim$ on $\Omega$) that guarantee the existence of a solution for Problem \eqref{eq:preference-based_optimization_problem} and allow us to analyze the convergence of any preference-based procedure as we would in the global optimization framework.

We conclude this Section with two Remarks on Theorem \ref{theo:convergence_of_GLISp-r}. The first deals with the importance of adding a zero entry inside $\Delta_{cycle}$ in \eqref{eq:Delta_cycle}, while the second addresses the selection of the cycling set.
\begin{remark}
    The omission of a zero entry inside $\Delta_{cycle}$ in \eqref{eq:Delta_cycle} does not necessarily preclude the global convergence of Algorithm \ref{alg:GLISp-r_algorithm} on all possible preference-based optimization problems. For example, if $f\left( \boldsymbol{x}\right)$ is a constant function then, after we evaluate the first sample $\boldsymbol{x}_1$, any other point brings no improvement (i.e. we would have $\boldsymbol{x}_i \sim \boldsymbol{x}_1, \forall i > 1$). The caveat is that, if $\nexists \delta_j \in \Delta_{cycle}$ such that $\delta_j = 0$, then the sequence $\langle i_{t''} \rangle_{t'' = 1}^{t_{max}''}$ for which \eqref{eq:GLISp_r_proof_9} holds  is likely to be finite (i.e. \texttt{GLISp-r} does not improve upon its current best candidate infinitely often). Moreover, differently from Problem \eqref{eq:GLISp_r_proof_4}, we have no guarantee that the solutions of:
    \begin{align*}
        \boldsymbol{x}_{N + 1} & = \argmin{\boldsymbol{x}} a_{N}\left(\boldsymbol{x}\right), \quad \forall N: N \neq i_{t''} - 1, 1 \leq t'' \leq t_{max}'' \\
        \text{s.t.}            & \quad\boldsymbol{x}\in\Omega,\nonumber
    \end{align*}
    with $a_N\left(\boldsymbol{x}\right)$ defined as in \eqref{eq:Acquisition_function_GLISp-r} and for $\delta \neq 0$, are not already present in $\mathcal{X}_N$. 
    Therefore, we cannot ensure the existence of a strictly increasing sequence of positive integers that is infinite and for which \eqref{eq:convergence_theorem_CORS_condition} holds. Instead, the result in Theorem \ref{theo:convergence_theorem_cors} does not apply for $\langle i_{t''} \rangle_{t'' = 1}^{t_{max}''}$, since the sequence is finite. Consequently, Algorithm \ref{alg:GLISp-r_algorithm} does not necessarily produce a sequence of iterates $\langle \boldsymbol{x}_i \rangle_{i \geq 1}$ that is dense in $\Omega$, preventing its convergence on some (but not all) preference-based optimization problems.
\end{remark}
\begin{remark}
    Theorem \ref{theo:convergence_of_GLISp-r} guarantees that, under some hypotheses, \GLISprmethod{} converges to the global minimum of Problem \eqref{eq:preference-based_optimization_problem}, however it does not give any indication on its convergence rate. In particular, if $\Delta_{cycle}$ is actually $\langle 0 \rangle$, Algorithm \ref{alg:GLISp-r_algorithm} amounts to performing an exhaustive search without considering the information carried by the preferences in $\mathcal{B}$ \eqref{eq:Set_of_preferences_B} and $\mathcal{S}$ \eqref{eq:Mapping_set_for_preferences_S}, which is quite inefficient \cite{audet2017derivative}. Therefore, it is best to include some $\delta_j$'s in $\Delta_{cycle}$ that allow the surrogate model to be taken into consideration. For this reason, we suggest including terms that are well spread within the $\left[0, 1\right]$ range, including a zero entry to ensure the result in Theorem \ref{theo:convergence_of_GLISp-r}. Intuitively, the rate of convergence will be dependent on how well $\hat{f}_{N}\left(\boldsymbol{x}\right)$ in \eqref{eq:RBF_surrogate_model} approximates $f\left(\boldsymbol{x}\right)$ as well as on the choice of $\Delta_{cycle}$ in \eqref{eq:Delta_cycle}.
\end{remark}
\section{Empirical results}
\label{sec:Empirical_results}
In this Section, we compare the performances of algorithms \GLISprmethod{} and \GLISpmethod{} on a variety of benchmark bound-constrained global optimization problems taken from \cite{Bemporad2020,gramacy2012cases,jamil2013literature,mishra2006some}. Consistently with the preference-based optimization literature, we stick to benchmark problems with less than $n = 10$ decision variables \cite{Bemporad2021,brochu2007active,benavoli2021preferential,gonzalez2017preferential}. 
We also consider the revisited version of the \IDW{} distance function in \eqref{eq:Inverse_distance_weighting_distance_function} employed by \CGLISpmethod{}, which is an extension of \GLISpmethod{} proposed by the same authors\footnote{Note that, in this work, we use \CGLISpmethod{} only to test how the \IDW{} distance function in \eqref{eq:Inverse_distance_weighting_distance_function_C-GLISp} compares to the other formulation in \eqref{eq:Inverse_distance_weighting_distance_function}. However, \CGLISpmethod{} has been developed mainly to extend \GLISpmethod{} in order to handle black-box constraints. The different definition of $z_N\left(\boldsymbol{x}\right)$ is only a minor detail of such paper.}. In particular, in \CGLISpmethod{}:
\begin{align}
    \label{eq:Inverse_distance_weighting_distance_function_C-GLISp}
    z_N(\boldsymbol{x}) = & \left(\frac{N}{N_{max}} - 1\right) \cdot \arctan \left( \frac{\sum_{i = 1, i \neq i_{best}\left( N\right)}^N w_i\left(\boldsymbol{x_{best}}\left(N\right)\right)}{\sum_{i = 1}^N w_i\left(\boldsymbol{x}\right)} \right) + \\
                          & - \frac{N}{N_{max}} \cdot \arctan \left( \frac{1}{\sum_{i = 1}^N w_i\left(\boldsymbol{x}\right)} \right), \quad \forall \boldsymbol{x} \in \mathbb{R}^n \setminus \mathcal{X}, \nonumber
\end{align}
while $z_N(\boldsymbol{x}) = 0, \forall \boldsymbol{x} \in \mathcal{X}$. In \eqref{eq:Inverse_distance_weighting_distance_function_C-GLISp}, $i_{best}\left( N\right) \in \mathbb{N}, 1 \leq i_{best}\left( N\right) \leq N,$ represents the index of the best-found candidate when $\lvert \mathcal{X}\rvert = N$. 
In practice, $z_N\left(\boldsymbol{x}\right)$ in \eqref{eq:Inverse_distance_weighting_distance_function_C-GLISp} improves the exploratory capabilities of \GLISpmethod{} without the need to define an alternative acquisition function from the one in \eqref{eq:Acquisition_function_GLISp} (see \cite{zhu2021c}).

We point out that we could also consider the preferential Bayesian optimization algorithm in \cite{brochu2007active} as an additional competitor for \GLISprmethod{}. However, in \cite{Bemporad2021}, the authors show that \texttt{GLISp} outperforms the aforementioned method. As we will see in the next Sections, \GLISprmethod{} exhibits convergence speeds that are similar to those of \GLISpmethod{} and hence we have decided to omit the algorithm in \cite{brochu2007active} from our analysis. Moreover, after preliminary testing, the latter method was proven to not be on par w.r.t. the other competitors.

\subsection{Experimental setup}
All benchmark optimization problems have been solved on a machine with two Intel Xeon E5-2687W @3.00GHz CPUs and 128GB of RAM. \GLISprmethod{} has been implemented in MATLAB. Similarly, we have used the MATLAB code for \texttt{GLISp} provided in \cite{Bemporad2021} (formally, version 2.4 of the software package) and the one for \texttt{C-GLISp} supplied in \cite{zhu2021c} (version 3.0 of the same code package). For all the procedures, Problem \eqref{eq:Next_sample_search_no_black-box_constraints} has been solved using Particle Swarm Optimization (\texttt{PSWARM}). In particular, we have used its MATLAB implementation provided by \cite{vaz2007particle,vaz2009pswarm,le2012optimizing}.

To achieve a fair comparison, we have chosen the same hyper-parameters for both \GLISprmethod{} and \GLISpCGLISpmethods{}, whenever possible. This applies, for example, to the shape parameter $\epsilon$, which is initialized to $\epsilon = 1$, and the radial function $\varphi \left(\cdot\right)$, that is an inverse quadratic. Furthermore, the shape parameter for the surrogate model $\hat{f}_N\left(\boldsymbol{x}\right)$ in \eqref{eq:RBF_surrogate_model} is recalibrated using \LOOCV{} (see \GLISpmethod{}) at the iterations  $\mathcal{K}_{R} = \left\{1, 50, 100\right\}$. Its possible values are $\mathcal{E}_{\LOOCV{}} = \{0.1000,$ $0.1668,$ $0.2783,$ $0.4642,$ $0.7743,$ $1,$ $1.2915,$ $2.1544,$ $3.5938,$ $5.9948,$ $10\}$. The remaining hyper-parameters shared by \GLISpCGLISpmethods{} and \GLISprmethod{} are set to $\lambda = 10^{-6}$ and $\sigma = 10^{-2}$. Regarding \GLISprmethod{}, we have chosen $\Delta_{cycle} = \langle 0.95,0.7,0.35,0 \rangle$, where we have included a zero term to comply with the convergence result in Theorem \ref{theo:convergence_of_GLISp-r}. The reasoning behind this cycling sequence is that, after the initial sampling phase, we give priority to the surrogate as to drive the algorithm towards more promising regions of $\Omega$, for example where local minima are located. In practice, if $f\left(\boldsymbol{x}\right)$ is a function that can be approximated well by $\hat{f}_N\left(\boldsymbol{x}\right)$ with little data, starting with a $\delta$ in \eqref{eq:Acquisition_function_GLISp-r} that is close to $1$ might lead the procedure to converge quite faster. If that is not the case, the remaining terms contained inside $\Delta_{cycle}$ promote the exploration of other zones of the constraint set, either dependently or independently from $\hat{f}_N\left(\boldsymbol{x}\right)$. For the sake of completeness, we also analyze the performances of \GLISprmethod{} when equipped with two particular cycling sequences: $\Delta_{cycle} = \langle 0.95 \rangle$ (\quotes{pure} exploitation) and $\Delta_{cycle} = \langle 0 \rangle$ (pure exploration). Concerning \GLISpCGLISpmethods{}, which use $a_N\left(\boldsymbol{x}\right)$ in \eqref{eq:Acquisition_function_GLISp}, we set $\delta = 2$, as proposed by the authors in \cite{Bemporad2021}. Lastly, for \GLISprmethod{}, we select $K_{aug} = 5$ for all optimization problems since, empirically, after several preliminary experiments, it has proven to be good enough to rescale $z_N\left(\boldsymbol{x}\right)$ in \eqref{eq:Inverse_distance_weighting_distance_function} and $\hat{f}_N\left(\boldsymbol{x}\right)$ in \eqref{eq:RBF_surrogate_model} in most cases (see for example \figname{} \ref{fig:Comparison_between_acquisition_function_terms}).

We run the procedures using a fixed budget of $N_{max}=200$ samples and solve each benchmark optimization problem $N_{trial} = 100$ times, starting from $N_{init} = 4 \cdot n$  points generated by \LHD{}s \cite{mckay2000comparison} with different random seeds. To ensure fairness of comparison, all the algorithms are started from the same samples.

For the sake of clarity, we point out that the preferences between the couples of samples are expressed using the preference function in \eqref{eq:preference_function}, where $f\left(\boldsymbol{x}\right)$ is the corresponding cost function for the considered benchmark optimization problem. Therefore, the preferences are always expressed consistently, allowing us to compare \GLISprmethod{} and \GLISpCGLISpmethods{} properly. That would not necessarily be the case if a human decision-maker was involved.

\subsection{Results}
We compare the performances of \GLISprmethod{} and \GLISpCGLISpmethods{} on each benchmark optimization problem by means of \important{convergence plots} and \important{data profiles} \cite{audet2017derivative}. Convergence plots depict the median, best and worst case performances over the $N_{trial}$ instances and with respect to the cost function values achieved by $\boldsymbol{x_{best}}\left(N\right)$, as $N$ increases. Data profiles are one of the most popular tools for assessing efficiency and robustness of global optimization methods: efficient surrogate-based methods exhibit steep slopes (i.e. fast convergences speeds), while robust algorithms are able to solve more (instances of) problems within the budget $N_{max}$. In general, no method is both efficient and robust: a trade-off must be made \cite{thanedar1990robustness}. Typically, data profiles are used to visualize the performances of several algorithms on multiple benchmark optimization problems simultaneously. However, due to the stochastic nature of \LHD{}s \cite{mckay2000comparison}, here we will also depict the data profiles for \GLISprmethod{}, \GLISpmethod{} and \CGLISpmethod{} on each benchmark optimization problem on its own, highlighting how the algorithms behave when started from different samples. In practice, data profiles show, for $1 \leq N \leq N_{max}$, how many among the $N_{trial}$ instances of one (or several) benchmark optimization problem(s) have been solved to a prescribed \important{accuracy}, defined as \cite{audet2017derivative}:
\begin{equation}
    \label{eq:accuracy}
    acc\left(N\right) = \frac{f\left(\boldsymbol{x_{best}}\left(N\right)\right)-f\left(\boldsymbol{x}_{1}\right)}{f^*-f\left(\boldsymbol{x}_{1}\right)},
\end{equation}
where $f^* = \min{\boldsymbol{x} \in \Omega} f\left(\boldsymbol{x}\right)$. In particular, here we consider a benchmark optimization problem to be solved by some algorithm when $acc\left(N\right) > t, t = 0.95$. In what follows, the results achieved by \GLISprmethod{} equipped with $\Delta_{cycle} = \langle 0.95 \rangle$ and $\Delta_{cycle} = \langle 0 \rangle$ are reported only in the data profiles (and not in the convergence plots), to make the graphs more readable.

\figname{} \ref{fig:cumulative_data_profiles} shows the cumulative data profiles of \GLISprmethod{}, \GLISpmethod{} and \CGLISpmethod{}, which result from considering all the benchmark optimization problems simultaneously. Instead, \figname{} \ref{fig:performances_on_benchmarks_1} and \figname{} \ref{fig:performances_on_benchmarks_2} depict the convergence plots and the data profiles achieved by the considered algorithms on each benchmark. In \tablename{} \ref{tab:algorithm_comparison} we report the number of samples required to reach the accuracy $t = 0.95$, i.e.:
\begin{equation}
    \label{eq:number_of_samples_for_relative_accuracy}
    N_{acc > t} = \min{1 \leq N \leq N_{max}} N \quad \text{ such that } acc\left(N\right) > t.
\end{equation}
In practice, given that each benchmark optimization problem is solved multiple times, $N_{acc > t}$ in \tablename{} \ref{tab:algorithm_comparison} is assessed median-wise (over the $N_{trial}$ instances), giving us an indication on the efficiency of each method. In the same Table, we also report the percentages of instances of problems solved by \GLISprmethod{}, \GLISpmethod{} and \CGLISpmethod{} (which is an indicator of robustness) and the average execution times of each algorithm.

From our experiments, we gather that:
\begin{itemize}
    \item \important{\GLISprmethod{} ($\Delta_{cycle} = \langle 0.95, 0.7, 0.35, 0\rangle$) can be notably more robust than \GLISpmethod{} without excessively compromising its convergence speed}. On several occasions, the latter algorithm gets stuck on local minima of the benchmark optimization problems. That is particularly evident on the \bemporadGOP{} and \gramacyandleeGOP{} benchmarks, in which cases \GLISpmethod{} solves, respectively, only $70\%$ and $31\%$ of the $N_{trial}$ instances. Instead, \GLISprmethod{} ($\Delta_{cycle} = \langle 0.95, 0.7, 0.35, 0\rangle$) is able to solve both of them to the prescribed accuracy. 
          In practice, \GLISpmethod{} shines when exploitation is better suited for the benchmark optimization problem at hand. That is particularly relevant for the \bukinsixGOP{} problem, on which both \GLISpmethod{} and \GLISprmethod{} with $\Delta_{cycle} = \langle 0.95 \rangle$  (\quotes{pure} exploitation) perform quite well. Lastly, \CGLISpmethod{} is as robust as \GLISprmethod{} ($\Delta_{cycle} = \langle 0.95, 0.7, 0.35, 0\rangle$) but is the least efficient among the analyzed procedures.
    \item The pure exploration strategy (i.e. \GLISprmethod{} with $\Delta_{cycle} = \langle 0\rangle$) performs poorly, even for $n = 2$. When $n = 5$, it is unable solve any problem (in fact, the data profiles stay flat after the initial sampling phase). Only for $n = 1$ the pure exploration strategy is quite robust and relatively efficient.
    \item Vice-versa, a \quotes{pure} exploitatory approach (i.e. \GLISprmethod{} with $\Delta_{cycle} = \langle 0.95\rangle$), although not necessarily globally convergent (see Theorem \ref{theo:convergence_of_GLISp-r}), can actually be successful on some benchmark optimization problems. Often, such strategy exhibits a slightly lower $N_{acc > t}$ (median-wise) than the other procedures, see Table \ref{tab:algorithm_comparison}. Notably, the data profiles of \GLISprmethod{} ($\Delta_{cycle} = \langle 0.95\rangle$) can be quite similar to those of \GLISpmethod{}. Therefore, we could say that \GLISpmethod{} has limited exploratory capabilities. 
    \item The main disadvantage of \GLISprmethod{}, compared to \GLISpmethod{} and \CGLISpmethod{}, is the increased computational time, as reported in Table \ref{tab:algorithm_comparison}. That is due to the computational overhead of Algorithm \ref{alg:Augmented_sample_set}, which generates the augmented sample set $\mathcal{X}_{aug}$ for the proposed procedure by performing $K$-means clustering. On average, \GLISprmethod{} ($\Delta_{cycle} = \langle 0.95, 0.7, 0.35, 0\rangle$) is $31\%$ slower than \GLISpmethod{} and $72\%$ slower than \CGLISpmethod{}. However, we point out that these overheads are practically negligible when the queries to the decision-maker involve running computer simulations or performing real-world experiments which, contrary to the considered benchmark optimization problems, can take from a few minutes up to several hours. This is a common  assumption made by surrogate-based methods: at each iteration, the most time-consuming operation is the query to the decision-maker (or, in the context of black-box optimization, the measure of the cost function \cite{vu2017surrogate,jones2001taxonomy}).
\end{itemize}

\begin{figure}[!htb]
    \centering
    \includegraphics[width=0.8\textwidth]{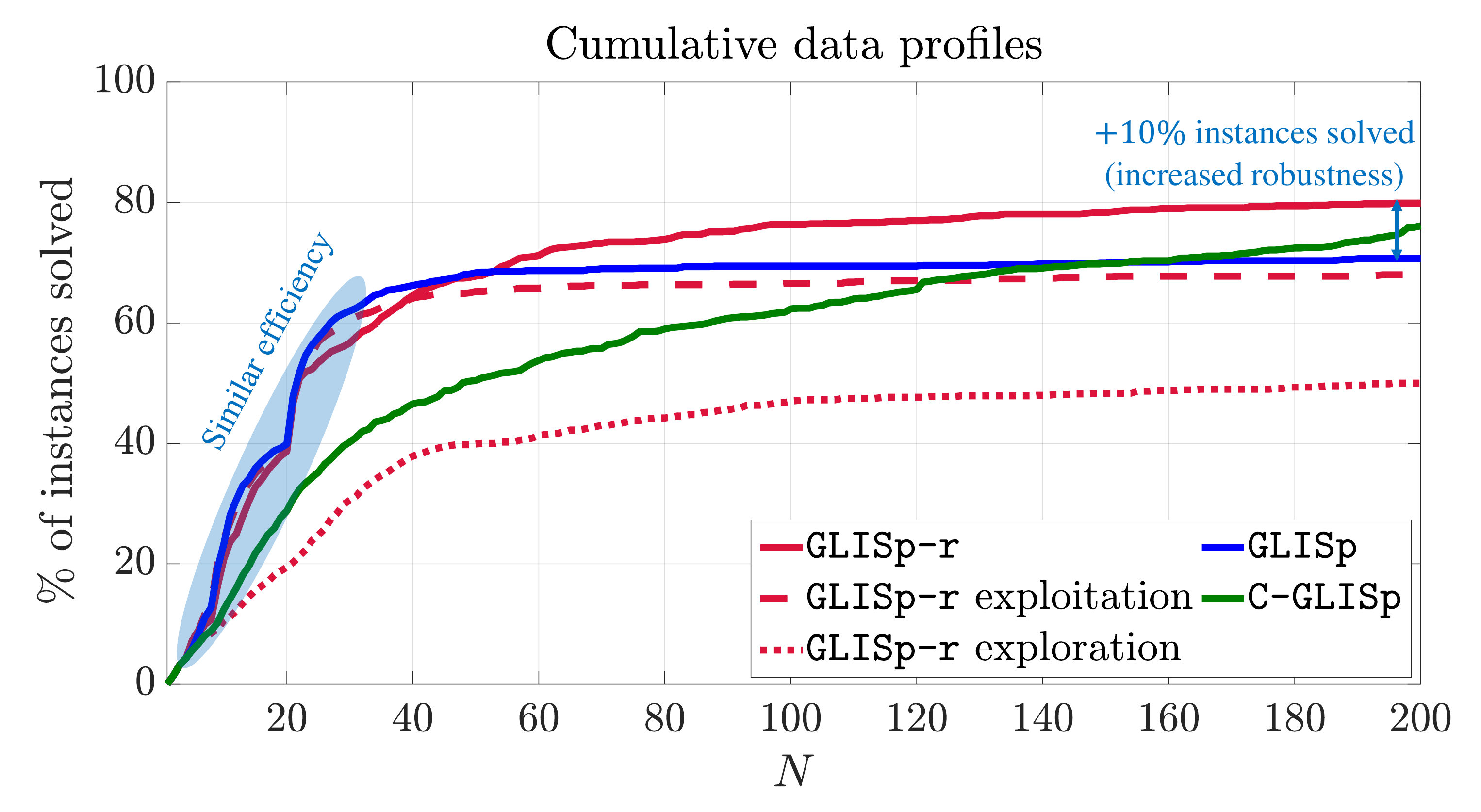}

    \caption{
        \label{fig:cumulative_data_profiles}
        Cumulative data profiles ($acc\left(N\right) > 95\%$) 
        of the considered preference-based optimization algorithms. \GLISprmethod{} is depicted in red ($\Delta_{cycle} = \langle 0.95, 0.7, 0.35, 0\rangle$ continuous line, $\Delta_{cycle} = \langle 0.95 \rangle$ dashed line, $\Delta_{cycle} = \langle 0 \rangle$ dotted line), \GLISpmethod{} in blue and \CGLISpmethod{} in green.
    }
\end{figure}

\section{Conclusions}
\label{sec:Discussion}
In this work, we introduced the preference-based optimization problem from a utility theory perspective. Then, we extended algorithm \GLISpmethod{}, giving rise \GLISprmethod{}. In particular, we have addressed the shortcomings of $a_N\left(\boldsymbol{x}\right)$ in \eqref{eq:Acquisition_function_GLISp}, defined a new acquisition function and proposed the greedy $\delta$-cycling strategy, which dynamically varies the exploration-exploitation weight $\delta$ in \eqref{eq:Acquisition_function_GLISp-r}. 
Furthermore, we have proven the global convergence of \GLISprmethod{}, which is strictly related to the choice of the cycling sequence $\Delta_{cycle}$. To the best of our knowledge, \GLISprmethod{} is the first preference-based surrogate-based method with a formal proof of convergence.

Compared to the original method, the proposed extension is less likely to get stuck on local minima of the scoring function and proves to be more robust on several benchmark optimization problems without particularly compromising its convergence speed. Moreover, we have also considered algorithm \CGLISpmethod{}, which 
improves the exploratory capabilities of \GLISpmethod{} by employing a different exploration function. In our experiments, we have observed that, even though \CGLISpmethod{} is as robust as \GLISprmethod{}, it often exhibits slower convergence rates compared to \GLISpmethod{} and the proposed extension.

Further research is devoted to extending \GLISprmethod{} in order to handle black-box constraints, which involve functions whose analytical formulations are not available. One possibility is to follow the same reasoning behind \CGLISpmethod{}, which does so by adding a term to the acquisition function that penalizes the exploration in those regions of $\Omega$ where the black-box constraints are likely to be violated. Additionally, a compelling line of research concerns preference-based optimization problems in the case where the human decision-maker is \quotes{irrational} (in a sense that some of the properties in Definition \ref{def:rational_decision_maker}, typically completeness, do not hold). From a practical perspective, when the preference relation $\succsim$ on $\Omega$ of the \DM{} is not complete, we would need a surrogate model that is able to handle the answer \quotes{I do not know} when the decision-maker cannot decide which, among two calibrations, he/she prefers.

\paragraph*{Data availability} The benchmark optimization problems considered in Section \ref{sec:Empirical_results} are reported in \cite{Bemporad2020,gramacy2012cases,jamil2013literature,mishra2006some}. The MATLAB code for algorithms \GLISpmethod{} and \CGLISpmethod{} is provided in the corresponding papers. The MATLAB code for the proposed method, \GLISprmethod{}, is supplied in the supplementary material provided in \cite{previtali2023glisp}.
\subsubsection*{Statements and declarations}
\paragraph*{Conflict of interest} All the authors declare that they have no conflict of interest.

\bibliographystyle{plain}
\bibliography{references_books.bib, references_thesis.bib, references_GO.bib, references_BBO.bib, references_PBO.bib, references_nonspecific.bib, references_our_contributions.bib, references_case_study.bib, reference_real_paper.bib}.

\begin{figure}[!htb]
    \centering
    \subfloat{
        \centering
        \includegraphics[width=.5\textwidth]{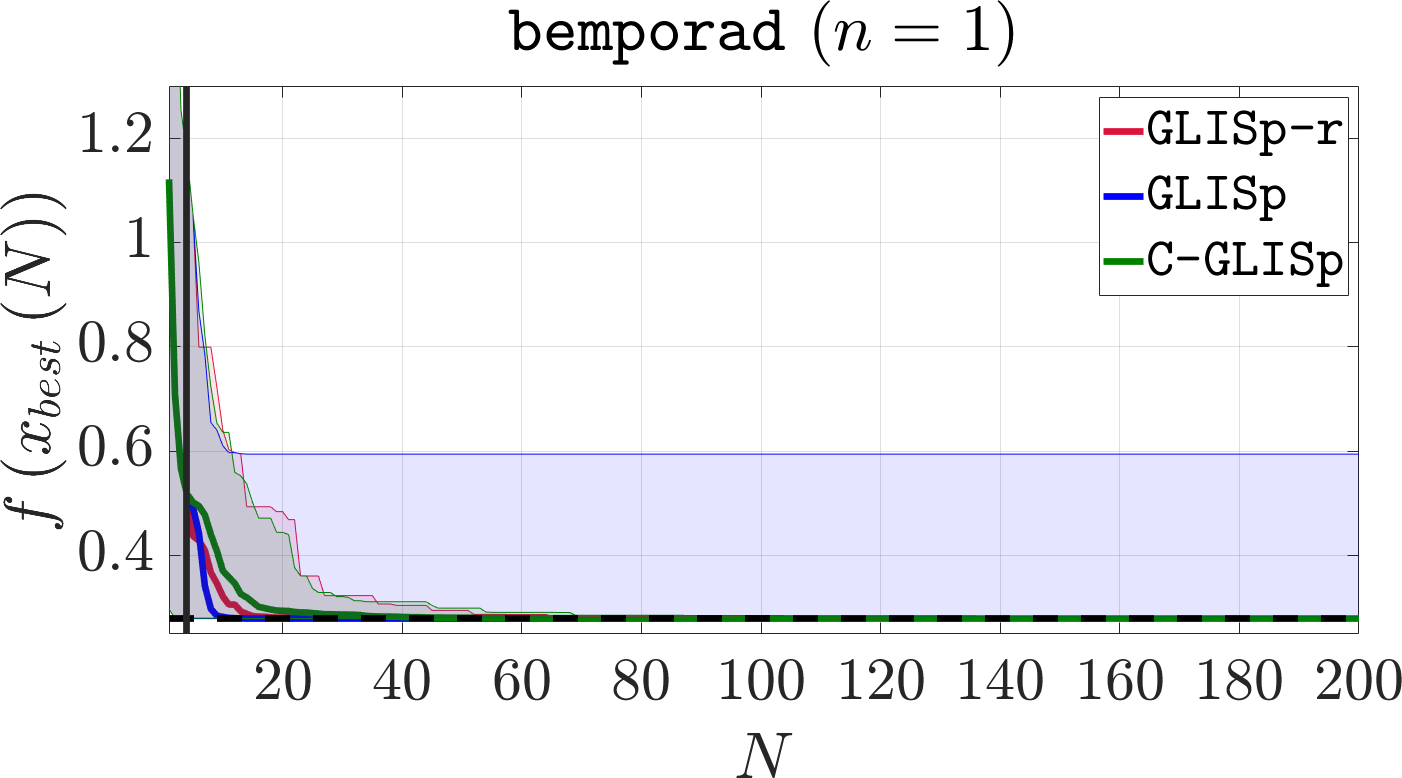}
    }
    \subfloat{
        \centering
        \includegraphics[width=.5\textwidth]{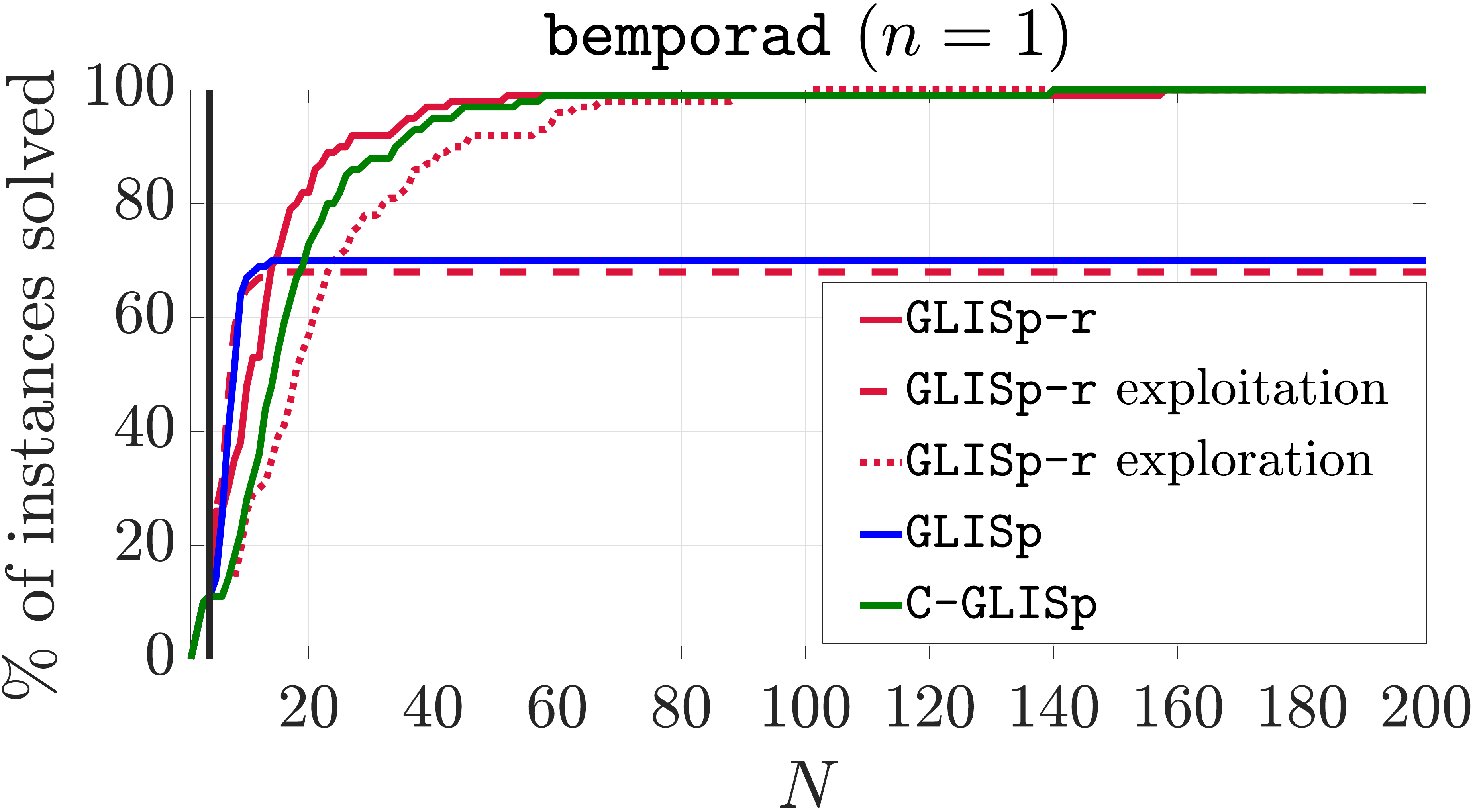}
    }

    \subfloat{
        \centering
        \includegraphics[width=.5\textwidth]{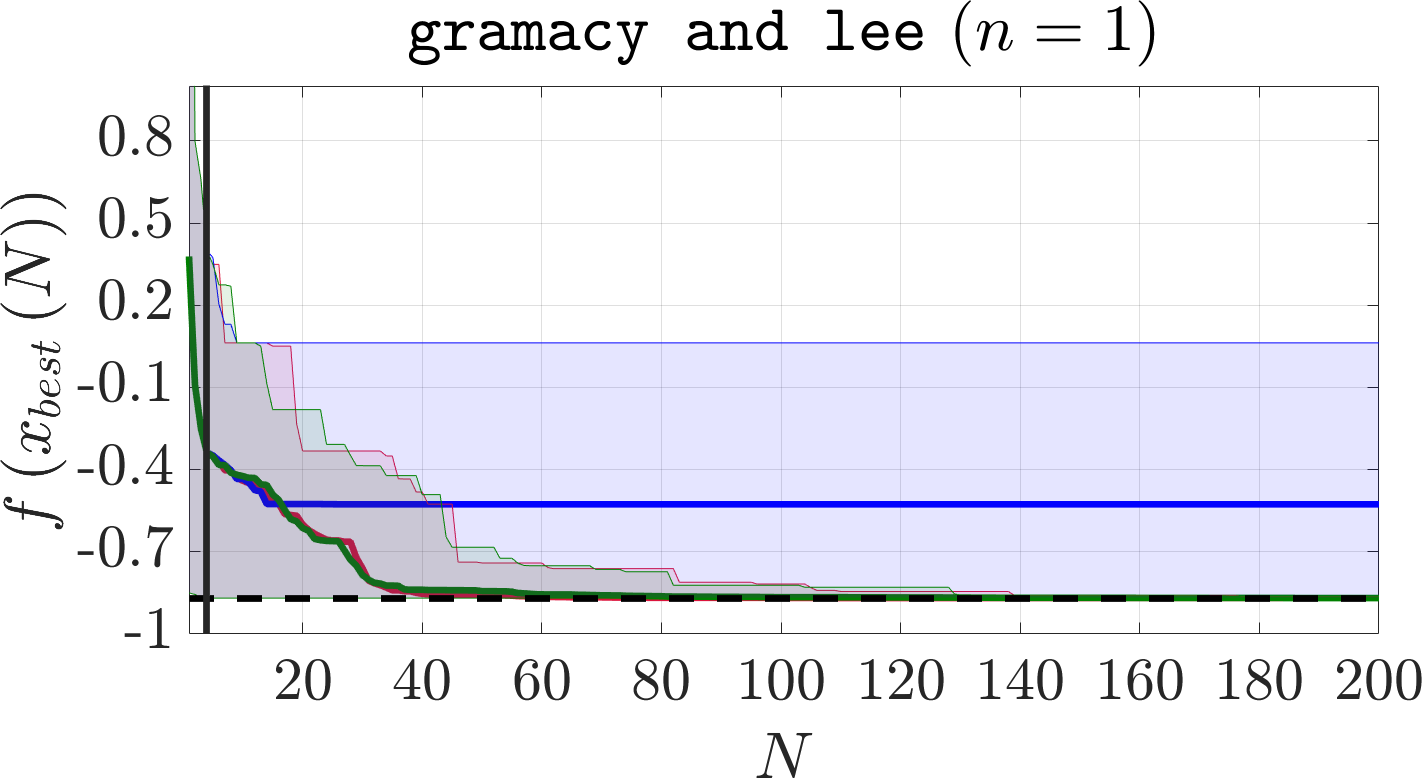}
    }
    \subfloat{
        \centering
        \includegraphics[width=.5\textwidth]{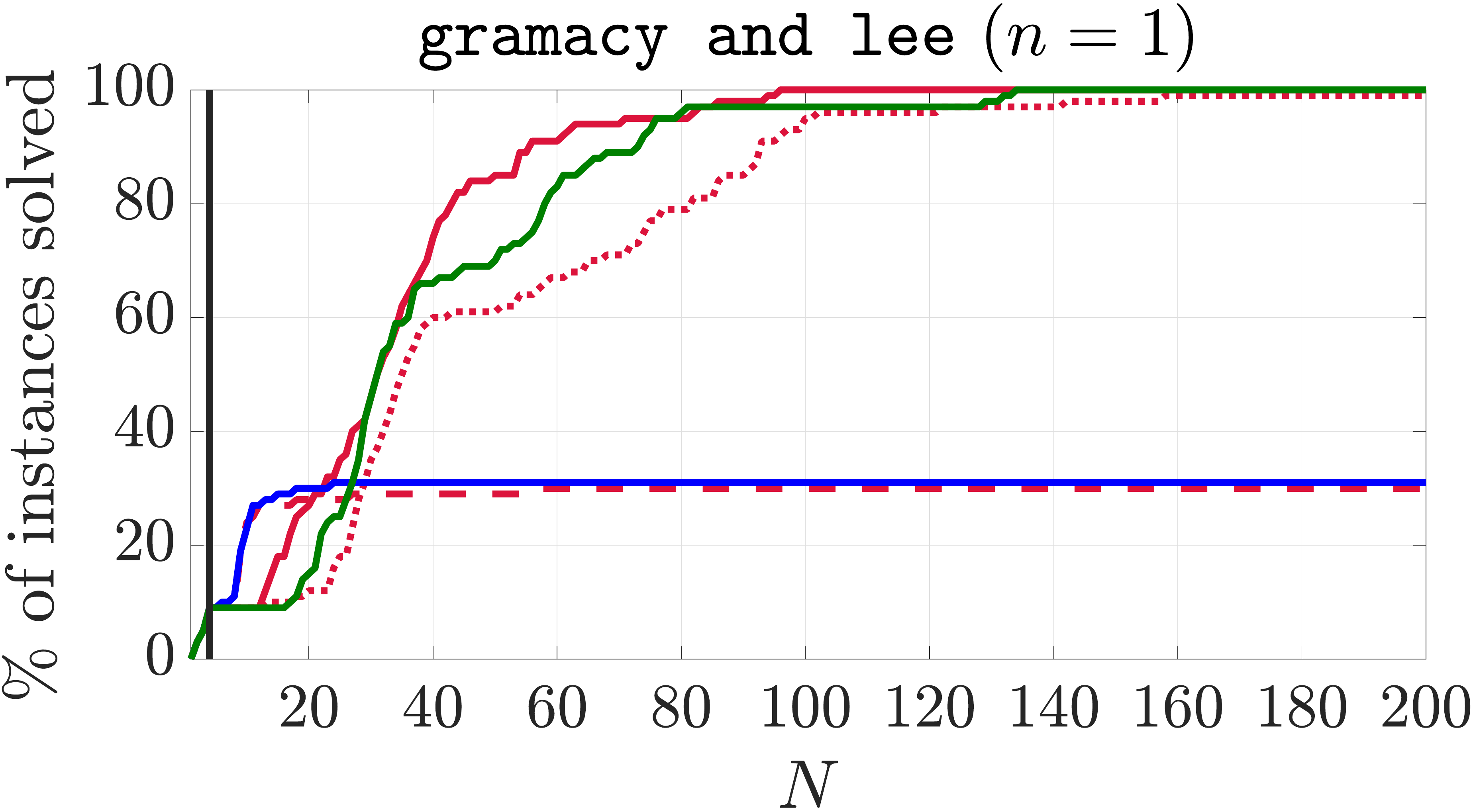}
    }

    \subfloat{
        \centering
        \includegraphics[width=.5\textwidth]{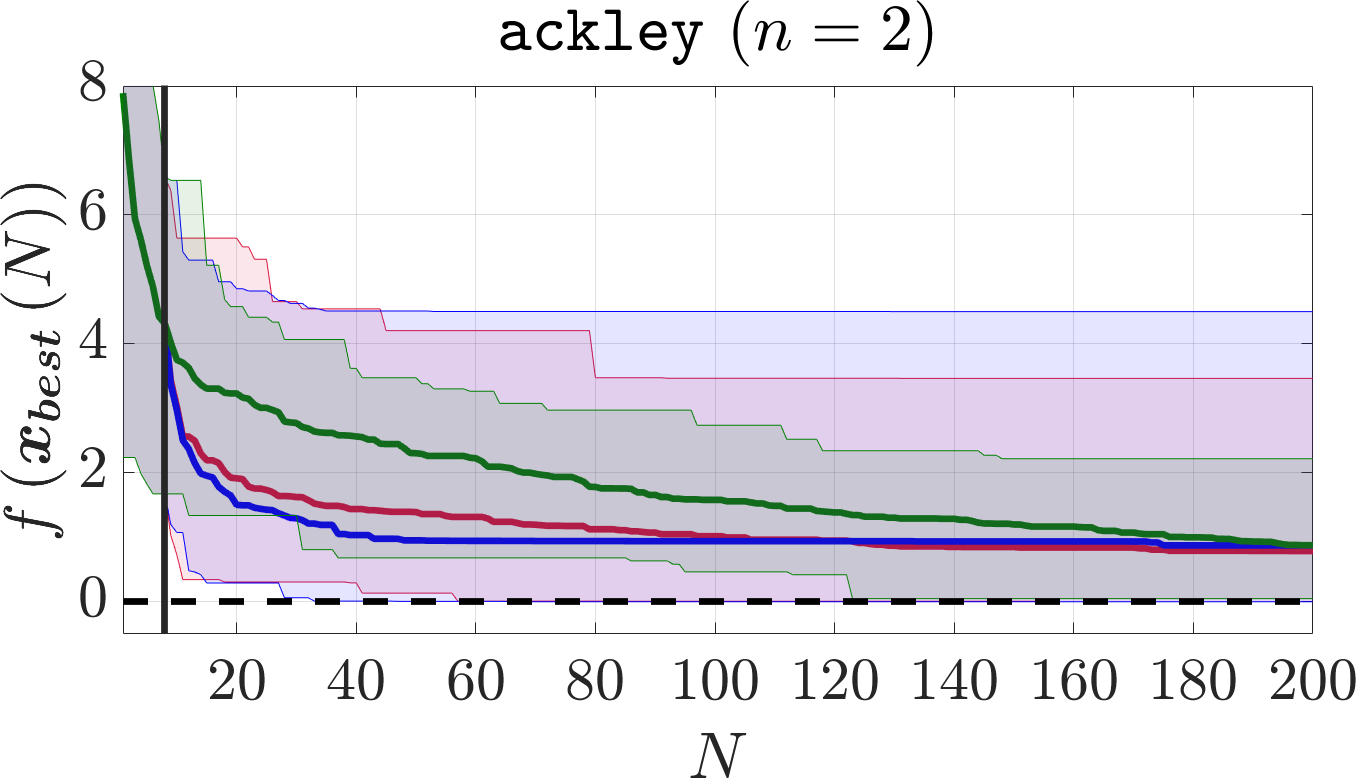}
    }
    \subfloat{
        \centering
        \includegraphics[width=.5\textwidth]{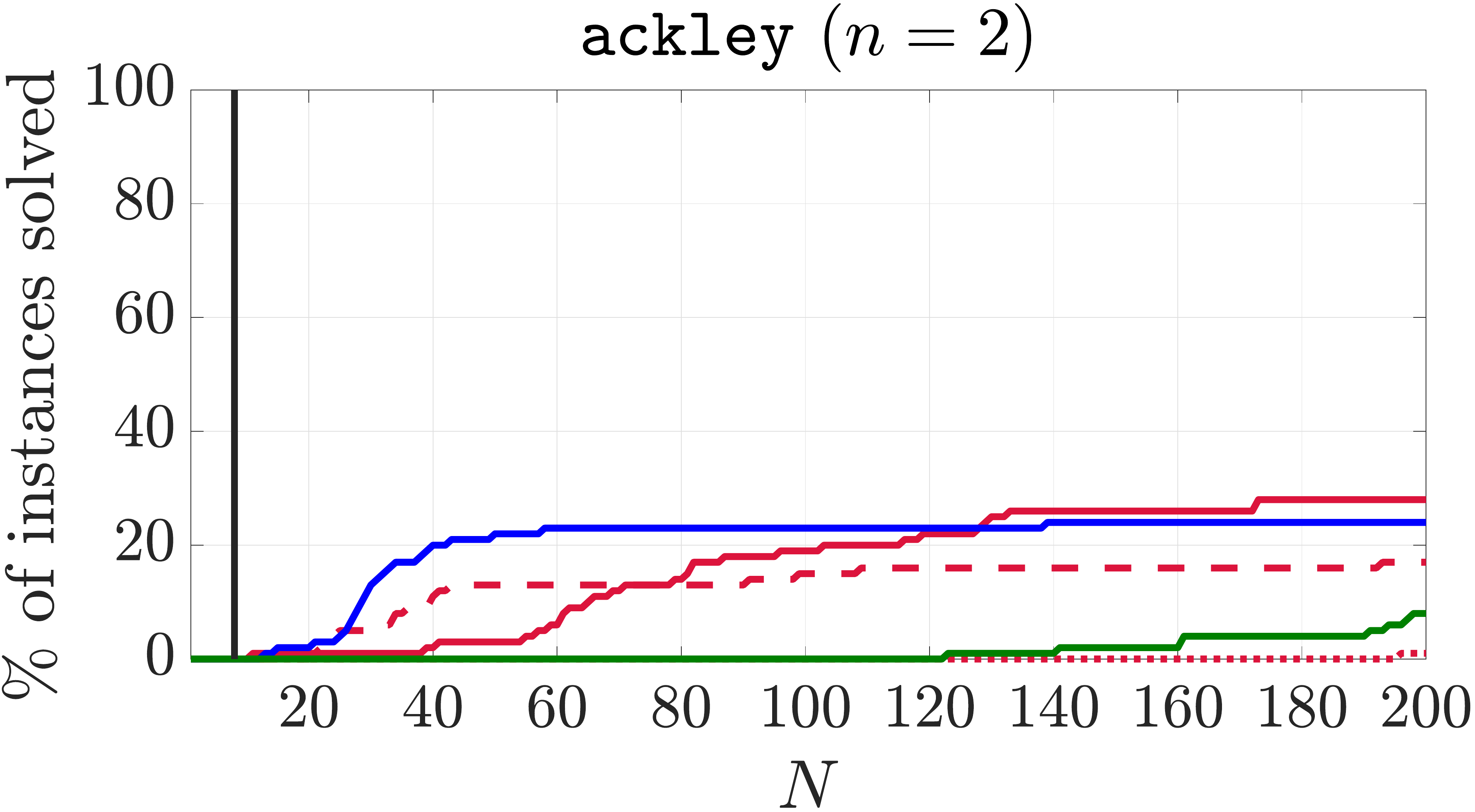}
    }

    \subfloat{
        \centering
        \includegraphics[width=.5\textwidth]{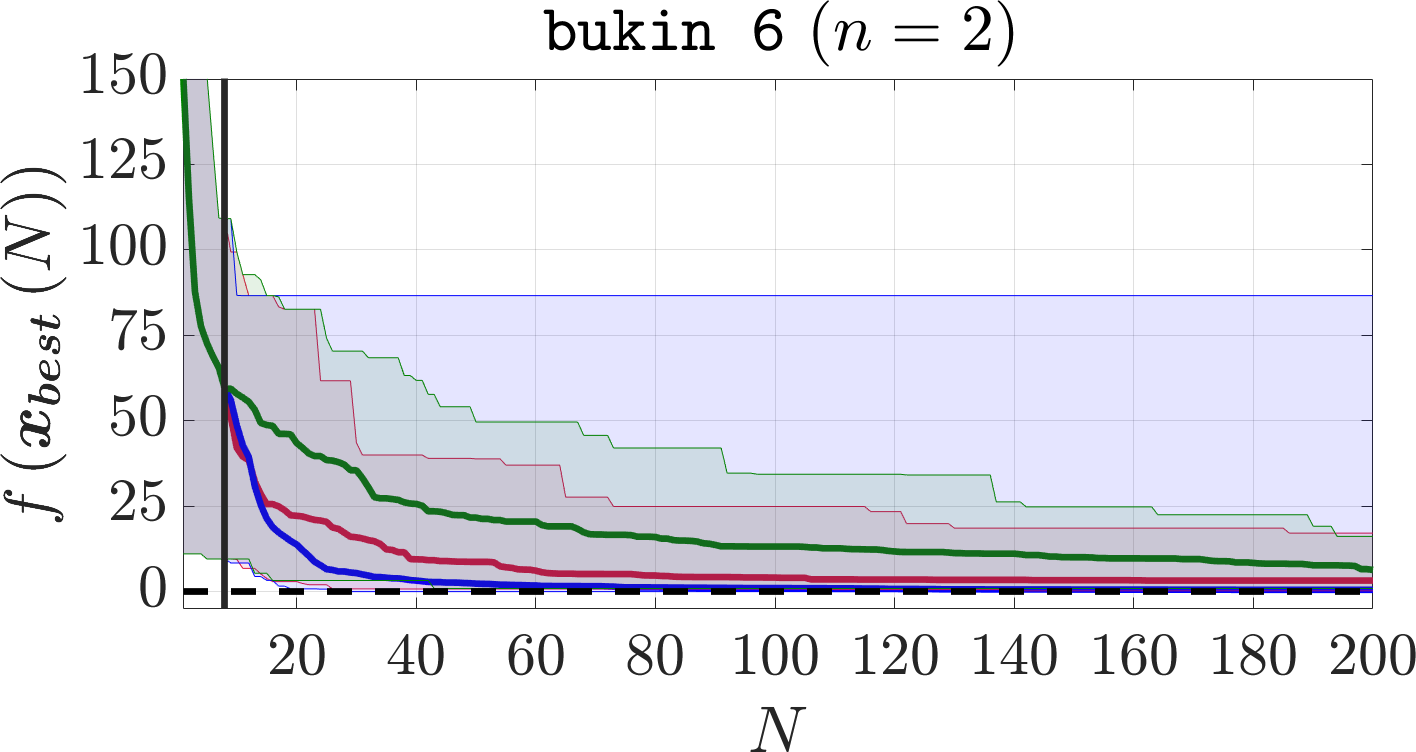}
    }
    \subfloat{
        \centering
        \includegraphics[width=.5\textwidth]{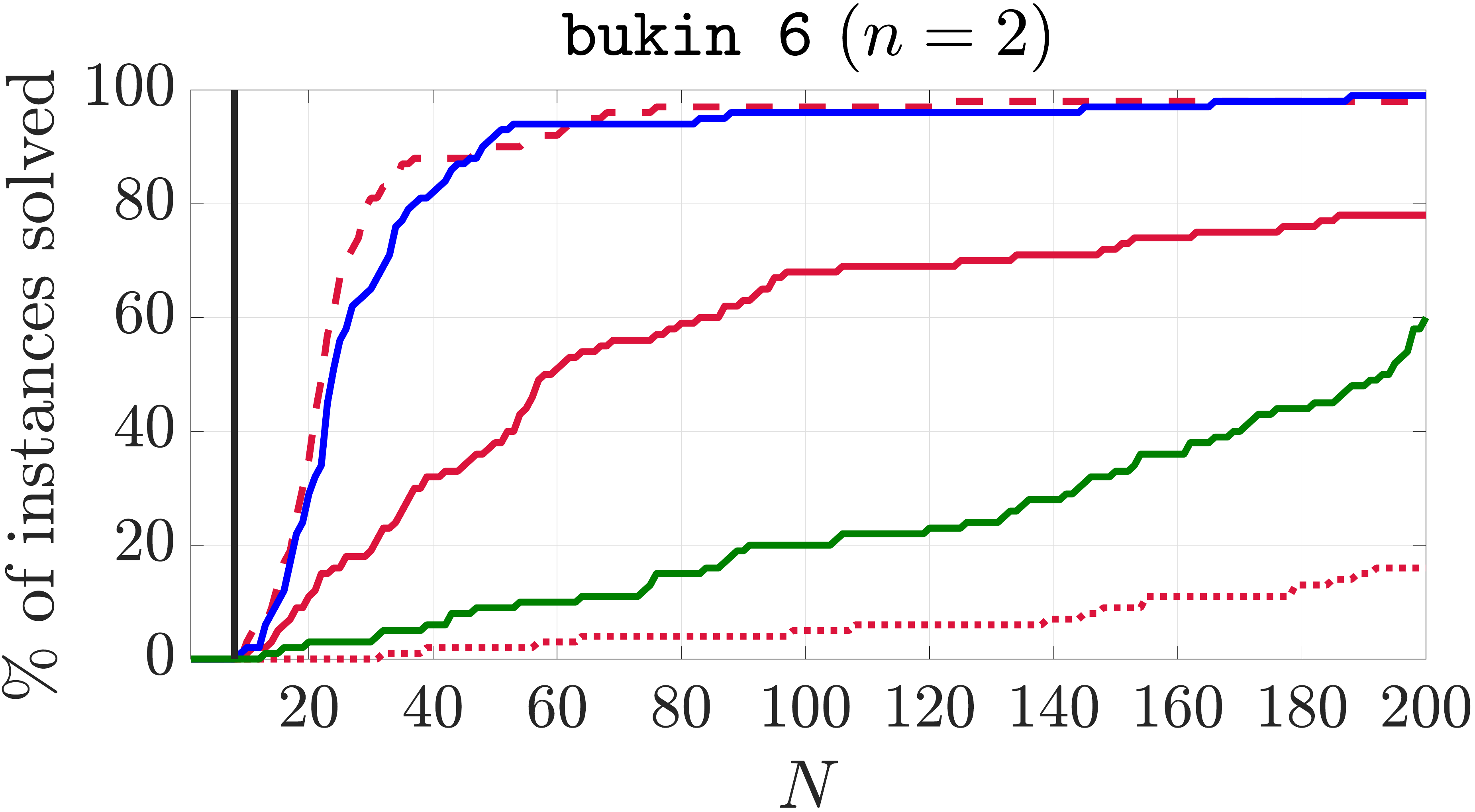}
    }

    \caption{
        \label{fig:performances_on_benchmarks_1}
        Performances achieved by the considered preference-based optimization algorithms on the benchmark optimization problems: convergence plots on the left and data profiles ($acc\left(N\right) > 95\%$) on the right. \GLISprmethod{} (with $\Delta_{cycle} = \langle 0.95, 0.7, 0.35, 0\rangle$) is depicted in red, \GLISpmethod{} in blue and \CGLISpmethod{} in green. The dashed black-line in the convergence plots represents $f^*$. We also show the number of initial samples, $N_{init}$, with a black vertical line. The results obtained by \GLISprmethod{} with $\Delta_{cycle} = \langle 0.95\rangle$ (dashed red line) and \GLISprmethod{} with $\Delta_{cycle} = \langle 0\rangle$ (dotted red line) are shown only in the data profiles.
    }
\end{figure}

\begin{figure}[!htb]
    \centering
    \subfloat{
        \centering
        \includegraphics[width=.5\textwidth]{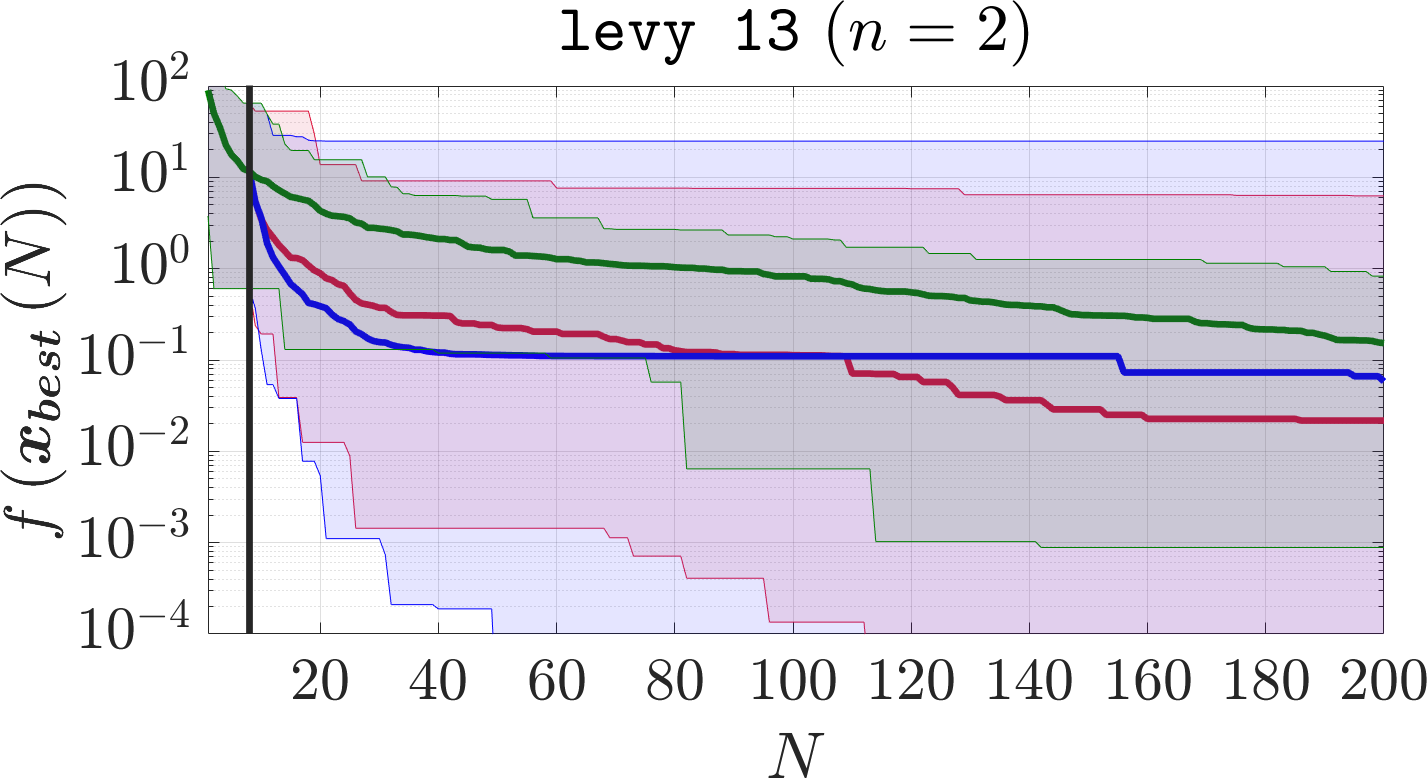}
    }
    \subfloat{
        \centering
        \includegraphics[width=.5\textwidth]{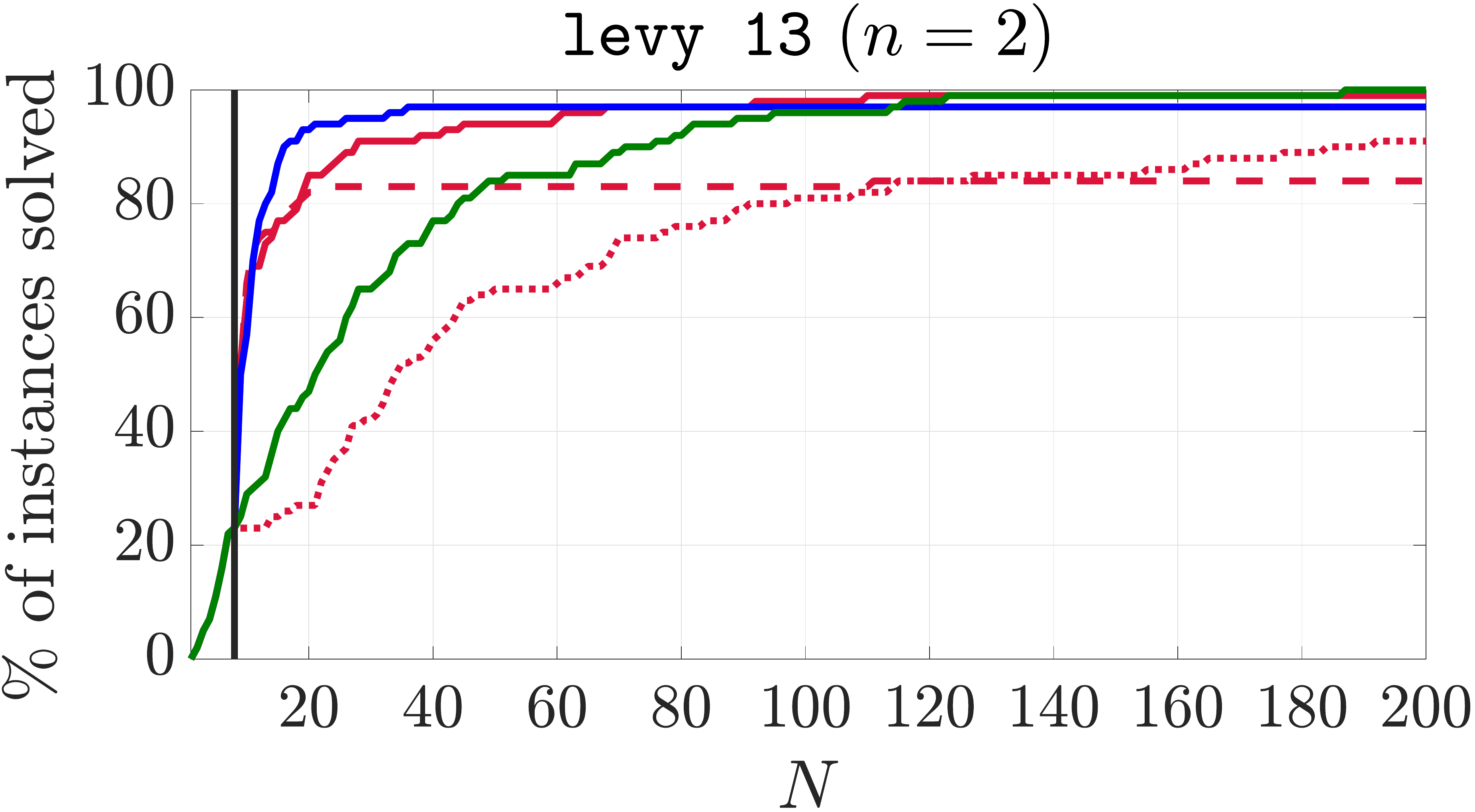}
    }

    \subfloat{
        \centering
        \includegraphics[width=.5\textwidth]{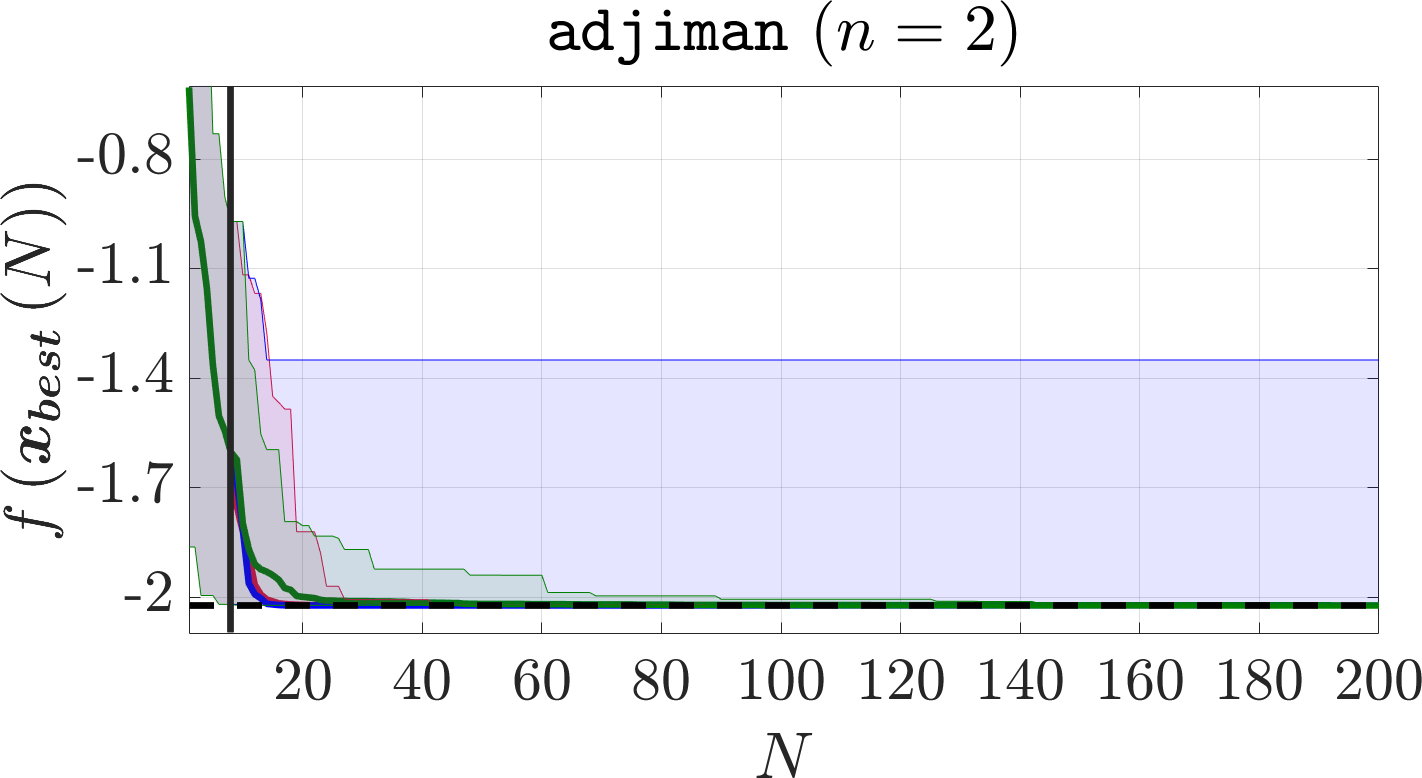}
    }
    \subfloat{
        \centering
        \includegraphics[width=.5\textwidth]{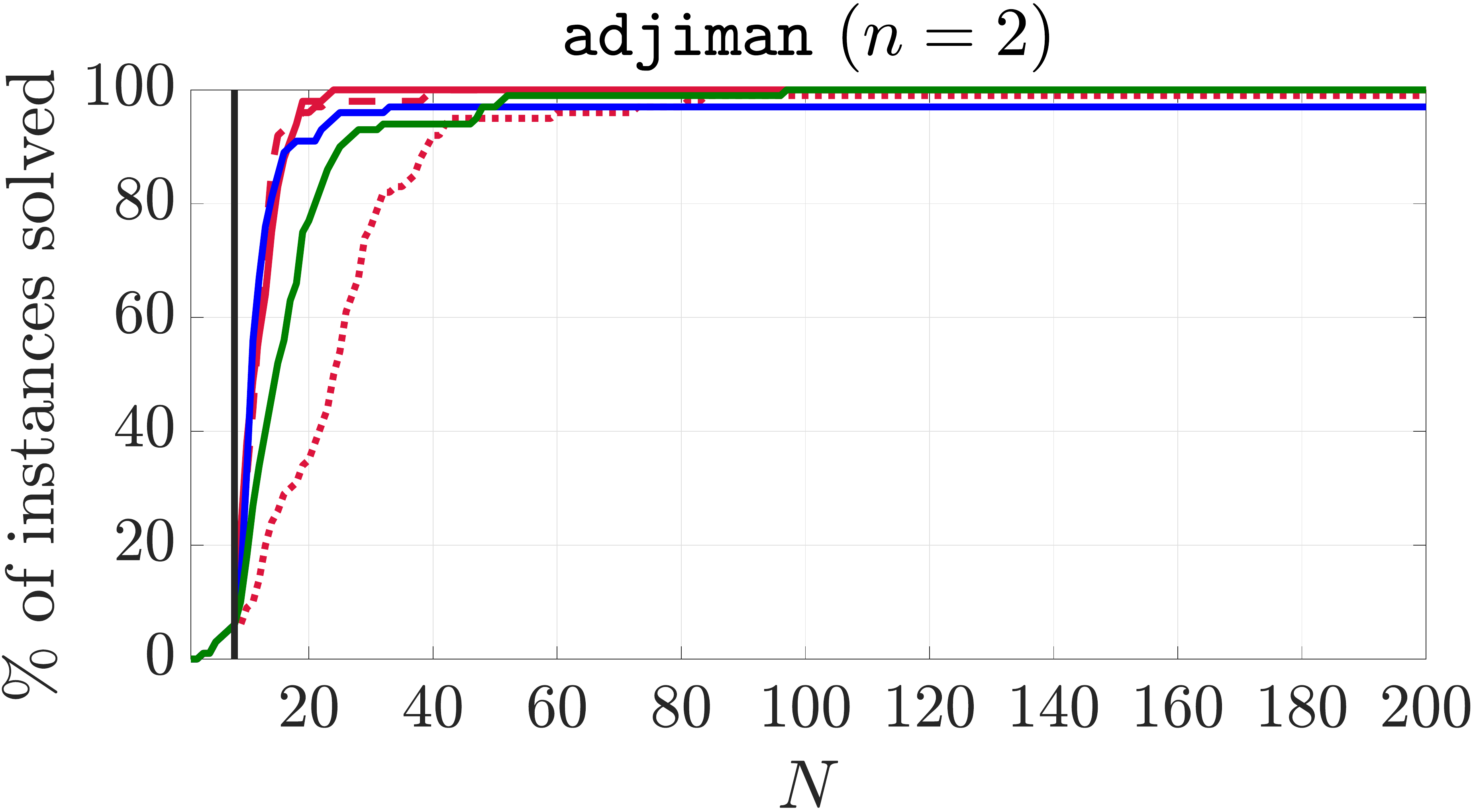}
    }

    \subfloat{
        \centering
        \includegraphics[width=.5\textwidth]{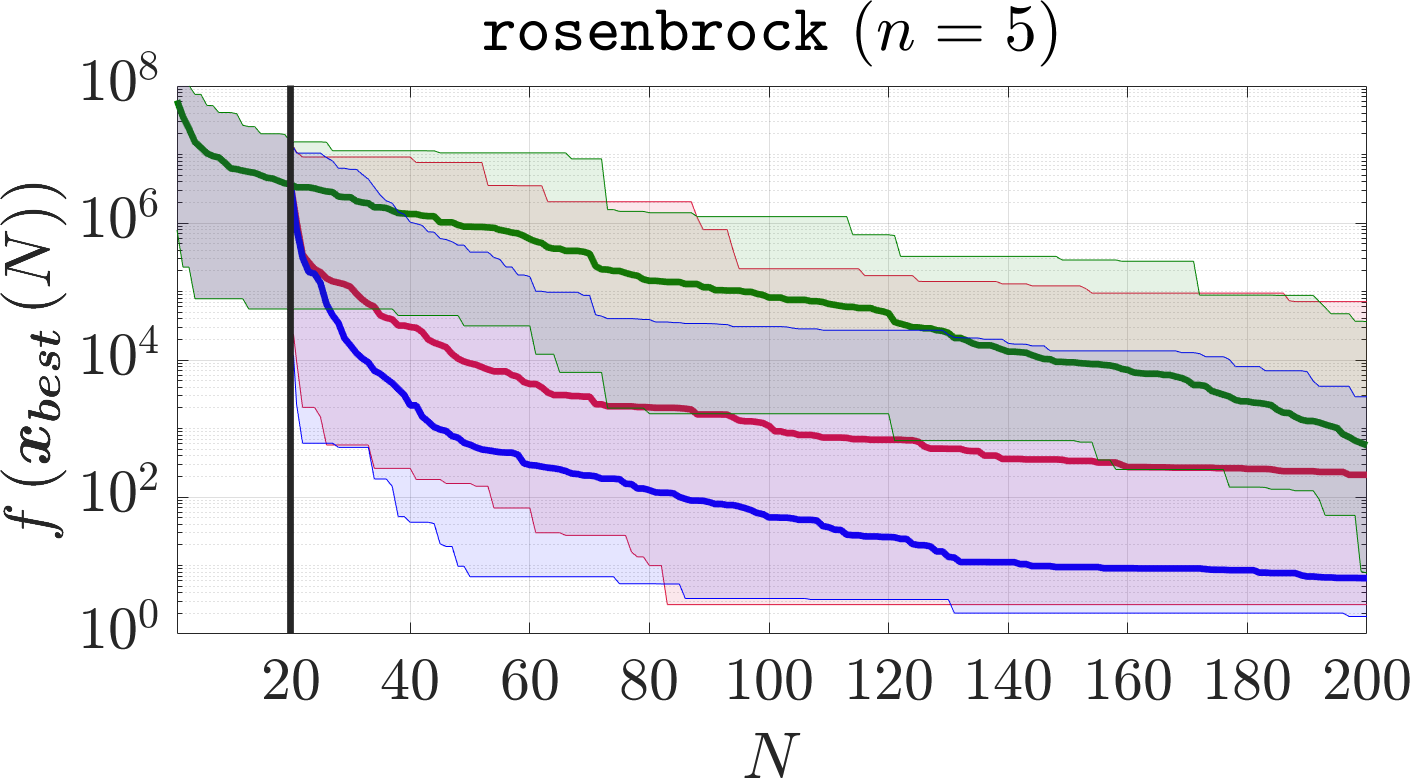}
    }
    \subfloat{
        \centering
        \includegraphics[width=.5\textwidth]{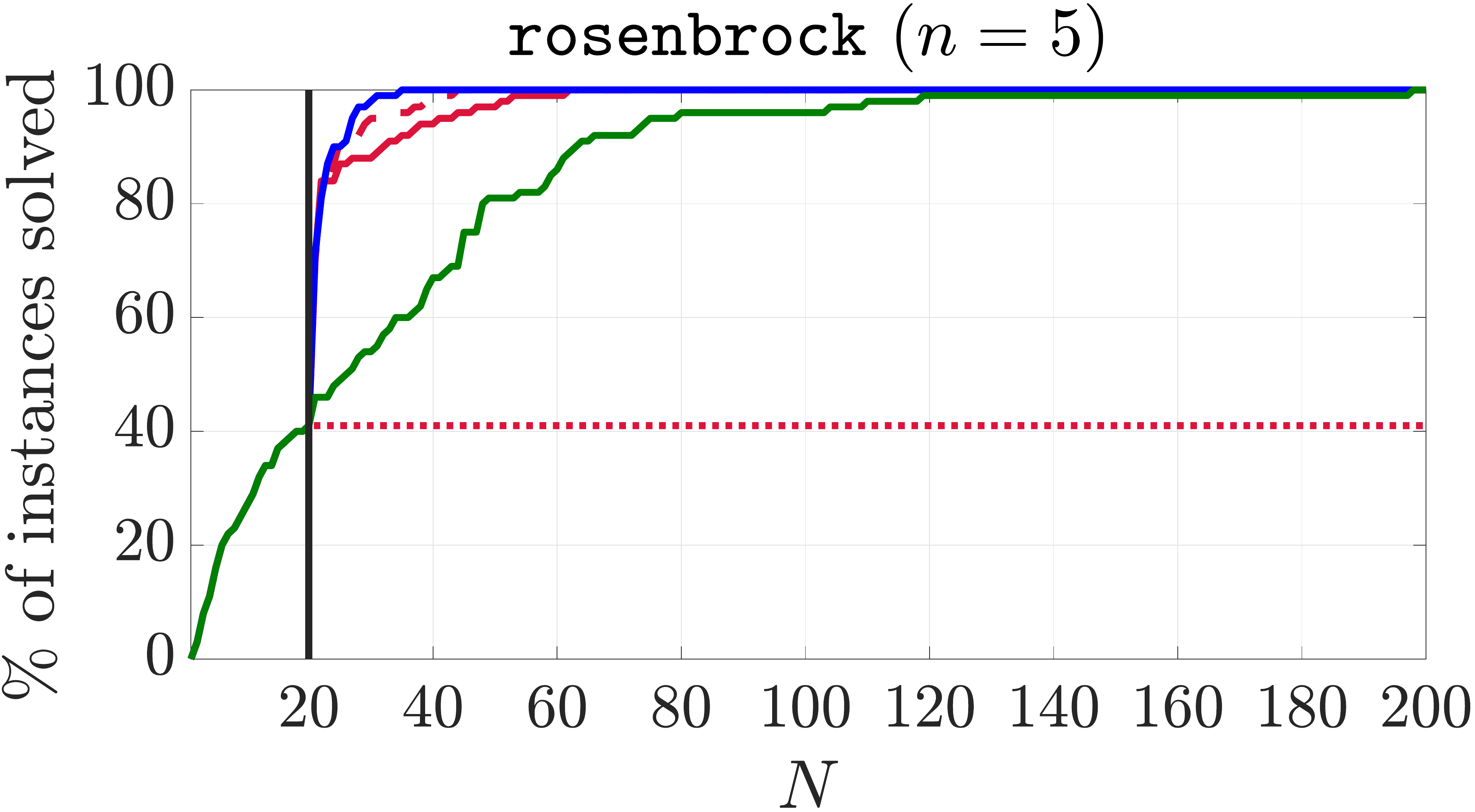}
    }

    \subfloat{
        \centering
        \includegraphics[width=.5\textwidth]{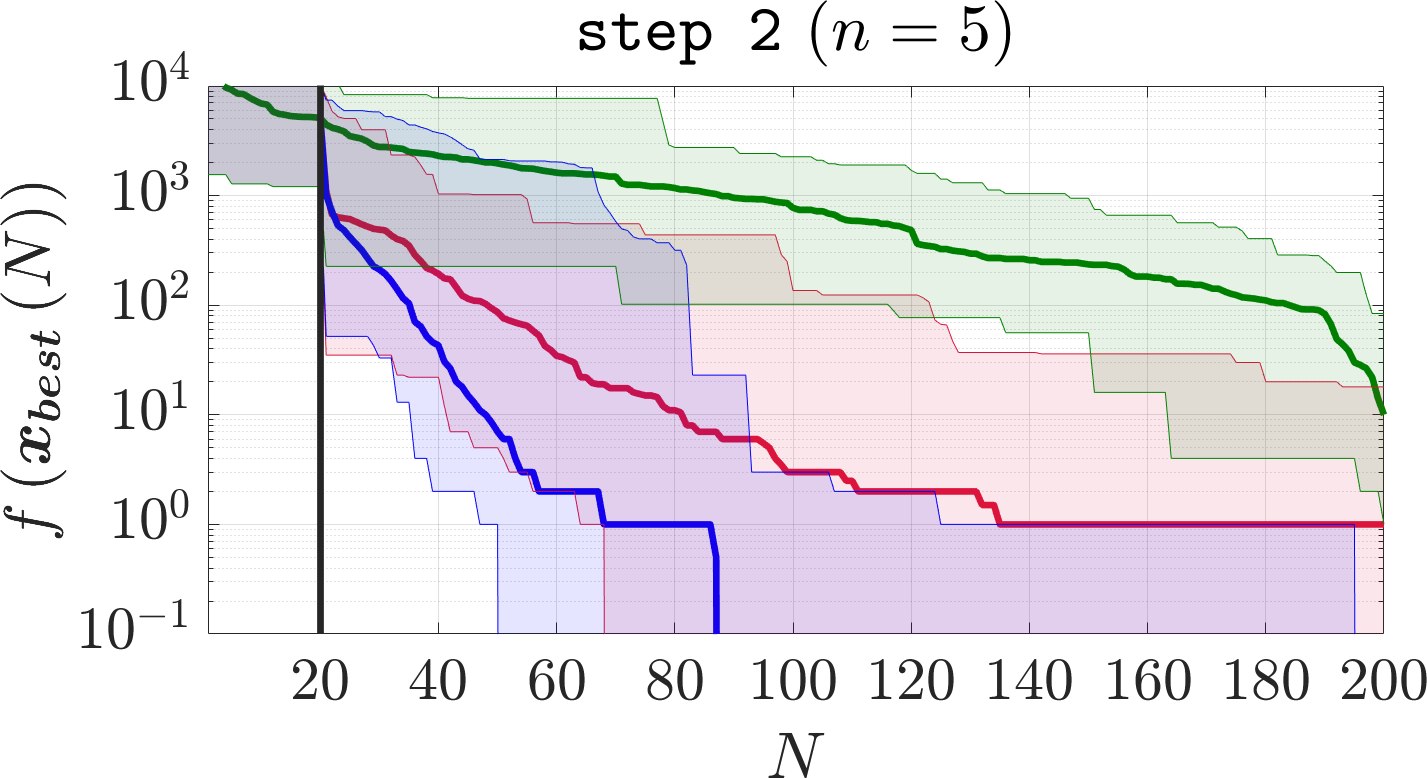}
    }
    \subfloat{
        \centering
        \includegraphics[width=.5\textwidth]{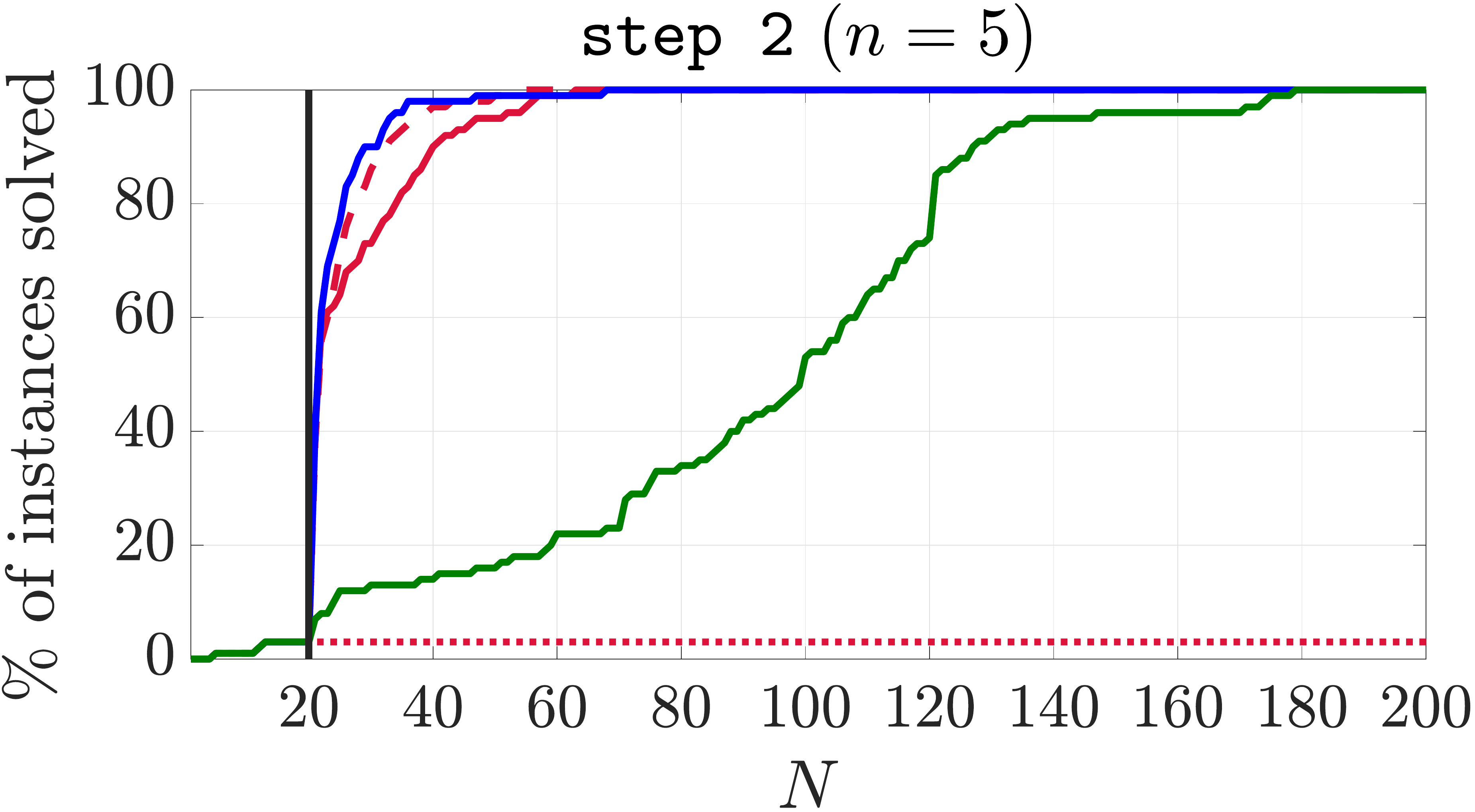}
    }

    \subfloat{
        \centering
        \includegraphics[width=.5\textwidth]{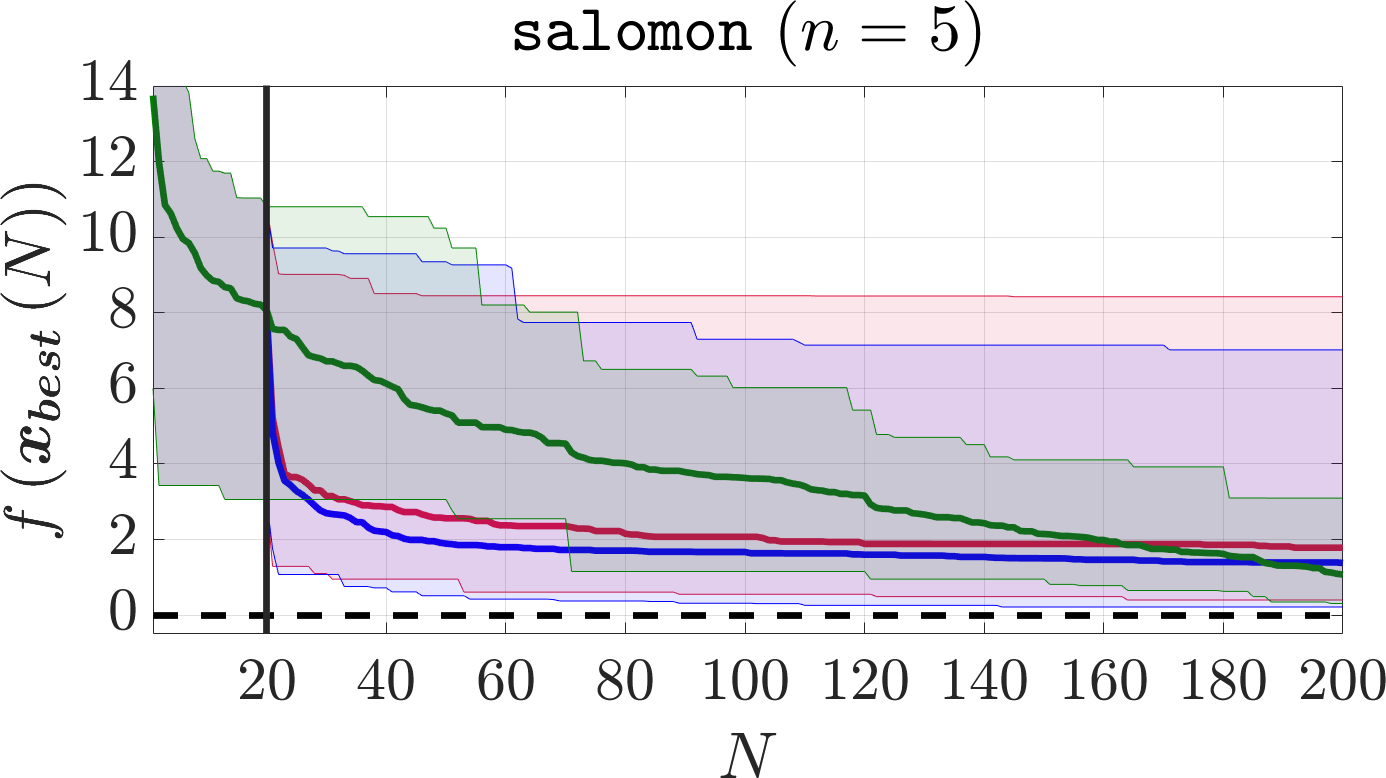}
    }
    \subfloat{
        \centering
        \includegraphics[width=.5\textwidth]{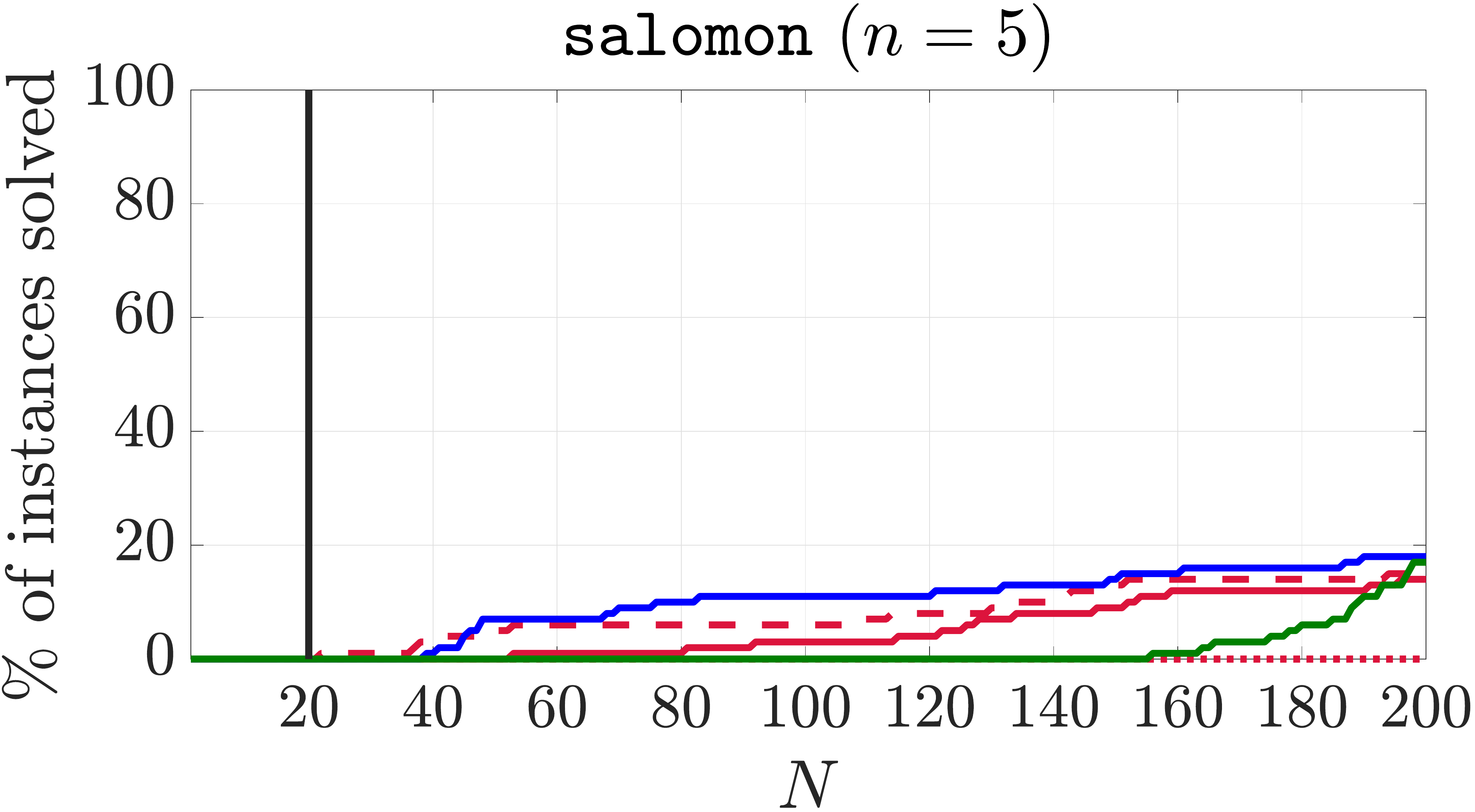}
    }

    \caption{
        \label{fig:performances_on_benchmarks_2}
        \figname{} \ref{fig:performances_on_benchmarks_1} cont'd.
    }
\end{figure}

\begin{sidewaystable}[!htb]
    \settablestretch
    \settablefontsize
    \caption{
        \label{tab:algorithm_comparison}
        Comparison between the considered preference-based optimization algorithms. Several indicators are taken into account: (a) \important{efficiency indicator}, i.e. number of samples required to solve the benchmark optimization problems to an accuracy of $0.95$ (median-wise), as defined in \eqref{eq:number_of_samples_for_relative_accuracy}, (b) \important{robustness indicator}, i.e. percentage of (instances of) problems solved and (c) average execution times (in seconds). The acronym \texttt{n.r.} stands for \quotes{not reached}. The best results are highlighted in bold font. We also report the average performances over all the considered benchmark optimization problems. Note that the averages for indicator (a) do not consider those benchmarks for which the algorithms did not reach the desired accuracy. 
    }
    \centering
    \begin{tabular}{cc|ccc|ccc|ccc|ccc|ccc}
        \multicolumn{8}{c}{}                  & \multicolumn{9}{c}{\GLISprmethod{}}\tabularnewline
        \cline{9-17} \cline{10-17} \cline{11-17} \cline{12-17} \cline{13-17} \cline{14-17} \cline{15-17} \cline{16-17} \cline{17-17}
        \multicolumn{2}{c|}{}                 & \multicolumn{3}{c|}{\GLISpmethod{}}                & \multicolumn{3}{c|}{\CGLISpmethod{}} & \multicolumn{3}{c|}{$\Delta_{cycle}=\langle0.95,0.7,0.35,0\rangle$} & \multicolumn{3}{c|}{$\Delta_{cycle}=\langle0.95\rangle$} & \multicolumn{3}{c}{$\Delta_{cycle}=\langle0\rangle$}\tabularnewline
        \hline
        \textbf{Benchmark}                    & \multirow{1}{*}{$n$}                               & (a)                                  & (b)                                                                 & (c)                                                      & (a)                                                                 & (b)                  & (c)                 & (a)                   & (b)                  & (c)                 & (a)               & (b)                  & (c)     & (a)           & (b)                    & (c)\tabularnewline
        \hline
        \bemporadGOP{}                        & $1$                                                & $\boldsymbol{8}$                     & $70\%$                                                              & $79.8$                                                   & $15$                                                                & $\boldsymbol{100\%}$ & $\boldsymbol{57.5}$ & $11$                  & $\boldsymbol{100\%}$ & $79.4$              & $\boldsymbol{8}$  & $68\%$               & $109.7$ & $18$          & $\boldsymbol{100\%}$   & $71.4$\tabularnewline
        \multirow{1}{*}{\gramacyandleeGOP{}}  & \multirow{1}{*}{$1$}                               & \texttt{n.r.}                        & $31\%$                                                              & $74.4$                                                   & $32$                                                                & $\boldsymbol{100\%}$ & $\boldsymbol{58.1}$ & $\boldsymbol{31}$     & $\boldsymbol{100\%}$ & $67.5$              & \texttt{n.r.}     & $30\%$               & $90.8$  & $36$          & $99\%$                 & $63.8$\tabularnewline
        \multirow{1}{*}{\ackleyGOP{}}         & \multirow{1}{*}{$2$}                               & \texttt{n.r.}                        & $24\%$                                                              & $92.1$                                                   & \texttt{n.r.}                                                       & $8\%$                & $\boldsymbol{70.3}$ & \texttt{n.r.}         & $\boldsymbol{28\%}$  & $127.9$             & \texttt{n.r.}     & $17\%$               & $122.4$ & \texttt{n.r.} & $1\%$                  & $121.1$\tabularnewline
        \multirow{1}{*}{\bukinsixGOP{}}       & \multirow{1}{*}{$2$}                               & $24$                                 & $\boldsymbol{99\%}$                                                 & $84.5$                                                   & $193$                                                               & $60\%$               & $\boldsymbol{73.3}$ & $58$                  & $78\%$               & $115.0$             & $\boldsymbol{23}$ & $98\%$               & $114.6$ & \texttt{n.r.} & $16\%$                 & $123.6$\tabularnewline
        \multirow{1}{*}{\levythirteenGOP{}}   & \multirow{1}{*}{$2$}                               & $\boldsymbol{9}$                     & $97\%$                                                              & $82.0$                                                   & $22$                                                                & $\boldsymbol{100\%}$ & $\boldsymbol{69.3}$ & $\boldsymbol{9}$      & $99\%$               & $126.8$             & $\boldsymbol{9}$  & $84\%$               & $118.2$ & $34$          & $91\%$                 & $120.6$\tabularnewline
        \multirow{1}{*}{\adjimanGOP{}}        & \multirow{1}{*}{$2$}                               & $\boldsymbol{11}$                    & $97\%$                                                              & $105.4$                                                  & $15$                                                                & $\boldsymbol{100\%}$ & $\boldsymbol{71.6}$ & $12$                  & $\boldsymbol{100\%}$ & $146.1$             & $12$              & $\boldsymbol{100\%}$ & $124.0$ & $24$          & $99\%$                 & $117.5$\tabularnewline
        \multirow{1}{*}{\rosenbrockGOP{}}     & \multirow{1}{*}{$5$}                               & $\boldsymbol{21}$                    & $\boldsymbol{100\%}$                                                & $133.3$                                                  & $26$                                                                & $\boldsymbol{100\%}$ & $107.1$             & $\boldsymbol{21}$     & $\boldsymbol{100\%}$ & $190.2$             & $\boldsymbol{21}$ & $\boldsymbol{100\%}$ & $222.3$ & \texttt{n.r.} & $41\%$                 & $\boldsymbol{100.2}$\tabularnewline
        \multirow{1}{*}{\steptwoGOP{}}        & \multirow{1}{*}{$5$}                               & $\boldsymbol{22}$                    & $\boldsymbol{100\%}$                                                & $148.1$                                                  & $100$                                                               & $\boldsymbol{100\%}$ & $104.2$             & $\boldsymbol{22}$     & $\boldsymbol{100\%}$ & $190.6$             & $\boldsymbol{22}$ & $\boldsymbol{100\%}$ & $216.7$ & \texttt{n.r.} & $3\%$                  & $\boldsymbol{98.9}$\tabularnewline
        \salomonGOP{}                         & $5$                                                & \texttt{n.r.}                        & $\boldsymbol{18\%}$                                                 & $137.7$                                                  & \texttt{n.r.}                                                       & $17\%$               & $103.5$             & \texttt{n.r.}         & $14\%$               & $183.4$             & \texttt{n.r.}     & $16\%$               & $219.1$ & \texttt{n.r.} & $0\%$                  & $\boldsymbol{100.8}$\tabularnewline
        \hline
        \multicolumn{2}{c|}{\textbf{Average}} & $\boldsymbol{15.8}$                                & $70.7\%$                             & $104.1$                                                             & $57.6$                                                   & $76.1\%$                                                            & $\boldsymbol{79.4}$  & $23.4$              & $\boldsymbol{79.8\%}$ & $136.3$              & $\boldsymbol{15.8}$ & $68.1\%$          & $148.6$              & $28.0$  & $56.3\%$      & $102.0$\tabularnewline
    \end{tabular}
\end{sidewaystable}
\appendix
\section{Additional proofs}
\label{sec:additional_proofs}
\paragraph*{Proof of Lemma \ref{lemma:Gradient_of_IDW_distance_function}}
The \IDW{} distance function in \eqref{eq:Inverse_distance_weighting_distance_function} is differentiable at any $\boldsymbol{x} \in \mathbb{R}^{n}$ (see Proposition \ref{prop:Differentiability_of_IDW_distance_function}); thus, its gradient $\nabla _{\boldsymbol{x}} z_N\left(\boldsymbol{x}\right)$ is defined everywhere and can be computed by repeatedly applying the chain rule.

Consider the case $\boldsymbol{x} \in \mathbb{R}^n \setminus \mathcal{X}$. First of all, it is easy to prove that the \IDW{} function $w_{i}\left(\boldsymbol{x}\right)$ in \eqref{eq:Inverse_distance_weighting_function} is differentiable at any $\boldsymbol{x} \in \mathbb{R}^{n}\setminus\left\{ \boldsymbol{x}_{i}\right\}$. That is because the squared Euclidean norm $\euclideannorm{\boldsymbol{x}-\boldsymbol{x}_{i}}^2$ is differentiable everywhere and also:
\begin{equation*}
    \euclideannorm{\boldsymbol{x}-\boldsymbol{x}_{i}}^2 \neq 0, \quad \forall \boldsymbol{x} \in \mathbb{R}^{n}\setminus\left\{\boldsymbol{x}_{i}\right\}.
\end{equation*}
Therefore, due to the reciprocal rule, the \IDW{} function is differentiable at any $\boldsymbol{x} \in \mathbb{R}^{n}\setminus\left\{\boldsymbol{x}_{i}\right\}$. In order to compute the gradient of $w_{i}\left(\boldsymbol{x}\right)$ in \eqref{eq:Inverse_distance_weighting_function}, recall that the gradient of the squared Euclidean norm is equal to:
\begin{equation}
    \label{eq:Gradient_of_Euclidean_Norm}
    \nabla_{\boldsymbol{x}}\euclideannorm{\boldsymbol{x}-\boldsymbol{x}_{i}}^2 = 2\cdot\left(\boldsymbol{x}-\boldsymbol{x}_{i}\right), \quad \forall \boldsymbol{x} \in \mathbb{R}^n.
\end{equation}
Using \eqref{eq:Gradient_of_Euclidean_Norm}, it is easy to prove that:
\begin{align}
    \label{eq:Gradient_of_IDW_function}
    \nabla_{\boldsymbol{x}}w_{i}\left(\boldsymbol{x}\right) & =-2\cdot\frac{\boldsymbol{x}-\boldsymbol{x}_{i}}{\euclideannorm{\boldsymbol{x}-\boldsymbol{x}_{i}}^{4}} \nonumber                                                                             \\
                                                            & =-2\cdot\left(\boldsymbol{x}-\boldsymbol{x}_{i}\right)\cdot w_{i}\left(\boldsymbol{x}\right)^{2}, \quad \forall \boldsymbol{x} \in \mathbb{R}^{n}\setminus\left\{ \boldsymbol{x}_{i}\right\}.
\end{align}
Now, let us consider the argument of the $\arctan \left(\cdot\right)$ function in $z_N\left(\boldsymbol{x}\right)$, which is (see \eqref{eq:Inverse_distance_weighting_distance_function}):
\begin{equation*}
    \label{eq:Reciprocal_of_sum_of_IDW_functions}
    h\left(\boldsymbol{x}\right) = \frac{1}{\sum_{i=1}^{N}w_{i}\left(\boldsymbol{x}\right)}, \quad \forall \boldsymbol{x} \in \mathbb{R}^n \setminus \mathcal{X}.
\end{equation*}
We can find the gradient of $h\left(\boldsymbol{x}\right)$ by applying the chain rule in combination with \eqref{eq:Gradient_of_IDW_function}:
\begin{equation}
    \label{eq:Gradient_of_reciprocal_of_sum_of_IDW_functions}
    \nabla_{\boldsymbol{x}}h\left(\boldsymbol{x}\right) =2\cdot\frac{\sum_{i=1}^{N}\left(\boldsymbol{x}-\boldsymbol{x}_{i}\right)\cdot w_{i}\left(\boldsymbol{x}\right)^{2}}{\left[\sum_{i=1}^{N}w_{i}\left(\boldsymbol{x}\right)\right]^{2}},
    \quad \forall \boldsymbol{x} \in \mathbb{R}^{n} \setminus \mathcal{X}.
\end{equation}
Finally, using \eqref{eq:Gradient_of_reciprocal_of_sum_of_IDW_functions} and applying the chain rule one last time, we can compute the gradient of the \IDW{} distance function in \eqref{eq:Inverse_distance_weighting_distance_function}:
\begin{align}
    \label{eq:Inverse_distance_weighting_distance_function_gradient_partial_1}
    \nabla_{\boldsymbol{x}}z_N\left(\boldsymbol{x}\right) & =\frac{d}{dt} \left[-\frac{2}{\pi}\cdot \arctan\left(t\right)\right]\evaluatedat{t=h\left(\boldsymbol{x}\right)}\cdot\nabla_{\boldsymbol{x}}h\left(\boldsymbol{x}\right) \nonumber                                                                                                            \\
                                                          & =-\frac{4}{\pi}\cdot\frac{1}{1+\left[\sum_{i=1}^{N}w_{i}\left(\boldsymbol{x}\right)\right]^{-2}}\cdot\frac{\sum_{i=1}^{N}\left(\boldsymbol{x}-\boldsymbol{x}_{i}\right)\cdot w_{i}\left(\boldsymbol{x}\right)^{2}}{\left[\sum_{i=1}^{N}w_{i}\left(\boldsymbol{x}\right)\right]^{2}} \nonumber \\                                        & =-\frac{4}{\pi}\cdot\frac{\sum_{i=1}^{N}\left(\boldsymbol{x}-\boldsymbol{x}_{i}\right)\cdot w_{i}\left(\boldsymbol{x}\right)^{2}}{1+\left[\sum_{i=1}^{N}w_{i}\left(\boldsymbol{x}\right)\right]^{2}}, \quad
    \forall \boldsymbol{x} \in \mathbb{R}^{n} \setminus \mathcal{X}.
\end{align}
Now, let us consider the case $\boldsymbol{x}_i \in \mathcal{X}$. In \cite{Bemporad2020}, the authors have proven that all the partial derivatives of $z_N\left(\boldsymbol{x}\right)$ in \eqref{eq:Inverse_distance_weighting_distance_function} are zero at each $\boldsymbol{x}_i \in \mathcal{X}$, i.e.
\begin{equation}
    \label{eq:Inverse_distance_weighting_distance_function_gradient_partial_2}
    \nabla_{\boldsymbol{x}}z_N\left(\boldsymbol{x}_i\right) = \boldsymbol{0}_{n}, \quad \forall \boldsymbol{x}_i \in \mathcal{X}.
\end{equation}
Lastly, combining \eqref{eq:Inverse_distance_weighting_distance_function_gradient_partial_1} and \eqref{eq:Inverse_distance_weighting_distance_function_gradient_partial_2}, we obtain the expression for the gradient of the \IDW{} distance function $\forall \boldsymbol{x} \in \mathbb{R}^{n}$, as reported in \eqref{eq:Inverse_distance_weighting_distance_function_gradient}.
\qedwhite{}


\end{document}